\documentclass[12pt]{amsart}
\usepackage{graphicx}
\usepackage{amsmath,amssymb}
\usepackage{color}

\newtheorem{thm}{Theorem}[section]
\newtheorem{cor}[thm]{Corollary}
\newtheorem{lem}[thm]{Lemma}
\newtheorem{prop}[thm]{Proposition}

\theoremstyle{definition}

\theoremstyle{remark}
\newtheorem{rem}[thm]{Remark}
\theoremstyle{example}

\numberwithin{equation}{section}

\newcommand{\del}{\delta}

\newcommand{\B}{{\mathbb B}}
\newcommand{\D}{{\mathbb D}}
\newcommand{\R}{{\mathbb R}}
\newcommand{\SSS}{{\mathbb S}}
\newcommand{\C}{{\mathbb C}}

\newcommand{\N}{{\mathbb N}}
\newcommand{\Z}{{\mathbb Z}}

\newcommand{\re}{{\rm Re}\,}

\newcommand{\calB}{{\mathcal B}}
\newcommand{\calC}{{\mathcal C}}

\newcommand{\calM}{{\mathcal M }}
\newcommand{\calN}{{\mathcal N}}

\begin{document}

\title{$\calN(p, q, s)$-type spaces in the unit ball of $\C^n$}%

\date{\today}%

\author{Bingyang Hu and Songxiao Li$^\dag$}%

\address{Bingyang Hu: Department of Mathematics, University of Wisconsin, Madison, WI 53706-1388, USA.}%
\email{bhu32@wisc.edu}

\address{Songxiao Li:  Institute of Fundamental and Frontier Sciences, University of Electronic Science and Technology of China, 610054, Chengdu, Sichuan, P.R. China.}
\address{Institute System Engineering, Macau University of Science and Technology,  Avenida Wai Long,
Taipa, Macau. }%
\email{jyulsx@163.com}


\thanks{$\dag$ Corresponding author.}

\maketitle

\begin{abstract}
 In this paper, we consider a new class of space, called $\calN(p, q, s)$-type spaces, in the unit ball $\B$ of $\C^n$.  We study some basic properties, Hadamard gaps, Hadamard products, Random power series, Korenblum's inequality, Gleason's problem, atomic decomposition of  $\calN(p, q, s)$-type spaces. Moreover, we also establish several equivalent characterizations, including Carleson measure characterization and various derivative characterizations. Finally, we also characterize the distance
  between Bergman-type spaces and $\calN(p, q, s)$-type spaces, Riemann-Stieltjes operators and multipliers on $\calN(p, q, s)$-type spaces.
\end{abstract}

\tableofcontents


\section{Introduction}


Let $\B$ be the open unit ball in $\C^n$ with $\SSS$ as its boundary and $H(\B)$
the collection of all holomorphic functions in $\B$.
$H^\infty$ denotes the Banach space consisting of all bounded
holomorphic functions in $\B$ with the norm $\|f\|_{\infty}=\sup_{z
\in \B} |f(z)|$. For $l>0$, the  Bergman-type space $A^{-l}(\B)$ is the space
of all   $f\in H(\B)$ such that
$$|f|_l=\sup_{z \in \B} |f(z)|(1-|z|^2)^l<\infty.$$
Let $A^{-l}_0(\B)$ denote the closed subspace of  $A^{-l}(\B)$ such
that  $\lim\limits_{|z|\rightarrow 1} |f(z)|(1-|z|^2)^l=0.$

Denote $dV$ the normalized volume measure over $\B$ and for $\alpha>-1$, the weighted Lebesgue measure $dV_\alpha$ is defined by
$$
dV_\alpha(z)=c_\alpha (1-|z|^2)^\alpha dV(z)
$$
where $c_\alpha=\frac{\Gamma(n+\alpha+1)}{n! \Gamma(\alpha+1)}$ is a normalizing constant so that $dV_\alpha$ is a probability measure on $\B$.

For $\alpha>-1$ and $p>0$, the weighted Bergman space $A^p_\alpha$ consists of all $f\in H(\B)$ satisfying
$$
\|f\|_{p, \alpha}=\left[\int_\B |f(z)|^pdV_\alpha(z)\right]^{1/p}<\infty.
$$
It is well-known that when $1 \le p<\infty$, $A^p_\alpha$ is a Banach space and when $0<p<1$, $A^p_\alpha$ becomes a complete metric space. We refer the reader to the books \cite{HKZ, ZZ, Zhu} for more  information.

Let $\Phi_a(z)$  be the automorphism of $\B$ for $a \in \B$, i.e.,
$$ \Phi_a(z)=\frac{a-P_az-s_aQ_a z}{1-\langle z,a\rangle}, $$
where $s_a=  \sqrt{1-|a|^2} $, $P_a$ is the orthogonal projection
into the space spanned by $a$ and $Q_a=I-P_a$ (see, e.g.,
\cite{Zhu}). For $p>0$,  the space $\calN_p$ on $\B$ was studied in
\cite{HK1}, i.e.,
\[ \begin{split}
&\calN_p=\calN_p(\B)\\
&=\left\{f \in H(\B):  \sup_{a \in \B} \left(\int_{\B}
|f(z)|^2(1-|\Phi_a(z)|^2)^pdV(z)\right)^{1/2}<\infty\right\}.\end{split}
\]   The little space of $\calN_p$-space, denoted by
$\calN_p^0$, which consists of all $f \in \calN_p$ such that
$$
  \lim_{|a| \to 1} \int_{\B}
|f(z)|^2(1-|\Phi_a(z)|^2)^pdV(z) =0 .
$$

Let $d\lambda(z)=\frac{dV(z)}{(1-|z|^2)^{n+1}}$. Then $d\lambda$ is M\"obius invariant (see, e.g., \cite{Rud}), which means,
$$
\int_\B f(z)d\lambda(z)=\int_\B f\circ \phi(z)d\lambda(z)
$$
for each $f \in L^1(\lambda)$  and $\phi$ an automorphism of $\B$.

For $p >0$ and $q, s>0$, in this paper we consider the $\calN(p, q, s)$-type space as follows:
$$ \calN(p, q, s):= \left\{ f\in H(\B): \|f\| <\infty \right\},$$
where
$$
\|f\|^p=\sup_{a \in \B} \int_\B |f(z)|^p(1-|z|^2)^q
(1-|\Phi_a(z)|^2)^{ns} d\lambda(z).
$$
The corresponding little space, denoted by $\calN^0(p, q, s) $, is the space of all $f\in \calN(p, q, s)$ such that
\[ \lim_{|a| \to 1} \int_\B
|f(z)|^p(1-|z|^2)^q (1-|\Phi_a(z)|^2)^{ns} d\lambda(z)=0. \]

It should be noted that when $p=2, q=n+1, s>0$, $\calN(2, n+1, s)$ coincides the $\calN_{ns}$-space, as well as the little space, in particular, by \cite[Theorem 2.1]{HK1}, when $p=2, q=n+1, s>1$,
$$
\calN(2, n+1, s)=A^{-\frac{n+1}{2}}(\B).
$$
Moreover, as we will show later (see, e.g., Remark \ref{NpqsFpqsID}), the $\calN(p, q, s)$-spaces coincide with $F(p, q, t)$-spaces by the identification
$$
\calN(p, q, s)=F(p, p+q-n-1, ns).
$$
Moreover,
$$
\calN^0(p, q, s)=F_0(p, p+q-n-1, ns).
$$
Let $0<p<\infty, 0 \le t<\infty, -1<q+t<\infty, -1-n<q<\infty$. Recall that an $f \in H(\B)$ is said to belong to the $F(p, q, t)$-space if
$$
|f(0)|^p+\sup_{a \in\B} \int_\B |\nabla f(z)|^p(1-|z|^2)^q g^t(z, a)dV(z)<\infty,
$$
and  an $f \in H(\B)$ is said to belong to the $F_0(p, q, t)$-space  if
$$
\lim_{|a|\rightarrow1}\int_\B |\nabla f(z)|^p(1-|z|^2)^q g^t(z, a)dV(z)=0,
$$
where  
 $$\nabla f(z)=\left( \frac{\partial f}{\partial z_1}(z), \dots, \frac{\partial f}{\partial z_n}(z) \right) ~~\mbox{and}~~ g(z, a)=\log\frac{1}{|\Phi_a(z)|}.$$
The $F(p, q, t)$-spaces were first introduced by R. Zhao on the unit disk in $\C$ in \cite{RZ}, and later studied on the unit disk and the unit ball $\B$ of $\C^n$ by various authors (see, e.g, \cite{SL, PR, ZHC, XZ} and the reference therein). However, to the best of our knowledge, currently there are few results in the theory of $F(p, q, t)$-spaces in the unit ball due to the complexity of the parameters $p, q$ and $t$, as well as the high dimension.

As is well-known from the literature, the weighted Bergman space $A_\alpha^p$ has a lot of nice properties. Moreover, many concrete operators (including composition operators, Toeplitz operators, Hankel operators, Riemann-Stieltjes operators and etc.) on  weighted Bergman spaces have been completely characterized.  However, the Dirichlet type space $D^p_\alpha$ does not.  Here, we say an $f \in H(\B)$ belongs to the   Dirichlet type space, denoted by $D_\alpha^p$, if
$$
\int_\B |\nabla f(z)|^p(1-|z|^2)^\alpha dV(z)<\infty.
$$
Here $0<p<\infty, \alpha>-1$. Note that when $\alpha=p-n-1, p>n$, the Dirichlet type space becomes the classical Besov space $B_p$. (see, e.g., \cite{SL2, Zhu} for more information on these spaces). Inspired by this fact, it is natural for us to consider the $\calN(p, q, s)$-spaces.

Another important feature of $\calN(p, q, s)$-spaces is that for any $\frac{n-1}{n}<s<\frac{n}{n-1}, n \ge 2$,
$$
Q_s \subseteq \calN(2, 1, s) .
$$
 For detailed studies of $Q_s$ spaces on the unit disk and the unit ball, we refer the readers to the monographs \cite{OYZ2, JX1, JX2}. The above claim follows from a comparison of these two classes of function spaces. More precisely, for $Q_s$ space, we have
\begin{enumerate}
\item[(a).] When $1<s<\frac{n}{n-1}$, $Q_s=\calB$ (the Bloch space);
\item[(b).] When $\frac{n-1}{n}<s \le 1$, $f \in Q_s$ if and only if
$$
\sup_{a \in \B} \int_\B |\widetilde{\nabla} f(z)|^2(1-|\Phi_a(z)|^2)^{ns}d\lambda(z)<\infty,
$$
where $(\widetilde{\nabla}f)(z)=\nabla( f \circ \Phi_z)(0)$ denotes the invariant gradient of $\B$ (see, e.g., \cite{LH}). In fact,  $Q_s=F(2, 2-(n+1), ns)$.
\end{enumerate}
On the other hand, for $\calN(2, 1, s)$-space, we have
\begin{enumerate}
\item[(a').] When $s>1$, $\calN(2, 1, s)=A^{-\frac{1}{2}}(\B)=\calB^{3/2}$ (the $3/2$-Bloch space);
\item[(b').] When $\frac{n-1}{n}<s \le 1$, $f \in \calN(2, 1, s)$ if and only if
$$
\sup_{a \in \B} \int_\B |\widetilde{\nabla} f(z)|^2(1-|z|^2)(1-|\Phi_a(z)|^2)^{ns}d\lambda(z)<\infty.
$$
\end{enumerate}
The assertion (a') follows from Proposition \ref{boundaryineq}, Proposition \ref{Npqsbigs} and Lemma \ref{BlochBertype}, while assertion (b') follows from Corollary \ref{equivalentnorm}. In fact, $\calN(2, 1, s)=F(2, 3-(n+1), ns).$  From these facts, it is easy to see that the desired inclusion holds. 

Moreover, if we allow $q=0$ in the definition of the $\calN(p, q, s)$-spaces, that is the endpoint case, then from the proofs of  Propositions \ref{boundaryineq} and \ref{Npqsbigs}, one can easily check that $\calN(p, 0, s) \subseteq H^\infty, p \ge 1, s>0$, while the equality holds when $s>1$ (see \cite{RZ2} for the case of the unit disk). Combining this observation with the fact that $Q_s=\calB$ when $1<s<\frac{n}{n-1}$ and $Q_s=\{ \textrm{constant functions} \}$ when $s \ge \frac{n}{n-1}$, we can also see that the $\calN(p, q, s)$-spaces are independent of the $Q_s$ spaces and of their own interest.  

The aim of this paper is to  systematically  study $\calN(p, q, s)$-spaces in the unit ball,  and we believe the methods developed in this paper, as well as some new features of $\calN(p, q, s)$-type spaces and $\calN^0(p, q, s)$-type spaces, can be generalized to  $F(p, q, s)$-spaces and $F_0(p, q, s)$-spaces.

 In this paper, we study various properties of $\calN(p, q, s)$-type spaces, including some basic properties, Hadamard gaps, Hadamard products, Gleason's problem, Random power series, Korenblum's inequality, atomic decomposition of   $\calN(p, q, s)$-type spaces. We also establish several equivalent characterizations, including Carleson measure characterization and various derivative characterizations. Finally, we   investigate the distance between Bergman-type spaces and $\calN(p, q, s)$-type spaces, Riemann-Stieltjes operators and multipliers on $\calN(p, q, s)$-type spaces.

Throughout this paper, for $a, b \in \R$, $a \lesssim b$ ($a \gtrsim b$, respectively) means there exists a positive number $C$, which is independent of $a$ and $b$, such that $a \leq Cb$ ($ a \geq Cb$, respectively). Moreover, if both $a \lesssim b$ and $a \gtrsim b$
hold, then we say $a \simeq b$.


\section{Basic properties of $\calN(p, q, s)$-type spaces}


\subsection{Basic structure}

 We first show that $\calN(p, q, s)$ is a functional Banach space (see, e.g., \cite{cowen}) when $p \ge 1$ and $q, s>0$.  For $a \in \B$ and $0<R<1$, define
$$
D(a, R)=\Phi_a\left( \left\{ z \in\B: |z|<R\right\}\right)=\left\{z \in \B: |\Phi_a(z)|<R\right\}.
$$
The following result plays an important role in the sequel.

\begin{prop} \label{boundaryineq}
Let $p \ge 1$ and $q, s>0$. The point evaluation $K_z: f \mapsto f(z)$ is a continuous linear functional on $\calN(p, q, s)$. Moreover, $\calN(p, q, s) \subseteq A^{-\frac{q}{p}}(\B)$.
\end{prop}

\begin{proof}
Denote $\B_{1/2}:=\{z: |z|<\frac{1}{2}\}$. For each $f \in \calN(p, q, s)$ and $a_0 \in \B$, we have
\begin{eqnarray*}
\|f\|^p%
&=& \sup_{a \in \B} \int_\B |f(z)|^p(1-|z|^2)^q (1-|\Phi_a(z)|^2)^{ns} d\lambda(z)\\
&\ge& \int_{D(a_0, 1/2)} |f(z)|^p(1-|z|^2)^q (1-|\Phi_{a_0}(z)|^2)^{ns} d\lambda(z)\\
&\ge& C_{n, s}  \int_{D(a_0, 1/2)} |f(z)|^p(1-|z|^2)^q  d\lambda(z)\\
&=& C_{n, s}  \int_{\B_{1/2}} |f(\Phi_{a_0}(w))|^p(1-|\Phi_{a_0}(w)|^2)^q  d\lambda(w)\\
&& (\textrm{change variable}\ z=\Phi_{a_0}(w))\\
&=&  C_{n, s}  \int_{\B_{1/2}} |f(\Phi_{a_0}(w))|^p \frac{(1-|a_0|^2)^q(1-|w|^2)^{q-n-1}}{\left |1- \langle a_0, w \rangle \right|^{2q}}  dV(w)\\
&\ge& C_{q, s, n} (1-|a_0|^2)^q \int_{\B_{1/2}}|f(\Phi_{a_0}(w))|^pdV(w) \\
&\ge& C'_{q, s, n} (1-|a_0|^2)^q |f(a_0)|^p \quad  (|f \circ \Phi_{a_0}(\cdot)|^p \ \textrm{is subharmonic}),
\end{eqnarray*}
where $C_{n, s}$ is a constant depending on $n$ and $s$ and $C_{q, s, n}$ and $C'_{q, s, n}$ are some constants depending on $q, s$ and $n$. Hence, for any $z \in \B$, we have
$$
|f(z)| \lesssim \frac{\|f\|}{(1-|z|^2)^{\frac{q}{p}}},
$$
 which implies the point evaluation is a continuous linear functional, as well as, $\calN(p, q, s) \subseteq A^{-\frac{q}{p}}(\B)$.
\end{proof}

Letting $q=n+1$ and $p=2$ in Proposition \ref{boundaryineq}, we get \cite[Theorem 2.1, (a)]{HK1} as a particular case.

\begin{cor}
For $p>0$,  $\calN_p(\B) \subseteq A^{-\frac{n+1}{2}}(\B)$.
\end{cor}

\begin{thm} \label{functionalBS}
Let $p \ge 1$ and $q, s>0$. $\calN(p, q, s)$ is a functional Banach space.
\end{thm}

\begin{proof}
It is clear that $\calN(p, q, s)$ is a normed vector space with respect to the norm $\| \cdot \|$. It suffices to show the completeness of $\calN(p, q, s)$. Let $\{f_m\}$ be a Cauchy sequence in $\calN(p, q, s)$. From this, by Proposition \ref{boundaryineq}, it follows that $\{f_m\}$ is a Cauchy sequence in the space $H(\B)$, and hence it converges to some $f \in H(\B)$. It remains to show that $f \in \calN(p, q, s)$. Indeed, there exists a $\ell_0 \in \N$ such that for all $m, \ell \ge \ell_0$, it holds $\|f_m-f_\ell\| \le 1$. Take and fix an arbitrary $a \in \B$, by Fatou's lemma, we have
\begin{eqnarray*}
&&\int_\B |f(z)-f_{\ell_0}(z)|^p(1-|z|^2)^q (1-|\Phi_a(z)|^2)^{ns}d\lambda(z)\\
&\le& \lim_{\ell \to \infty} \int_\B |f_{\ell}(z)-f_{\ell_0}(z)|^p(1-|z|^2)^q (1-|\Phi_a(z)|^2)^{ns}d\lambda(z)\\
&\le & \lim_{\ell \to \infty} \|f_{\ell}-f_{\ell_0}\|^p \le 1,
\end{eqnarray*}
which implies that
$$
\|f-f_{\ell_0}\|=\sup_{a \in \B} \int_\B |f(z)-f_{\ell_0}(z)|^p(1-|z|^2)^q (1-|\Phi_a(z)|^2)^{ns}d\lambda(z) \le 1,
$$
and hence $\|f\| \le 1+\|f_{\ell_0}\|< \infty$. Thus, combining this with the fact that the point evaluation on $\calN(p, q, s)$ is a continuous linear functional, we conclude that $\calN(p, q, s)$ is a functional Banach space.
\end{proof}

Next we show that when $p \ge 1$, $q>0$ and $s>1$, $\calN(p, q, s)=A^{-\frac{q}{p}}(\B)$. More precisely, we have the following result.

\begin{prop} \label{Npqsbigs}
Let $p \ge 1$, $q,s>0$. If $s>1-\frac{q-kp}{n}, k \in \left(0, \frac{q}{p}\right]$, then $A^{-k}(\B) \subseteq \calN(p, q, s)$. In particular, when $s>1, \calN(p, q, s)=A^{-\frac{q}{p}}(\B)$.
\end{prop}

\begin{proof}
Suppose $p \ge 1$, $q>0$ and $s>1-\frac{q-kp}{n}$ for some $k \in \left(0, \frac{q}{p}\right]$. Since $s>1-\frac{q-kp}{n}$, we have $q+ns-n-1-kp>-1$, then by \cite[Theorem 1.12]{Zhu},  for each $a \in \B$, we have
$$
\int_\B \frac{(1-|z|^2)^{q+ns-n-1-pk}}{|1-\langle a, z \rangle|^{2ns}}dV(z) \simeq \begin{cases}
\textrm{bounded in} \ \B, & \textrm{if} \ ns+pk<q;\\
\log \frac{1}{1-|a|^2}, &  \textrm{if} \ ns+pk=q;\\
(1-|a|^2)^{q-ns-pk}, & \textrm{if}\ ns+pk>q,
\end{cases}
$$
which implies, there exists a positive constant $C$ such that
\begin{equation} \label{eq04}
\sup_{a \in \B}  (1-|a|^2)^{ns} \int_\B \frac{(1-|z|^2)^{q+ns-n-1-pk}}{|1-\langle a, z \rangle|^{2ns}}dV(z)  \le C.
\end{equation}
Let $f \in A^{-k}(\B)$. By \eqref{eq04}, we have
\begin{eqnarray*}
\|f\|^p
&=& \sup_{a \in \B} \int_\B |f(z)|^p(1-|z|^2)^q(1-|\Phi_a(z)|^2|)^{ns}d\lambda(z) \\
&=& \sup_{a \in \B} \int_\B |f(z)|^p (1-|z|^2)^{pk} (1-|z|^2)^{q-pk} (1-|\Phi_a(z)|^2)^{ns} d\lambda(z)\\
&\le& |f|_k^p \sup_{a \in \B}   (1-|a|^2)^{ns} \int_\B \frac{(1-|z|^2)^{q+ns-n-1-pk}}{|1-\langle a, z \rangle|^{2ns}}dV(z) \\
&\le& C|f|_k^p,
\end{eqnarray*}
which implies $A^{-k}(\B) \subseteq \calN(p, q, s)$.

Now if $s>1$, then in particular, we can take $k=\frac{q}{p}$ and hence by the above argument, we have $A^{-\frac{q}{p}}(\B) \subseteq \calN(p, q, s)$. Combing this fact with Proposition \ref{boundaryineq}, we get the desired result.
\end{proof}

Letting $q=n+1$ and $p=2$ in Proposition \ref{Npqsbigs}, we get \cite[Theorem 2.1, (b)]{HK1} as a particular case.


\medskip

\subsection{The closure of all polynomials in $\calN(p, q, s)$-type spaces}


\begin{prop} \label{littleNpqs}
Let $p \ge 1$ and $q, s>0$. Then $\calN^0(p, q, s)$ is a closed subspace of $\calN(p, q, s)$ and hence $\calN^0(p, q, s)$ is a Banach space.
\end{prop}

\begin{proof}
First we note that it is trivial to see that $\calN^0(p, q, s)$ is a subspace of $\calN(p, q, s)$ and hence it suffice to show that $\calN^0(p, q, s)$ is complete. Suppose $\{f_n\}$ is a Cauchy sequence in $\calN^0(p, q, s)$, by Theorem \ref{functionalBS}, there is a limit $f \in \calN(p, q, s)$ of $\{f_n\}$. For any $\varepsilon>0$, there exists a $N \in \N$, such that when $n>N, \|f-f_n\|<\left(\frac{\varepsilon}{2^p}\right)^{1/p}$. Take $n_0>N$, since $f_{n_0} \in \calN^0(p, q, s)$, there exists a $\del \in (0, 1)$ such that when $\del<|a|<1$,
$$
\sup_{\del<|a|<1} \int_\B |f_{n_0}(z)|^p(1-|z|^2)^q(1-|\Phi_a(z)|^2)^{ns}d\lambda(z)<\frac{\varepsilon}{2^p}.
$$
Hence, we have
\begin{eqnarray*}
&&\sup_{\del<|a|<1} \int_\B |f(z)|^p(1-|z|^2)^q(1-|\Phi_a(z)|^2)^{ns}d\lambda(z)\\
&\le &  2^{p-1} \|f-f_{n_0}\|^p\\
&&  +2^{p-1} \sup_{\del<|a|<1} \int_\B |f_{n_0}(z)|^p(1-|z|^2)^q(1-|\Phi_a(z)|^2)^{ns}d\lambda(z)<\varepsilon,
\end{eqnarray*}
which implies that $f \in \calN^0(p, q, s)$.
\end{proof}

\begin{thm} \label{thm01}
Suppose $f \in \calN(p, q, s)$ with $ns+q>n$. Then $f \in \calN^0(p, q, s)$ if and only if
$\|f_r-f\| \to 0,  \ r\to 1,$ where $f_r(z)=f(rz)$ for all $z \in \B$.
\end{thm}

\begin{proof} {\it Necessity.}  Suppose $f \in \calN^0(p, q, s)$. This implies that for any $\varepsilon>0$, there exists a $\del>0$, such that with $\del<|a|<1$, we have
\begin{equation} \label{eq02}
\int_\B |f(z)|^p(1-|z|^2)^q(1-|\Phi_a(z)|^2)^{ns} d\lambda(z)<\frac{\varepsilon}{3 \cdot 2^{2n+p-1}}.
\end{equation}
Furthermore, by Schwarz-Pick Lemma (see, e.g., \cite[Theorem 8.1.4]{Rud}), we have
\begin{equation} \label{eq03}
|\Phi_{ra}(rz)| \le |\Phi_a(z)|, \ \textrm{for all} \ r \in (0,1) \  \textrm{and} \ a, z \in \B.
\end{equation}
Now take and fix $\del_0 \in (\del, 1)$. Consider $r$ satisfying $\max\left\{ \frac{1}{2}, \frac{\del}{\del_0}\right\}<r<1$. In this case, for all $a \in \B$ with $|a| \in (\del_0, 1)$, by \eqref{eq02} and \eqref{eq03}, we have
\begin{eqnarray*}
&& \int_\B |f(rz)|^p(1-|z|^2)^q (1-|\Phi_a(z)|^2)^{ns}d\lambda(z)\\
&\le & \int_\B |f(rz)|^p(1-|rz|^2)^q (1-|\Phi_{ra}(rz)|^2)^{ns} d\lambda(z)\\
&= &\left( \frac{1}{r}\right)^{2n} \int_\B |f(w)|^p(1-|w|^2)^q(1-|\Phi_{ra}(w)|^2)^{ns}d\lambda(w)\\
&\le & 4^n \int_\B |f(w)|^p(1-|w|^2)^q(1-|\Phi_{ra}(w)|^2)^{ns}d\lambda(w)<\frac{\varepsilon}{3 \cdot 2^{p-1}}.
\end{eqnarray*}

On the other hand, since $ns+q>n$, we have $ns+q-n-1>-1$. By \cite[Propostion 2.6]{Zhu}, $f_r$ converges to $f$ as $r \to 1$, in the norm topology of the weighted Bergman space $A^p_{ns+q-n-1}(\B)$. This implies that there exists a $r_1 \in (0, 1)$, such that for $r_1<r<1$, we have $$\int_\B |f(rz)-f(z)|^p(1-|z|^2)^{ns+q-n-q}dV(z)<\frac{(1-\del_0)^{2ns} \cdot \varepsilon}{3}.$$

Hence, for $|a| \le \del_0$ and $r_1<r<1$, we have
\begin{eqnarray*}
&&\sup_{|a| \le \del_0} \int_\B |f(rz)-f(z)|^p(1-|z|^2)^q(1-|\Phi_a(z)|^2)^{ns}d\lambda(z)\\
&= &\sup_{|a| \le \del_0} \left\{  (1-|a|^2)^{ns} \int_\B |f(rz)-f(z)|^p \frac{(1-|z|^2)^{ns+q-n-1}}{|1-\langle a, z \rangle|^{2ns}}dV(z)\right\}\\
&\le& \sup_{|a| \le \del_0}\int_\B |f(rz)-f(z)|^p \frac{(1-|z|^2)^{ns+q-n-1}}{|1-\langle a, z \rangle|^{2ns}}dV(z) \\
&\le& \frac{1}{(1-\del_0)^{2ns}} \int_\B |f(rz)-f(z)|^p (1-|z|^2)^{ns+q-n-1}dV(z)<\frac{\varepsilon}{3}.
\end{eqnarray*}
Consequently, for all $r$ with $\max\left\{\frac{1}{2}, \frac{\del}{\del_0}, r_1\right\}<r<1$, we have
\begin{eqnarray*}
&&\|f_r-f\|^p=\sup_{a \in \B} \int_\B |f(rz)-f(z)|^p(1-|z|^2)^q(1-|\Phi_a(z)|^2)^{ns} d\lambda(z)\\
&\le & \left(\sup_{|a| \le \del_0}+\sup_{\del_0<|a|<1}\right) \int_\B |f(rz)-f(z)|^p(1-|z|^2)^q(1-|\Phi_a(z)|^2)^{ns} d\lambda(z)\\
&\le& \frac{\varepsilon}{3}+2^{p-1}\sup_{\del_0<|a|<1} \int_\B (|f(rz)|^p+|f(z)|^p)(1-|z|^2)^q(1-|\Phi_a(z)|^2)^{ns} d\lambda(z)\\
& <&\frac{\varepsilon}{3}+2^{p-1} \left(\frac{\varepsilon}{3\cdot 2^{p-1}}+\frac{\varepsilon}{3 \cdot 2^{2n+p-1}}\right)
\\
&<&\frac{\varepsilon}{3}+2^{p-1} \left(\frac{\varepsilon}{3\cdot 2^{p-1}}+\frac{\varepsilon}{3 \cdot 2^{p-1}}\right)=\varepsilon,
\end{eqnarray*}
which shows that $\|f_r-f\| \to 0$ as $r \to 1$.

{\it Sufficiency.}  First we show that for each $r \in (0, 1)$ and each $f \in \calN(p, q, s)$, then $f_r \in \calN^0(p, q, s)$.
By \cite[Theorem 1.12]{Zhu}, we have
$$
\int_\B \frac{(1-|z|^2)^{q+ns-n-1}}{|1-\langle a, z \rangle|^{2ns}}dV(z)\simeq\begin{cases}
\textrm{bounded in} \ \B, & \textrm{if} \ ns<q;\\
\log \frac{1}{1-|a|^2}, &  \textrm{if} \ ns=q;\\
(1-|a|^2)^{q-ns}, & \textrm{if}\ ns>q,
\end{cases}
$$
which implies
$$
\sup_{a \in \B} (1-|a|^2)^{ns} \int_\B \frac{(1-|z|^2)^{q+ns-n-1}}{|1-\langle a, z \rangle|^{2ns}}dV(z)<\infty.
$$
Hence, we have
\begin{eqnarray*}
\|f_r\|^p
&=& \sup_{a \in \B} \int_\B |f_r(z)|^p(1-|z|^2)^q(1-|\Phi_a(z)|^2)^{ns} d\lambda(z)\\
&\le& \left( \sup_{z \in \B} |f(rz)|\right)^{p} \cdot  \sup_{a \in \B} \int_\B (1-|z|^2)^q(1-|\Phi_a(z)|^2)^{ns} d\lambda(z)\\
&=& \left( \sup_{z \in \B} |f(rz)|\right)^{p} \cdot \sup_{a \in \B} (1-|a|^2)^{ns} \int_\B \frac{(1-|z|^2)^{q+ns-n-1}}{|1-\langle a, z \rangle|^{2ns}}dV(z),
\end{eqnarray*}
which is finite. Thus, $f_r \in \calN(p, q, s)$. Moreover, we  have
\begin{eqnarray*}
&&\int_\B |f_r(z)|^p(1-|z|^2)^q(1-|\Phi_a(z)|^2)^{ns} d\lambda(z)\\
& \le&  M \int_\B (1-|z|^2)^q(1-|\Phi_a(z)|^2)^{ns} d\lambda(z),
\end{eqnarray*}
for some $M>0$ (this is due to the boundedness of the function $f(z)$ on $\{|z| \le r\}$).

Noting that $ns+q>n$ and by the estimation on the term

$\int_\B \frac{(1-|z|^2)^{q+ns-n-1}}{|1-\langle a, z \rangle|^{2ns}}dV(z)$ above, we have
$$
\int_\B (1-|z|^2)^q(1-|\Phi_a(z)|^2)^{ns} d\lambda(z) \to 0, \ \textrm{as} \ |a| \to 1,
$$
which implies that $f_r \in \calN^0(p, q, s)$ for all $r \in (0, 1)$.

Now suppose that $\|f_r-f\| \to 0$ as $r \to 1$. Then by the fact that $\calN^0(p, q, s)$ is a closed subspace of $\calN(p, q, s)$, it follows that $f \in \calN^0(p, q, s)$.
\end{proof}

As a corollary of Theorem \ref{thm01}, we obtain the following result.

\begin{cor} \label{cor01}
The set of polynomials is dense in $\calN^0(p, q, s)$ when $ns+q>n$.
\end{cor}

\begin{proof}
By Theorem \ref{thm01}, for any $f \in \calN^0(p, q, s)$, we have
$$
\lim_{r \to 1} \|f_r-f\|=0.
$$
Since each $f_r$ can be uniformly approximated by polynomials, and moreover, by the proof of Theorem \ref{thm01}, the sup-norm  in $\B$ dominates the $\calN(p, q, s)$-norm, we conclude that every $f \in \calN^0(p, q, s)$ can be approximated in the $\calN(p, q, s)$-norm by polynomials.
\end{proof}

However, generally, it is not true that $\calN(p, q, s)$ contains all the polynomials. Precisely, we have the following result.

\begin{prop} \label{prop02}
Let $p \ge 1$ and $q, s>0$. Then the set of polynomials are contained in $\calN(p, q, s)$ if and only if $ns+q>n$.
\end{prop}

\begin{proof}
The sufficiency clearly follows from Theorem \ref{thm01} and Corollary \ref{cor01}. Next we prove the necessity.  Assume $ns+q \le n$ and denote $\alpha=n-ns-q, \alpha \ge 0$. We claim that in this case, the constant function $F(z)=1, z \in \B$ does not belong to $\calN(p, q, s)$. Indeed,
\begin{eqnarray*}
\|F\|^p%
&=& \sup_{a \in \B} \int_\B (1-|z|^2)^q(1-|\Phi_a(z)|^2)^{ns}d\lambda(z)\\
&=&\sup_{a \in \B} (1-|a|^2)^{ns} \int_\B \frac{(1-|z|^2)^{q+ns-n-1}}{|1-\langle a, z \rangle|^{2ns}}dV(z)\\
&\ge& \int_\B \frac{1}{(1-|z|^2)^{\alpha+1}}dV(z)\quad (\textrm{Put} \ a=0) \\
&\simeq& \int_0^1 \frac{r^{2n-1}}{(1-r^2)^{1+\alpha}}dr \gtrsim \int_{1/2}^1 \frac{1}{(1-r)^{1+\alpha}}dr\\
&=& \infty,
\end{eqnarray*}
which is a contradiction. Hence, we get the desired result.
\end{proof}


\medskip

\subsection{Description given by Green's function}


In \cite{MS} and \cite{DU}, the invariant Green's function is defined as $G(z, a)=g(\Phi_a(z))$, where
$$
g(z)=\frac{n+1}{2n} \int_{|z|}^1 (1-t^2)^{n-1}t^{-2n+1}dt.
$$
The following property of $g$ is important (see, e.g., \cite{OYZ}).

\begin{prop} \label{gproperty}
Let $n \ge 2$  be an integer. Then there are positive constants $C_1$ and $C_2$ such that for all $z \in \B \backslash \{0\}$,
$$
C_1(1-|z|^2)^n |z|^{-2(n-1)} \le g(z) \le C_2 (1-|z|^2)^n |z|^{-2(n-1)}.
$$
\end{prop}

For $p \ge 1, q, s>0$ and $n \ge 2$, we define the following $\calN_*(p, q, s)$-type space:
\[ \begin{split}
&\calN_*(p, q, s):=\\
&\left\{ f\in H(\B): \|f\|_*^p=\sup_{a \in \B} \int_\B
|f(z)|^p(1-|z|^2)^q G^s(z, a) d\lambda(z)<\infty \right\},
\end{split}
\]
where the corresponding little space is defined as
\[ \begin{split}
&\calN_*^0(p, q, s):=\\
&\left\{ f\in \calN_*(p, q, s): \lim_{|a| \to 1} \int_\B |f(z)|^p(1-|z|^2)^q G^s(z, a) d\lambda(z)=0\right\}.
\end{split}
\]
Noting that by Proposition \ref{gproperty}, it is clear that $\| \cdot\| \lesssim \| \cdot\|_*$, that is, $\calN_*(p, q, s) \subseteq \calN(p, q, s)$. Combining this fact with the proof of Proposition \ref{boundaryineq}, Theorem \ref{functionalBS} and Proposition \ref{littleNpqs}, we have the following result.\

\begin{thm}
For $p \ge 1, q,s>0$ and $n \ge 2$. $\calN_*(p, q, s)$ is a functional Banach space and  $\calN_*^0(p, q, s)$ is a closed subspace of $\calN_*(p, q, s)$. Moreover, $\calN_*(p, q, s) \subseteq A^{-\frac{q}{p}}(\B)$.
\end{thm}

However, generally, it is not true that $\calN_*(p, q, s)$ contain all the polynomials, for example, by Proposition \ref{prop02} and the fact that $\|\cdot\| \lesssim \|\cdot\|_*$, we know that when $ns+q \le n$, $F \notin \calN_*(p, q, s)$, where $F(z)=1,  z \in \B$. We are interested in the following question: when does $\calN_*(p, q ,s)$ contain the set of polynomials?  We have the following result.

\begin{prop} \label{N*poly}
Let $p \ge 1, q, s>0$ and $n \ge 2$. Then the set of polynomials are contained in $\calN_*(p, q, s)$ if and only if $ns+q>n$ and $s<\frac{n}{n-1}$.
\end{prop}

\begin{proof} {\it Necessity.} We prove it by contradiction. Consider the constant function $F(z)=1, z\in \B$. By Proposition \ref{gproperty}, we have
\begin{eqnarray*}
\|F\|_*^p%
&=& \sup_{a \in \B} \int_\B (1-|z|^2)^q G^s(z, a)d\lambda(z) \ge \int_\B (1-|z|^2)^q G^{s}(z) d\lambda(z)\\
& \simeq& \int_\B (1-|z|^2)^q \frac{(1-|z|^2)^{ns}}{|z|^{2(n-1)s}} d\lambda(z)= \int_\B \frac{(1-|z|^2)^{q+ns-n-1}}{|z|^{2(n-1)s}}dV(z)\\
&=& I_1+I_2,
\end{eqnarray*}
where
$$
I_1=\int_{\B_{1/2}} \frac{(1-|z|^2)^{q+ns-n-1}}{|z|^{2(n-1)s}}dV(z)
$$
and
$$
 I_2=\int_{\B \backslash \B_{1/2}} \frac{(1-|z|^2)^{q+ns-n-1}}{|z|^{2(n-1)s}}dV(z).
$$
We consider two cases.\\
{\it Case I: $ns+q \le n$.} We have $I_2=\infty$, which is a contradiction. Indeed,
$$
I_2 \simeq \int_{\B \backslash \B_{1/2}} (1-|z|^2)^{q+ns-n-1}dV(z) \simeq \int_{1/2}^1 \frac{1}{(1-r)^{n+1-q-ns}}dr=\infty.
$$
{\it Case II: $s \ge \frac{n}{n-1}$.} We claim in this case, $I_1=\infty$, which contradicts to the assumption that $F \in \calN_*(p, q, s)$. Indeed, we have
$$
I_1 \simeq \int_{\B_{1/2}} \frac{1}{|z|^{2(n-1)s}}dV(z) \simeq \int_0^{1/2} \frac{1}{r^{2(n-1)s-2n+1}}dr=\infty.
$$

{\it Sufficiency.} Suppose $ns+q>n, s<\frac{n}{n-1}$ and $P$ is a polynomial defined on $\B$. Then we have
\begin{eqnarray*}
\|P\|_*^p%
&=& \sup_{a \in \B} \int_\B |P(z)|^p(1-|z|^2)^q G^s(z, a)d\lambda(z)\\
&\lesssim& \sup_{a \in \B} \int_\B  (1-|z|^2)^q G^s(z, a)d\lambda(z) \quad (\textrm{$P$ is bounded on} \ \B)\\
&\simeq&  \sup_{a \in \B} \int_\B (1-|z|^2)^p \frac{(1-|\Phi_a(z)|^2)^{ns}}{|\Phi_a(z)|^{2(n-1)s}}d\lambda(z) \\
&=& \sup_{a \in \B} \int_\B (1-|\Phi_a(w)|^2)^q \frac{(1-|w|^2)^{ns}}{|w|^{2(n-1)s}}d\lambda(w)\\
&& \quad (\textrm{change variable} \  w=\Phi_a(z)) \\
&=& \sup_{a \in \B} (1-|a|^2)^q \int_\B \frac{(1-|w|^2)^{q+ns-n-1}}{|w|^{2(n-1)s}|1-\langle a, w\rangle|^{2q}}dV(w). \\
\end{eqnarray*}
For each $a \in \B$, consider the term
$$
\int_\B \frac{(1-|w|^2)^{q+ns-n-1}}{|w|^{2(n-1)s}|1-\langle a, w\rangle|^{2q}}dV(w)=I_{1, a}+I_{2,a},
$$
where
$$
I_{1, a}=\int_{\B_{1/2}} \frac{(1-|w|^2)^{q+ns-n-1}}{|w|^{2(n-1)s}|1-\langle a, w\rangle|^{2q}}dV(w)
$$
and
$$
I_{2, a}=\int_{\B \backslash \B_{1/2}} \frac{(1-|w|^2)^{q+ns-n-1}}{|w|^{2(n-1)s}|1-\langle a, w\rangle|^{2q}}dV(w).
$$

For $I_{1, a}$, by the proof of necessary part, we have
$$
I_{1, a} \simeq \int_{\B_{1/2}} \frac{1}{|w|^{2(n-1)s}}dV(w)<M,
$$
for some $M>0$, which is independent of the choice of $a$. Thus, we have $\sup\limits_{a \in \B} (1-|a|^2)^q I_{1, a}<\infty$.

For $I_{2, a}$, by \cite[Theorem 1.12]{Zhu}, we have
\begin{eqnarray*}
I_{2, a}%
&\simeq& \int_{\B \backslash \B_{1/2}} \frac{(1-|w|^2)^{q+ns-n-1}}{|w|^{2(n-1)s}|1-\langle a, w\rangle|^{2q}}dV(w)\\
& \lesssim&  \int_{\B} \frac{(1-|w|^2)^{q+ns-n-1}}{|1-\langle a, w\rangle|^{2q}}dV(w)\\
&\simeq&\begin{cases}
\textrm{bounded in} \ \B, & \textrm{if} \ ns>q;\\
\log \frac{1}{1-|a|^2}, &  \textrm{if} \ ns=q;\\
(1-|a|^2)^{ns-q}, & \textrm{if}\ ns<q,
\end{cases}
\end{eqnarray*}
which implies $\sup\limits_{a \in \B} (1-|a|^2)^q I_{2, a}<\infty$.

Combing the above estimations, we see that $\|P\|_*<\infty$, which implies the desired result.
\end{proof}

The above proposition provide us a hint on describing the $\calN(p, q, s)$-spaces by using the invariant Green's function. More precisely, we have the following result.

\begin{thm} \label{relationGreen}
Let $p \ge 1, q, s>0$ and $n \ge 2$. If $s<\frac{n}{n-1}$, then $\calN(p, q, s)=\calN_*(p, q, s)$. In particular, if $1<s<\frac{n}{n-1}$, then $\calN(p, q, s)=\calN_*(p, q, s)=A^{-\frac{q}{p}}(\B)$.
\end{thm}

\begin{proof}
Clearly, $\calN_*(p, q, s) \subseteq \calN(p, q, s)$ and hence it suffices to show $\calN(p, q, s) \subseteq \calN_*(p, q, s)$ when $s<\frac{n}{n-1}$. Take $f \in \calN_*(p, q, s)$. For each $a \in \B$, we have
\begin{eqnarray*}
&&\int_\B |f(z)|^p(1-|z|^2)^q G^s(z, a)d\lambda(z)\\
&\simeq& \int_\B |f(z)|^p(1-|z|^2)^q \frac{(1-|\Phi_a(z)|^2)^{ns}}{|\Phi_a(z)|^{2(n-1)s}}d \lambda(z)\\
&= &\int_\B |f(\Phi_a(w))|^p (1-|\Phi_a(w)|^2)^q \frac{(1-|w|^2)^{ns}}{|w|^{2(n-1)s}}d\lambda(w)\\
&&  (\textrm{change variable} \  w=\Phi_a(z)) \\
& =&J_{1, a}+J_{2, a},
\end{eqnarray*}
where
$$
J_{1, a}=\int_{\B_{1/2}} |f(\Phi_a(w))|^p (1-|\Phi_a(w)|^2)^q \frac{(1-|w|^2)^{ns}}{|w|^{2(n-1)s}}d\lambda(w)
$$
and
$$
J_{2, a}=\int_{\B \backslash \B_{1/2}} |f(\Phi_a(w))|^p (1-|\Phi_a(w)|^2)^q \frac{(1-|w|^2)^{ns}}{|w|^{2(n-1)s}}d\lambda(w).
$$
For $J_{2, a}$, we have
\begin{eqnarray*}
J_{2, a}%
&\lesssim& \int_{\B \backslash \B_{1/2}} |f(\Phi_a(w))|^p (1-|\Phi_a(w)|^2)^q (1-|w|^2)^{ns}d\lambda(w)\\
&\le& \int_\B |f(\Phi_a(w))|^p (1-|\Phi_a(w)|^2)^q (1-|w|^2)^{ns}d\lambda(w)\\
&=& \int_\B |f(z)|^p(1-|z|^2)^q(1-|\Phi_a(z)|^2)^{ns} d\lambda(z) \le \|f\|^p.
\end{eqnarray*}
For $J_{1, a}$, by Proposition \ref{boundaryineq}, $s<\frac{n}{n-1}$ and the proof of Proposition \ref{N*poly}, we have
\begin{eqnarray*}
J_{1, a}%
&\le& |f|_{q/p}^p \int_{\B_{1/2}} \frac{(1-|w|^2)^{ns-n-1}}{|w|^{2(n-1)s}}dV(w) \\
&\lesssim& \|f\|^p \int_{\B_{1/2}} \frac{1}{|w|^{2(n-1)s}}dV(w) \le M \|f\|^p,
\end{eqnarray*}
for some $M>0$.

Hence, combining the estimations on $J_{1, a}$ and $J_{2, a}$, we have, for each $a \in \B$,
$$
\int_\B |f(z)|^p(1-|z|^2)^q G^s(z, a)d\lambda(z) \lesssim \|f\|^p,
$$
which implies $\|f\|_* \lesssim \|f\|$, that is, $\calN(p, q, s) \subseteq \calN_*(p, q, s)$ if $s<\frac{n}{n-1}$.

Finally, when $1<s<\frac{n}{n-1}$, the desired result follows from the above argument and Proposition \ref{Npqsbigs}.
\end{proof}

The following result is straightforward from Theorem \ref{relationGreen}.

\begin{cor}
Let $n \ge 2$ and $0<s<\frac{n}{n-1}$. Then
$$
\calN_{ns}=\calN(2, n+1, s)=\calN_*(2, n+1, s).
$$
Moreover, if $1<s<\frac{n}{n-1}$, then $A^{-\frac{n+1}{2}}=\calN(2, n+1, s)=\calN_*(2, n+1, s)$.
\end{cor}

Our next result shows that  $\calN_*(p, q, s)$ is trivial if $s \ge \frac{n}{n-1}$.

\begin{prop} \label{bigscondition}
Let $p \ge 1, q, s>0$ and $n \ge 2$. If $s \ge\frac{n}{n-1}$, then $\calN_*(p, q, s)$ only contains the zero function.
\end{prop}

\begin{proof}
We prove it by contradiction. Assume there exists a $a_0 \in \B$, such that $|f(a_0)| \ge \del>0$. Then there exists a $r>0$, such that
$$
|f(\Phi_a(w))| \ge \frac{\del}{2}, \quad |w|<r.
$$
 Therefore we have
\begin{eqnarray*}
\|f\|_*^p%
&=&  \sup_{a \in \B} \int_\B |f(z)|^p (1-|z|^2)^q G^s(z, a) d\lambda(z)\\
& \ge& \int_\B |f(z)|^p(1-|z|^2)^q G^s(z, a_0) d\lambda(z) \\
& \gtrsim& \int_\B |f(z)|^p(1-|z|^2)^q \frac{(1-|\Phi_{a_0}(z)|^2)^{ns}}{|\Phi_{a_0}(z)|^{2(n-1)s}}d\lambda(z)\\
&=& \int_\B |f(\Phi_{a_0}(w)|^p(1-|\Phi_{a_0}(w)|^2)^q \frac{(1-|w|^2)^{ns-n-1}}{|w|^{2(n-1)s}}dV(w)\\
&& (\textrm {change variable} \ z=\Phi_a(w))\\
&\ge& \int_{|w|<r} |f(\Phi_{a_0}(w)|^p(1-|\Phi_{a_0}(w)|^2)^q \frac{(1-|w|^2)^{ns-n-1}}{|w|^{2(n-1)s}}dV(w)\\
&\ge& \left(\frac{\del}{2}\right)^p (1-|a_0|^2)^q \int_{|w|<r} \frac{(1-|w|^2)^{ns+q-n-1}}{|w|^{2(n-1)s} |1-\langle a, w \rangle|^{2q}}dV(w) \\
& \gtrsim& \int_{|w|<r} \frac{1}{|w|^{2(n-1)s}}dV(w) \gtrsim \int_0^r \frac{1}{t^{2(n-1)s-2n+1}}dt=\infty,
\end{eqnarray*}
which is a contradiction.
\end{proof}

\begin{rem}
By Theorem \ref{relationGreen} and Proposition \ref{bigscondition}, it is clear that $\calN_*(p, q, s)$-type space is a special case of $\calN(p, q, s)$-type space. Hence, in the sequel, we will focus our interest on $\calN(p, q, s)$-type spaces.
\end{rem}


\bigskip

\section{Hadamard gaps in $\calN(p, q, s)$-type spaces}


A holomorphic function $f$ on $\B$ written in the form
$$
f(z)=\sum_{k=0}^\infty P_{n_k}(z),
$$
where $P_{n_k}$ is a homogeneous polynomial of degree $n_k$, is said to have \emph{Hadamard gaps} if  for some $c>1$ (see, e.g., \cite{SS}),
$$
n_{k+1}/n_k \ge c, \forall k \ge 0.
$$

Given a Hadamard gap series, we are interested in the following question: for $p \ge 1$ and $q, s>0$, when does this Hadamard gap series belongs to the $\calN(p, q, s)$ spaces?

Observing that a constant function has Hadamard gaps, and hence by Proposition \ref{prop02} or Proposition \ref{N*poly}, we always assume that the condition $ns+q>n$ holds, that is, $s>1-\frac{q}{n}$. Moreover, note that  by Proposition \ref{Npqsbigs}, if $s>1$, then $\calN(p, q, s)=A^{-\frac{q}{p}}(\B)$ and for this case, it was already studied by the authors in \cite[Theorem 2.5]{HL}. Hence, we also assume that $s \le 1$ in this section.

To formulate our main result in this section, we denote
$$
M_k=\sup_{\xi \in \SSS} |P_{n_k}(\xi)| \quad \textrm{and} \quad L_{k, p}=\left( \int_\SSS |P_{n_k}(\xi)|^pd\sigma(\xi) \right)^{1/p}, \ p \ge 1,
$$
where $d\sigma$ is the normalized surface measure on $\SSS$, that is, $\sigma(\SSS)=1$. Clearly for each $k \ge 0$ and $p \ge 1$, $M_k$ and $L_{k, p}$ are well-defined.

We have the following result.

\begin{thm} \label{HadamardNpqs}
Let $p \ge 1, q>0$ and $\max\left\{0, 1-\frac{q}{n}\right\}<s \le 1$ and $f(z)=\sum_{k=0}^\infty P_{n_k}(z)$ with Hadamard gaps. Consider the following statements
\begin{enumerate}
\item[(a)] $\sum\limits_{k=0}^\infty \frac{1}{2^{k(ns+q-n)}}\Big( \sum\limits_{2^k \le n_j<2^{k+1}} M_j^p\Big)<\infty$;
\item[(b)] $f \in \calN^0(p, q, s)$;
\item[(c)] $f \in \calN(p, q, s)$;
\item[(d)] $\sum\limits_{k=0}^\infty \frac{1}{2^{k(ns+q-n)}}\Big( \sum\limits_{2^k \le n_j<2^{k+1}} L_{j, p}^p\Big)<\infty$.
\end{enumerate}
We have $(a) \Longrightarrow (b) \Longrightarrow (c) \Longrightarrow (d)$.
\end{thm}

\begin{proof} $   (a) \Longrightarrow (b).$  Suppose that $(a)$ holds. First, we prove that $f \in \calN(p, q, s)$. For $f(z)=\sum_{k=0}^\infty P_{n_k}(z)$, using the polar coordinates and \cite[Lemma 1.8]{Zhu}, we have
\begin{eqnarray*}
&&\|f\|^p= \sup_{a \in \B} \int_\B \Big|\sum_{k=0}^\infty P_{n_k}(z)\Big|^p(1-|z|^2)^q(1-|\Phi_a(z)|^2)^{ns}d\lambda(z)\\
&\le &\sup_{a \in \B} \int_\B \Big(\sum_{k=0}^\infty |P_{n_k}(z)|\Big)^p(1-|z|^2)^q(1-|\Phi_a(z)|^2)^{ns}d\lambda(z)\\
&= &\sup_{a \in \B}  \int_\B \Big (\sum_{k=0}^\infty |P_{n_k}(z)|\Big)^p \frac{(1-|a|^2)^{ns}(1-|z|^2)^{ns+q-n-1}}{|1-\langle a, z \rangle|^{2ns}}dV(z) \\
&\lesssim& \sup_{a \in \B}   \int_0^1 (1-|r|^2)^{ns+q-n-1}  \int_\SSS  \frac{(1-|a|^2)^{ns}\Big( \sum\limits_{n=0}^\infty |P_{n_k}(r\xi)|\Big)^p}{|1-\langle a, r\xi \rangle|^{2ns}}d\sigma(\xi)  dr .
\end{eqnarray*}
Since for each $k \in \N$, $P_{n_k}$ is homogeneous, we have for each $\xi \in \SSS$,
\begin{equation} \label{Hadamard01}
\sum_{k=0}^\infty |P_{n_k}(r\xi)|=\sum_{k=0}^\infty |P_{n_k}(\xi)r^{n_k}| \le \sum_{k=0}^\infty M_{n_k}r^{n_k}.
\end{equation}
Moreover, for each $a \in \B$, by \cite[Theorem 1.12]{Zhu}, we have
\begin{eqnarray*}
 &&\int_\SSS \frac{1}{|1-\langle r\xi, a \rangle|^{2ns}}d\sigma(\xi)=\int_\SSS \frac{1}{|1-\langle \xi, ra \rangle|^{2ns}}d\sigma(\xi)\\
&\simeq&\begin{cases}
\textrm{bounded in} \ \B, & \textrm{if} \ s<\frac{1}{2};\\
\log \frac{1}{1-r^2|a|^2} \le \log \frac{1}{1-|a|^2}, &  \textrm{if} \ s=\frac{1}{2};\\
(1-r^2|a|^2)^{n-2ns} \le(1-|a|^2)^{n-2ns}, & \textrm{if}\ \frac{1}{2}<s<1,
\end{cases}
\end{eqnarray*}
which implies, there exists a positive constant $C$ such that
\begin{equation} \label{Hadamard02}
\sup_{a \in \B}  (1-|a|^2)^{ns} \int_\SSS \frac{1}{|1-\langle r\xi, a \rangle|^{2ns}}dV(z)  \le C.
\end{equation}
Thus, by \eqref{Hadamard01}, \eqref{Hadamard02} and \cite[Theorem 1]{MM}, we have
\begin{eqnarray*}
\|f\|^p%
&\lesssim& \int_0^1  \Big(\sum_{k=0}^\infty M_{n_k}r^{n_k}\Big)^p(1-|r|^2)^{ns+q-n-1}dr\\
&\simeq& \sum_{k=0}^\infty \frac{1}{2^{k(ns+q-n)}} \Big(\sum_{2^k \le n_j<2^{k+1}} M_j \Big)^p.
\end{eqnarray*}
Since $f$ is in the Hadamard gaps class, there exists a constant $c>1$ such that $n_{j+1} \le cn_j$ for all $j \ge 0$. Hence, the maximum number of $n_j$'s between $2^k$ and $2^{k+1}$ is less or equal to $[\log_c 2]+1$ for $k=0, 1, 2, \dots$.

Since for every $k \ge 0$, by H\"older inequality,
$$
\Big( \sum_{2^k \le n_j<2^{k+1}} M_j \Big)^p \le \left( [\log_c 2]+1\right)^{p-1} \Big( \sum_{2^k \le n_j<2^{k+1}} M_j^p\Big),
$$
Thus, we have
\begin{equation} \label{Hadamard03}
\|f\|^p \lesssim \sum_{k=0}^\infty  \frac{1}{2^{k(ns+q-n)}}\Big( \sum_{2^k \le n_j<2^{k+1}} M_j^p\Big)<\infty,
\end{equation}
which implies $f \in \calN(p, q, s)$.

Next, we prove that $f \in \calN(p, q, s)$. Put $$f_m(z)=\sum_{k=0}^m P_{n_k}(z), m \in \N,$$
  which is bounded in $\overline{\B}$. Thus, by the proof of Theorem \ref{thm01}, we know that for each $m \in \N, f_m \in \calN^0(p, q, s)$. Moreover, by Corollary \ref{cor01}, $\calN^0(p, q, s)$ is closed and the set of all polynomials is dense in $\calN^0(p, q, s)$, and hence it suffices to show that $\|f_m-f\| \to 0$ as $ \to \infty$. By \eqref{Hadamard03}, we have
\begin{equation} \label{Hadamard04}
\|f_m-f\|^p \lesssim \sum_{k=m'}^\infty \Big( \frac{1}{2^{k(ns+q-n)}} \sum_{2^k \le n_j<2^{k+1}}M_j^p\Big),
\end{equation}
where $m'=\left[ \frac{m+1}{[\log_c 2]+1}\right]$. The result follows from condition $(a)$ and \eqref{Hadamard04}.

\medskip

 $   (b) \Longrightarrow (c)$.  It is obvious.

\medskip

 $    (c) \Longrightarrow (d)$.  Suppose $f \in \calN(p, q, s)$. As the proof in \cite[Theorem 1]{SS}, we have
\begin{eqnarray*}
\|f\|^p%
&=&\sup_{a \in \B} \int_\B \left| \sum_{k=0}^\infty P_{n_k}(z)\right|^p(1-|z|^2)^q(1-|\Phi_a(z)|^2)^{ns}d\lambda(z)\\
&\ge& \int_\B  \left| \sum_{k=0}^\infty P_{n_k}(z)\right|^p(1-|z|^2)^{q+ns-n-1}dV(z)\\
&\simeq& \int_\SSS \Big(\sum_{k=0}^\infty \frac{1}{2^{k(ns+q-n)}} \sum_{2^k \le n_j<2^{k+1}} |P_{n_k}(\xi)|^p\Big)d\sigma(\xi) \\
&=&\sum_{k=0}^\infty \frac{1}{2^{k(ns+q-n)}}\Big( \sum_{2^k \le n_j<2^{k+1}} L_j^p\Big),
\end{eqnarray*}
which implies the desired result.
\end{proof}

Letting $q=n+1, p=2$ and $0<s<1$ in Theorem \ref{HadamardNpqs}, we get \cite[Theorem 2.1]{HL} as a particular case. Generally, condition $(d)$ does not imply condition $(a)$, an example can be found in \cite[Remark 2.2]{HL}.

Next we consider some special cases when all the conditions in Theorem \ref{HadamardNpqs} are equivalent.

In \cite[Corollary 1]{DU2}, for $p>0$, the authors constructed a sequence of homogeneous polynomials $\{W_k\}_{k \in \N}$ satisfying $\deg(W_k)=k$,
$$
\sup_{\xi \in \SSS} |W_k(\xi)|=1 \quad \textrm{and} \quad \int_\SSS |W_k(\xi)|^p d\sigma(\xi) \ge C(p, n),
$$
where $C(p, n)$ is a positive constant depending on $p$ and $n$.

An immediate corollary of Theorem \ref{HadamardNpqs} is stated as follows.

\begin{cor} \label{RWsequence}
Let $p \ge 1$, $q>0$ and $\max\left\{0, 1-\frac{q}{n} \right\}<s \le 1$ and $f(z)=\sum_{k=0}^\infty a_kW_{n_k}(z)$ with Hadamard gaps, where $a_k \in \C, k \ge 0$. Then the following statements are equivalent.
\begin{enumerate}
\item[(a)] $\sum\limits_{k=0}^\infty \frac{1}{2^{k(ns+q-n)}}\Big( \sum\limits_{2^k \le n_j<2^{k+1}} |a_j|^p\Big)<\infty$;
\item[(b)] $f \in \calN^0(p, q, s)$;
\item[(c)] $f \in \calN(p, q, s)$.
\end{enumerate}
\end{cor}

\begin{proof}
The desired result follows form the fact that for each $k \ge 0, M_k \simeq L_{k, p}$.
\end{proof}

Let $p \ge 1, q>0$ and $\max\left\{0, 1-\frac{q}{n}\right\}<s_1<s_2 \le 1$. It is clear that
\begin{equation} \label{embeddingrelation}
\calN(p, q, s_1) \subseteq \calN(p, q, s_2) \subseteq A^{-\frac{q}{p}}(\B).
\end{equation}
The second application of our main result in this section is to show that the inclusions in \eqref{embeddingrelation} is strict.

\begin{cor} \label{Hadamard10}
Let $p \ge 1, q>0$ and $\max\left\{0, 1-\frac{q}{n}\right\}<s_1<s_2 \le 1$. Then
$$
\calN(p, q, s_1) \subsetneq \calN(p, q, s_2) \subsetneq A^{-\frac{q}{p}}(\B).
$$
\end{cor}

\begin{proof} First we prove that $\calN(p, q, s_2) \subsetneq A^{-\frac{q}{p}}(\B).
$ Consider the series $f_1(z)=\sum\limits_{k=0}^\infty 2^{\frac{kq}{p}}W_{2^k}(z)$. On one hand, by \cite[Theorem 2.5]{HL}, $f_1 \in A^{-\frac{q}{p}}$. On the other hand,
$$
\sum_{k=0}^\infty \frac{1}{2^{k(ns_2+q-n)}} \Big(\sum_{2^k \le n_j<2^{k+1}} \left|2^{\frac{kq}{p}}\right|^p\Big)=\sum_{k=0}^\infty \frac{1}{2^{k(ns_2-n)}}=\infty,
$$
which, by Corollary \ref{RWsequence}, implies $f_1 \notin \calN(p, q, s_2)$.

Next we show that $ \calN(p, q, s_1) \subsetneq \calN(p, q, s_2).$ Consider the series $$f_2(z)=\sum\limits_{k=0}^\infty 2^{\frac{k(ns_1+q-n)}{p}}W_{2^k}(z).$$  On one hand, by Corollary \ref{RWsequence},
$$
\sum_{k=0}^\infty \frac{1}{2^{k(ns_2+q-n)}} \Big(\sum_{2^k \le n_j<2^{k+1}} \left|2^{\frac{k(ns_1+q-n)}{p}}\right|^p\Big)=\sum_{k=0}^\infty \frac{1}{2^{kn(s_2-s_1)}}<\infty,
$$
which implies that $f_2 \in \calN(p, q, s_2)$. On the other hand,
$$
\sum_{k=0}^\infty \frac{1}{2^{k(ns_1+q-n)}} \Big(\sum_{2^k \le n_j<2^{k+1}} \left|2^{\frac{k(ns_1+q-n)}{p}}\right|^p\Big)=\sum_{k=0}^\infty 1=\infty,
$$
which, again by Corollary \ref{RWsequence}, implies $f_2 \notin \calN(p, q, s_1)$.
\end{proof}

Letting $q=n+1, p=2$ and $0<s \le 1$ in Corollary \ref{Hadamard10}, we get the corresponding results in \cite{HKL} as a particular case.

\begin{rem}
From the above corollary, it is also straightforward to see that for any $k \in \left(0, \frac{q+ns-n}{p} \right)$, we have
$$
A^{-k}(\B) \subsetneq \calN(p, q, s).
$$
Indeed, we can find an $s'$ satisfying $k< \frac{q+ns'-n}{p}<\frac{q+ns-n}{p}$. By the above corollary and Corollary \ref{RWsequence}, it follows that
$$
A^{-k}(\B) \subseteq \calN(p, q, s') \subsetneq \calN(p, q, s),
$$
which implies the desired claim.
\end{rem}


\bigskip

\section{Carselon measure, Hadamard products and Random power series}


\subsection{Carselon measure}


First we give an equivalent expression of $\calN(p, q, s)$-norm by Carleson measures. Recall that for $\xi\in\SSS$ and $r>0$, a Carleson tube at $\xi$ is defined as (see, e.g., \cite{Zhu})
$$
Q_r(\xi)=\{z\in\B: |1-\langle z,\xi\rangle|<r\}.
$$
Moreover, we denote $Q(\xi, r)=\{w \in \SSS: |1-\langle w, \xi \rangle|<r\}$.

A positive Borel measure $\mu$ in $\B$ is called a \emph{$p$-Carleson measure} if there exists a constant $C>0$ such that
$$
\mu(Q_r(\xi)) \le Cr^{p}
$$
for all $\xi\in\SSS$ and $r>0$. Moreover, if
$$
\lim_{r\to0} \frac{\mu(Q_r(\xi))}{r^p}=0
$$
uniformly for $\xi \in \SSS$, then $\mu$ is called a \emph{vanishing $p$-Carleson measure}.

The following result describes a relationship between functions in $\calN(p, q, s)$ as well as $\calN^0(p, q, s)$ and Carleson measures.

\begin{prop} \label{Carleson01}
Let $f \in H(\B)$ and $p \ge 1, q, s>0$, and $d\mu_{f, p,q,s}(z)=|f(z)|^p(1-|z|^2)^{q+ns}d\lambda(z)$. The following assertions hold.
\begin{enumerate}
\item $f \in \calN(p, q, s)$ if and only if $d\mu_{f, p, q, s}$ is an $(ns)$-Carleson measure;
\item $f \in \calN^0(p, q, s)$ if and only if $d\mu_{f, p, q, s}$ is a vanishing $(ns)$-Carleson measure.
\end{enumerate}
Moreover,
\begin{equation} \label{Carleson}
 \begin{split}
\|f\|^p &\simeq \sup_{0<r<1, \xi \in \SSS} \frac{\mu_{f, p, q, s}(Q_r(\xi))}{r^{ns}}\\
&=\sup_{0<r<1, \xi \in \SSS} \frac{1}{r^{ns}} \int_{Q_r(\xi)}|f(z)|^p(1-|z|^2)^{q+ns}d\lambda(z).
\end{split}
\end{equation}
\end{prop}

\begin{proof}
(1) Note tht for $f \in \calN(p, q, s)$, we can write
\begin{eqnarray*}
\|f\|^p%
&=& \sup_{a \in \B} \int_\B |f(z)|^p \frac{(1-|a|^2)^{ns}(1-|z|^2)^{q+ns-n-1}}{|1-\langle a, z \rangle|^{2ns}}dV(z)\\
&=& \sup_{a \in \B} \int_\B \left( \frac{1-|a|^2}{|1-\langle a, z \rangle|^2}\right)^{ns} d\mu_{f, p, q, s}(z).
\end{eqnarray*}
Then (1) is obtained, by \cite[Theorem 45]{ZZ}. Moreover, the equation \eqref{Carleson} also follows from \cite[Theorem 45]{ZZ}.

(2) This is a consequence of the ``little-oh version" of \cite[Theorem 45]{ZZ}.
\end{proof}


\medskip

\subsubsection{Embedding relationship with weighted Bergman space}


For $k>0$ and $f$ holomorphic in $\B$, we denote
$$
M_k(r, f)=\left( \int_\SSS |f(r\xi)|^k d\sigma(\xi)\right)^{1/k}, 0 \le r<1.
$$
By using this expression, we can rewrite the norm of $\| \cdot \|_{k, \rho}$ of the weighted Bergman space $A^{k}_\rho$ with $k \ge 1$ and $\rho>-1$ as follows.
\begin{eqnarray*}
\|f\|_{k, \rho}%
&\simeq& \left(\int_0^1 r^{2n-1}(1-r^2)^{\rho} M^k_k(r, f) dr \right)^{1/k} \\
&\simeq&  \left(\int_0^1 r^{2n-1}(1-r)^{\rho} M^k_k(r, f) dr \right)^{1/k}.
\end{eqnarray*}
As an application of Proposition \ref{Carleson01}, we establish the following embedding relation between $A^k_\rho$ and $\calN(p, q, s)$ with some proper condition on $k$ and $\rho$. Note that, by Proposition \ref{prop02} and the fact that the set of all polynomials belongs to $A^k_\rho$, it is natural for us to assume that $ns+q>n$.

 We have the following result.

\begin{prop} \label{Bergmanembedding}
Let $p \ge 1, q>0$ and $s>\max\left\{0, 1-\frac{q}{n}\right\}$. Then the following assertions hold.
\begin{enumerate}
\item If $0<s<1$, then for $\max\{0, q-n\}<\rho<\frac{q+ns-n}{1-s}$, we have $\|f\| \lesssim \|f\|_{\frac{p(n+\rho)}{q}, \rho-1}$, that is, $A^{\frac{p(n+\rho)}{q}}_{\rho-1} \subseteq \calN(p, q, s)$;
\item If $s \ge 1$, then for $\rho>\max\{0, q-n\}$, we have $\|f\| \lesssim \|f\|_{\frac{p(n+\rho)}{q}, (\rho-1)}$, that is, $A^{\frac{p(n+\rho)}{q}}_{\rho-1} \subseteq \calN(p, q, s)$.
\end{enumerate}
\end{prop}

\begin{proof}
The proof for (2) is a simple modification of (1) and  hence we omit the proof for (2) here. Suppose $0<s<1$. Note that since $\rho<\frac{q+ns-n}{1-s}$, we have
$$
(q+ns-n-\rho) \cdot \frac{n+\rho}{n+\rho-q}+\rho>0.
$$

By \cite[Corollary 5.24]{Zhu} and H\"older's inequality, for fixed $\xi \in \SSS$ and $0<r<1$, we have
\begin{eqnarray*}
I_{r, \xi}%
&=& \frac{1}{r^{ns}} \int_{Q_r(\xi)} |f(z)|^p(1-|z|^2)^{q+ns}d\lambda(z)\\
&\simeq& \frac{1}{r^{ns}} \int_{Q_r(\xi)} |f(z)|^p(1-|z|^2)^{q+ns-n-\rho} dV_{\rho-1}(z)
\\
&\le& \frac{1}{r^{ns}} \left(\int_{Q_r(\xi)} |f(z)|^{\frac{p(n+\rho)}{q}} dV_{\rho-1}(z) \right)^{\frac{q}{n+\rho}}\\
&& \cdot \left(\int_{Q_r(s)} (1-|z|^2)^{(q+ns-n-\rho) \cdot \frac{n+\rho}{n+\rho-q}} dV_{\rho-1}(z)\right)^{\frac{n+\rho-q}{n+\rho}} \\
\end{eqnarray*}
\begin{eqnarray*}
&\lesssim& \frac{\|f\|_{\frac{p(n+\rho)}{q}, \rho-1}^p}{r^{ns}} \cdot \left(\int_{Q_r(s)} (1-|z|^2)^{(q+ns-n-\rho) \cdot \frac{n+\rho}{n+\rho-q}+\rho-1} dV(z)\right)^{\frac{n+\rho-q}{n+\rho}}\\
&\simeq& \frac{\|f\|_{\frac{p(n+\rho)}{q}, \rho-1}^p}{r^{ns}} \cdot \left(r^{n+1+(q+ns-n-\rho) \cdot \frac{n+\rho}{n+\rho-q}+\rho-1}\right)^{\frac{n+\rho-q}{n+\rho}} \\
&=& \|f\|_{\frac{p(n+\rho)}{q}, \rho-1}^p,
\end{eqnarray*}
which, by Proposition \ref{Carleson01}, implies
$$
\|f\|^p \simeq \sup_{\xi \in \SSS, 0<r<1} I_{r, \xi} \lesssim\|f\|_{\frac{p(n+\rho)}{q}, \rho-1}^p.
$$
Hence, we prove the desired result.
\end{proof}

\begin{cor} \label{embeddingBercor}
Let $p \ge 1, q>0$ and $s>\max\left\{0, 1-\frac{q}{n}\right\}$. If $q>n$,  then we have $\|f\| \lesssim \|f\|_{p, q-n-1}$, that is $A^p_{q-n-1} \subseteq \calN(p, q, s)$.
\end{cor}

\begin{proof}

Note that for fixed $\xi \in \SSS$ and $0<r<1$, if $z \in Q_r(\xi)$, then we have
\begin{equation} \label{weighted22}
r>|1-\langle z,\xi \rangle| \ge 1-|\langle z, \xi \rangle| \ge 1-|z||\xi|=1-|z|.
\end{equation}

For fixed $\xi \in \SSS$ and $0<r<1$, by \eqref{weighted22}, we have
\begin{eqnarray*}
I_{r, \xi}%
&=& \frac{1}{r^{ns}} \int_{Q_r(\xi)} |f(z)|^p(1-|z|^2)^{q+ns}d\lambda(z)\\
&\simeq& \frac{1}{r^{ns}} \int_{Q_r(\xi)} |f(z)|^p (1-|z|^2)^{ns} dV_{q-n-1}(z)\\
&\le& \frac{1}{r^{ns}} \cdot \sup_{z \in Q_r(\xi)} (1-|z|^2)^{ns} \cdot \int_{Q_r(\xi)} |f(z)|^p dV_{q-n-1}(z)\\
&\lesssim& \frac{1}{r^{ns}} \cdot \sup_{z \in Q_r(\xi)} (1-|z|)^{ns} \cdot \|f\|^p_{p, q-n-1}\\
& \le& \|f\|^p_{p, q-n-1},
\end{eqnarray*}
which implies that $A^p_{q-n-1} \subseteq \calN(p, q, s)$.
\end{proof}

\begin{rem} \label{Bergremark}
In Proposition \ref{Bergmanembedding}, in particular, putting $\rho=q+ns-n$, we get $A^{\frac{p(q+ns)}{q}}_{q+ns-n-1} \subseteq \calN(p, q, s)$. Moreover, it is clear that $\| \cdot \|_{p, q+ns-n-1} \le \| \cdot\|$, that is, $\calN(p, q, s) \subset A^p_{q+ns-n-1}$. Combining this fact with Proposition \ref{Bergmanembedding}, we  have, for $p \ge 1, q>0$ and $s>\max \left\{0, 1-\frac{q}{n}\right\}$, the following embedding relation holds
$$
A^{\frac{p(q+ns)}{q}}_{q+ns-n-1} \subseteq \calN(p, q, s) \subseteq  A^p_{q+ns-n-1}.
$$
\end{rem}


\medskip

\subsubsection{Embedding relationship with weighted Hardy space}

Recall that for $t>0$, the \emph{Hardy space} $H^t$ consists of all holomorphic functions $f$ in $\B$ such that (see, e.g., \cite{Zhu})
$$
\|f\|_{H^t}=\sup_{0<r<1} M_t(r, f)<\infty.
$$
 It is well-known that when $1 \le t<\infty$, $H^t$ is a Banach space with the norm $\| \cdot \|_{H^t}$; if  $0<t<1$, $H^t$ is a complete metric space.

More generally, for $\alpha>0$ and $\beta \ge 0$, the \emph{weighted Hardy space} $H^\alpha_\beta$ is defined as follows.
$$
H^\alpha_\beta=\left\{ f \in H(\B): \|f\|_{H^{\alpha}_\beta}=\sup_{0<r<1} (1-r)^\beta M_\alpha(r, f)<\infty \right\}.
$$
It is known that when $\alpha \ge 1$, $H^\alpha_\beta$ is a Banach space with the norm $\|\cdot\|_{H^\alpha_\beta}$ (see, e.g., \cite{SU}). Moreover, the little weighted Hardy space $H^\alpha_{\beta, 0}$ is the space of all $f\in H^\alpha_\beta$ such that   $$\lim_{r \to 1 } (1-r)^{\beta} M_\alpha(r, f)=0 .$$
It is easy to see that $H^\alpha_{\beta, 0} \subseteq H^\alpha_\beta$.

As an application of Proposition \ref{Carleson01}, we have the following result.

\begin{prop} \label{Hardyembedding}
Let $p \ge 1, q>0$ and $s>\max\left\{0, 1-\frac{q}{n}\right\}$. The following statements hold.
\begin{enumerate}
\item If $0<s<1$, then $H_{\frac{q}{p}-\frac{n}{\alpha}}^\alpha \subseteq \calN(p, q, s)$, where $\max\left\{p, \frac{np}{q}\right\} \le \alpha<\frac{p}{1-s}$;
\item If $s \ge 1$, then  $H_{\frac{q}{p}-\frac{n}{\alpha}}^\alpha \subseteq \calN(p, q, s)$, where $\alpha \ge \max\left\{p, \frac{np}{q}\right\}$. In particular, if $\alpha \ge \max\left\{p, \frac{np}{q}\right\}$, then $H_{\frac{q}{p}-\frac{n}{\alpha}}^\alpha \subseteq A^{-\frac{q}{p}}$.
\end{enumerate}
\end{prop}

\begin{proof}
The proof for (2) is a simple modification of (1) and  hence we omit the proof for (2) here.

Note that for fixed $\xi \in \SSS$ and $0<r<1$, if $z \in Q_r(\xi)$, by the argument in Corollary \ref{embeddingBercor}, we have
\begin{equation} \label{weighted01}
1-r<|z|<1.
\end{equation}

We consider two different cases.

{\it Case I: $\alpha=p$.}  First we note that by condition, $p=\alpha \ge \frac{np}{q}$, which implies $q \ge n$. For each $\xi \in \SSS$ and $0<r<1$, by \eqref{weighted01}, we have

\begin{eqnarray*}
I_{r, \xi}%
&=&\frac{1}{r^{ns}} \int_{Q_r(\xi)} |f(z)|^p(1-|z|^2)^{q+ns}d\lambda(z)\\
&=& \frac{1}{r^{ns}} \int_{\{z \in \B: |1-\langle z, \xi \rangle|<r\}} |f(z)|^p(1-|z|^2)^{q+ns-n-1}dV(z)\\
&\lesssim& \frac{1}{r^{ns}} \int_{1-r}^1 (1-t^2)^{q+ns-n-1} \left(\int_{Q(\xi, r)} |f(\gamma t)|^p d\sigma(\gamma)\right) dt\\
&=& \frac{1}{r^{ns}} \int_{1-r}^1 \frac{(1-t^2)^{q+ns-n-1}}{(1-t)^{q-n}} (1-t)^{q-n} M_p^p(t,f) dt \\
&\le& \|f\|^p_{H^p_{\frac{q-n}{p}}} \cdot \frac{1}{r^{ns}} \int_{1-r}^1 (1-t)^{ns-1}dt \simeq  \|f\|^p_{H^p_{\frac{q-n}{p}}}.
\end{eqnarray*}
Thus, we have
$$
\|f\|^p=\sup_{0<r<1, \xi \in \SSS} I_{r, \xi} \lesssim  \|f\|^p_{H^p_{\frac{q-n}{p}}},
$$
which implies the desired result.

{\it Case II: $\alpha>p$.} For each $\xi \in \SSS$ and $0<r<1$, by H\"older inequality and \cite[Lemma 4.6]{Zhu}, we have
\begin{eqnarray*}
\int_{Q(\xi, r)} |f(\gamma t)|^p d\sigma(\gamma)%
&\le& \sigma(Q(\xi, r))^{1-\frac{p}{\alpha}}\left(\int_{Q_(\xi, r)} |f(\gamma t)|^\alpha d\sigma(\gamma)\right)^{\frac{p}{\alpha}} \\
&\lesssim& \frac{ r^{n-\frac{np}{\alpha}} (1-t)^{q-\frac{np}{\alpha}} M_\alpha^p(t, f)}{ (1-t)^{q-\frac{np}{\alpha}}}\\
 &\le& \frac{r^{n-\frac{np}{\alpha}} \|f\|^p_{H_{\frac{q}{p}-\frac{n}{\alpha}}^\alpha}}{{ (1-t)^{q-\frac{np}{\alpha}}}}.
\end{eqnarray*}
Then, by \eqref{weighted01} and previous calculation, we have
\begin{eqnarray*}
I_{r,\xi}%
&\lesssim& \frac{1}{r^{ns}} \int_{1-r}^1 (1-t^2)^{q+ns-n-1} \left(\int_{Q(\xi, r)} |f(\gamma t)|^p d\sigma(\gamma)\right) dt\\
&\lesssim& r^{n-\frac{np}{\alpha}-ns} \|f\|^p_{H_{\frac{q}{p}-\frac{n}{\alpha}}} \int_{1-r}^1 \frac{(1-t)^{q+ns-n-1}}{{ (1-t)^{q-\frac{np}{\alpha}}}}dt\\
&=&  r^{n-\frac{np}{\alpha}-ns} \|f\|^p_{H_{\frac{q}{p}-\frac{n}{\alpha}}} \int_{1-r}^1  (1-t)^{ns+\frac{np}{\alpha}-n-1}dt \\
&\simeq&  \|f\|^p_{H_{\frac{q}{p}-\frac{n}{\alpha}}}.
\end{eqnarray*}
Thus,
$$
\|f\|^p \simeq \sup_{0<r<1, \xi \in \SSS} I_{r, \xi} \lesssim \|f\|^p_{H_{\frac{q}{p}-\frac{n}{\alpha}}},
$$
and hence the desired result follows.
\end{proof}

The following result describes the behavior of Hadamard gap series in $H^\alpha_\beta$, whose idea comes from \cite[Theorem 1]{LS}.

\begin{prop} \label{weightedHadamard}
Let $\alpha>0$, $\beta>0$ and $f(z)=\sum_{k=1}^\infty P_{n_k}(z)$ with Hadamard gaps. Then the following statements hold true:
\begin{enumerate}
\item $f \in H^\alpha_\beta$ if and only if $\sup_{k \ge 1} \frac{L_{k, \alpha}}{n_k^\beta}<\infty$.
\item $f \in H^\alpha_{\beta, 0}$ if and only if $\lim_{k \to \infty}  \frac{L_{k, \alpha}}{n_k^\beta}=0$.
\end{enumerate}
Here $L_{k, \alpha}= \left( \int_\SSS |P_{n_k}(\xi)|^{\alpha}d\sigma(\xi)\right)^{\frac{1}{\alpha}}$.
\end{prop}

\begin{proof} (1) {\it Necessity.}  Let $f \in H^\alpha_\beta$. By \cite[Proposition 1.4.7]{Rud} and \cite[Lemma 1]{LS}, we have, for each $k \in \N$,
\begin{eqnarray} \label{weightedeq01}
\int_\SSS |f(r\xi)|^\alpha d\sigma(\xi)%
&=& \int_\SSS \left(\int_0^{2\pi} |f(r\xi e^{i\theta})|^\alpha \frac{d\theta}{2\pi} \right)d\sigma(\xi) \nonumber\\
&=& \int_\SSS \left( \int_0^{2\pi} \left| \sum_{k=0}^\infty P_{n_k} (r\xi e^{i\theta}) \right|^\alpha \frac{d\theta}{2\pi} \right)d\sigma(\xi) \nonumber\\
& \simeq&  \int_\SSS \left( \sum_{k=0}^\infty |P_{n_k}(\xi)|^2r^{2n_k}\right)^{\alpha/2} d\sigma(\xi)  \\
&\ge& r^{\alpha n_k} \int_\SSS |P_{n_k}(\xi)|^\alpha d\sigma (\xi) \nonumber.
\end{eqnarray}
Hence,
\begin{equation} \label{weighted05}
(1-r)^\beta r^{n_k} L_{k, \alpha} \le (1-r)^\beta M_\alpha(r, f) \le  \|f\|_{H^\alpha_\beta}.
\end{equation}
Choosing $r=1-\frac{1}{n_k}$ and using the well-known inequality $(1+\frac{1}{m})^{m+1} \le 4, m \in \N$, we obtain
$$
\sup_{k \in \N} \frac{L_{k, \alpha}}{n_k^\beta} \le C\|f\|_{H^\alpha_\beta},
$$
as desired.

  {\it Sufficiency.}  Suppose that $\sup_{k \in \N} \frac{L_{k, \alpha}}{n_k^\beta}<\infty$. For a fixed $r \in (0, 1)$, we have
$$
\frac{\sum\limits_{k=0}^\infty r^{\alpha n_k} n_k^{\alpha \beta}}{1-r^\alpha}=\left( \sum_{k=0}^\infty r^{\alpha n_k} n_k^{\alpha \beta}\right) \cdot \left( \sum_{s=0}^\infty r^{\alpha s} \right) =\sum_{t=0}^\infty \left( \sum_{n_j \le t} n_j^{\alpha \beta} \right) r^{\alpha t}.
$$
Since,
$$
\lim_{k \to \infty} \frac{k^{\alpha \beta} k!}{(\alpha \beta) (\alpha \beta+1) \dots (\alpha \beta+k)}=\Gamma (\alpha \beta),  \alpha \beta>0,
$$
we have
$$
\sup_{k \in \N} \left(\frac{k^{\alpha \beta} k!}{(k+\alpha \beta)(k+\alpha \beta-1) \dots (\alpha \beta+1)}\right) \leq M,
$$
where $M$ is some positive number depending on $\alpha$ and $\beta$. Hence, for each $k \ge 0$,
\begin{eqnarray} \label{ineq009}
\frac{k^{\alpha \beta}}{(-1)^k {-\alpha \beta-1 \choose k}}%
&=&\frac{k^{\alpha \beta} k!}{(-1)^k (-\alpha \beta-1)(-\alpha \beta-2) \dots (-\alpha \beta-k)} \nonumber\\
&=&\frac{k^{\alpha \beta} k!}{(k+\alpha \beta)(k+\alpha \beta-1) \dots (\alpha \beta+1)} \leq M,
\end{eqnarray}
where ${\gamma \choose k}=\frac{\gamma(\gamma-1)\dots (\gamma-k+1)}{k!}, \gamma \in \R$.

Moreover, since $f$ is in Hadamard gaps class, there exists a constant $c>1$ such that $n_{j+1} \geq cn_j$ for all $j \ge 0$. Hence
\begin{equation} \label{ineq0091}
 \frac{1}{t^{\alpha \beta}} \left(\sum_{n_j \leq t} n^{\alpha \beta}_j \right) \leq \sum_{m=0}^\infty \left(\frac{1}{c^{\alpha \beta}}\right)^m =\frac{c^{\alpha \beta}}{c^{\alpha \beta}-1}.
\end{equation}
Combining \eqref{ineq009} and \eqref{ineq0091}, we have
$$
\frac{t^{\alpha \beta}}{(-1)^t {-{\alpha \beta}-1 \choose t}} \cdot \frac{1}{t^{\alpha \beta}} \left(\sum_{n_j \leq t} n_j^{\alpha \beta}\right)  \leq \frac{Mc^{\alpha \beta}}{c^{\alpha \beta}-1},
$$
which implies
\begin{equation} \label{ineq010}
\sum_{n_j \leq t} n_j^{\alpha \beta} \le (-1)^t {-{\alpha \beta}-1 \choose t} \frac{Mc^{\alpha \beta}}{c^{\alpha \beta}-1}.
\end{equation}

Hence, by \eqref{ineq010}, we have
$$
\frac{\sum\limits_{k=0}^\infty r^{\alpha n_k} n_k^{\alpha \beta}}{1-r^\alpha} \lesssim \sum_{r=0}^{\infty}(-1)^t {-{\alpha \beta}-1 \choose t} r^{\alpha t}=\frac{1}{(1-r^{\alpha})^{\alpha \beta+1}},
$$
which implies for $\alpha, \beta>0$,
\begin{equation} \label{weighted02}
(1-r^\alpha)^{\alpha \beta} \sum_{k=0}^\infty r^{\alpha n_k} n_k^{\alpha \beta} \lesssim 1.
\end{equation}

Now we consider two different cases.

{\it Case I: $\alpha \in (0, 2]$}. From \eqref{weightedeq01}, \eqref{weighted02} and by using the known inequality
$$
\left(\sum\limits_{k=1}^\infty a_k \right)^q \le \sum_{k=1}^\infty a_k^q,
$$
where $a_k \ge 0, k \in \N, q \in [0, 1]$, we have that
\begin{eqnarray*}
\|f\|_{H^\alpha_\beta}^{\alpha}%
&\simeq& \sup_{0<r<1}  (1-r)^{\alpha \beta} \int_\SSS  \left( \sum_{k=0}^\infty |P_{n_k}(\xi)|^2r^{2n_k}\right)^{\alpha/2} d\sigma(\xi)  \\
&\le& \sup_{0<r<1}  (1-r)^{\alpha \beta} \int_\SSS \left(\sum_{k=0}^\infty |P_{n_k}(\xi)|^\alpha r^{\alpha n_k} \right) d\sigma(\xi)   \\
&=& \sup_{0<r<1}   (1-r)^{\alpha \beta} \cdot \sum_{k=0}^\infty r^{ \alpha n_k} L_{k, \alpha}^\alpha   \\
&\lesssim& \sup_{0<r<1}   (1-r)^{\alpha \beta} \cdot \sum_{k=0}^\infty r^{\alpha n_k} n_k^{\alpha \beta}  \\
&\lesssim& \sup_{0<r<1} \left(\frac{1-r}{1-r^\alpha}\right)^{\alpha \beta}<\infty,
\end{eqnarray*}
which implies the desired result.

{\it Case II: $\alpha>2$.} For each $r \in (0, 1)$, by Minkowski's inequality and \eqref{weighted02}, we have
\begin{eqnarray} \label{weighted10}
\left[\int_\SSS  \left( \sum_{k=0}^\infty |P_{n_k}(\xi)|^2r^{2n_k}\right)^{\alpha/2} d\sigma(\xi)\right]^{\frac{2}{\alpha}}
&\le& \sum_{k=0}^\infty \left( \int_\SSS \left(|P_{n_k}(\xi)|^2 r^{2n_k}\right )^{\frac{\alpha}{2}}d\sigma(\xi) \right)^{\frac{2}{\alpha}} \nonumber \\
&=& \sum_{k=0}^\infty r^{2n_k} L_{k, \alpha}^2 \lesssim \sum_{k=0}^\infty r^{2n_k} n_k^{2\beta}.
\end{eqnarray}
Thus, by \eqref{weightedeq01} and \eqref{weighted02},
\begin{eqnarray} \label{weighted11}
\|f\|_{H^\alpha_\beta}^\alpha%
&\simeq&  \sup_{0<r<1}  (1-r)^{\alpha \beta} \int_\SSS  \left( \sum_{k=0}^\infty |P_{n_k}(\xi)|^2r^{2n_k}\right)^{\alpha/2} d\sigma(\xi)
\nonumber\\
& \le& \sup_{0<r<1}   (1-r)^{\alpha \beta} \left(\sum_{k=0}^\infty r^{2n_k} L_{k, \alpha}^2 \right)^{\frac{\alpha}{2}} \\
&\lesssim& \sup_{0<r<1}  (1-r)^{\alpha \beta} \left(\sum_{k=0}^\infty r^{2n_k} n_k^{2\beta} \right)^{\frac{\alpha}{2}}   \nonumber\\
&\lesssim&  \sup_{0<r<1} \left(\frac{1-r}{1-r^2}\right)^{\alpha \beta}<\infty \nonumber,
\end{eqnarray}
and hence $f \in H_\beta^\alpha$.

\medskip

(2)  {\it Necessity.}  Let $f \in H^\alpha_{\beta, 0}$. Then for every $\varepsilon>0$, there is a $\del>0$, such that
\begin{equation} \label{weighted06}
(1-r)^\beta M_\alpha(r, f)<\varepsilon
\end{equation}
whenever $\del<r<1$. From the first inequality in \eqref{weighted05} and \eqref{weighted06}, we have
$$
(1-r)^\beta r^{n_k} L_{k, \alpha}<\varepsilon
$$
for each $k \in \N$ and $r \in (\del, 1)$. Choosing $r=1-\frac{1}{n_k}$, where $n_k>\frac{1}{1-\del}$, we obtain
$$
\frac{L_{k, \alpha}}{n_k^\beta}<4\varepsilon.
$$
From this and since $\varepsilon$ is an arbitrary positive number, it follows that
$$
\lim_{k \to \infty}\frac{L_{k, \alpha}}{n_k^\beta}=0.
$$

{\it  Sufficiency. } Suppose $\lim_{k \to \infty} \frac{L_{k, \alpha}}{n_k^\beta}=0$. Take and fix a $\varepsilon>0$, there is a $k_0 \in \N$ such that
$$
L_{k, \alpha}<\varepsilon^{\frac{1}{\alpha}} n_k^\beta, \quad \textrm{for} \ k \ge k_0.
$$
Fix the $k_0$ chosen above. Then, there exists a $\del>0$, such that when $\del<r<1$,
$$
(1-r)^{\alpha \beta} \sum_{k=0}^{k_0} L_{k, \alpha}^\alpha<\varepsilon.
$$

Again, we consider two cases as the proof in part (1).

{\it Case I: $\alpha \in (0, 2]$.} By \eqref{weighted05} and the proof in part (1), for $r \in (\del, 1)$, we have
\begin{eqnarray*}
(1-r)^{\alpha \beta} M^\alpha_\alpha(r, f)%
&\simeq& (1-r)^{\alpha \beta} \int_\SSS \left(\sum_{k=0}^\infty |P_{n_k}(\xi)|^2 r^{2n_k}\right)^{\alpha/2} d\sigma(\xi)\\
&\le& (1-r)^{\alpha \beta} \cdot \left[ \left(\sum_{k=0}^{k_0}+\sum_{k=k_0+1}^\infty\right) r^{\alpha n_k} L_{k, \alpha}^\alpha\right]\\
&\le& \varepsilon+(1-r)^{\alpha \beta} \cdot \sum_{k=k_0+1}^\infty r^{\alpha n_k} L_{k, \alpha}^\alpha \\
&\le& \varepsilon+\varepsilon (1-r)^{\alpha \beta} \cdot \sum_{k=k_0+1}^\infty r^{\alpha n_k} n_k^{\alpha \beta}\\
&\lesssim& \varepsilon \cdot \left(1+\sup_{0<r<1}  \left(\frac{1-r}{1-r^2}\right)^{\alpha \beta} \right) \lesssim \varepsilon,
\end{eqnarray*}
whichi implies
$$
\lim_{r \to 1^{-}} (1-r)^{\beta} M_\alpha(r, f)=0
$$
and hence $f \in H^\alpha_{\beta, 0}$ as desired.

{\it Case II:  $\alpha>2$.} The implication for case $p \ge 2$ follows similarly, from \eqref{weighted10}, \eqref{weighted11}
and the known inequality
$$
(a+b)^{\frac{\alpha}{2}} \le 2^{\frac{\alpha}{2}-1} \left(a^{\frac{\alpha}{2}}+b^{\frac{\alpha}{2}}\right), \quad a, b \ge 0.
$$
Hence, we omit the detail here.
\end{proof}

As a corollary of Proposition \ref{weightedHadamard}, we can show that when $\frac{q}{p}-\frac{n}{\alpha}>0$, the inclusion
in Proposition \ref{Hardyembedding} is strict.

\begin{cor}
Let $p, q, s$ and $\alpha$ satisfy the condition in Proposition \ref{Hardyembedding}. If
$$
\frac{q}{p}-\frac{n}{\alpha}>0,
$$
then the inclusion in Proposition \ref{Hardyembedding} is strict.
\end{cor}

\begin{proof}
First we take a sequence of homogeneous polynomials $\{W_k\}_{k \in \N}$ satisfying $\deg(W_k)=k$,
$$
\sup_{\xi \in \SSS} |W_k(\xi)|=1, \quad \textrm{and} \quad \int_\SSS |W_k(\xi)|^pd\sigma(\xi) \ge C(p, n)>0.
$$
Note that since $p \le \alpha$, by H\"older's inequality, we have
$$
\int_\SSS |W_k(\xi)|^pd\sigma(\xi) \le \left( \int_\SSS |W_k(\xi)|^\alpha d\sigma(\xi) \right)^{\frac{p}{\alpha}},
$$
that is
\begin{equation} \label{weighted20}
\int_\SSS |W_k(\xi)|^\alpha d\sigma(\xi) \ge \left(\int_\SSS |W_k(\xi)|^pd\sigma(\xi) \right)^{\frac{\alpha}{p}} \ge C^{\frac{\alpha}{p}}(p, n).
\end{equation}

Note that when $s>1$, $\calN(p, q, s)=A^{-\frac{q}{p}}$, and hence we consider three cases.

{\it Case I: $s>1$.} Consider the series $f_1(z)=\sum\limits_{k=0}^\infty 2^{\frac{kq}{p}}W_{2^k}(z)$, which by \cite[Theorem 2.5]{HL}, belongs to $A^{-\frac{q}{p}}$. On the other hand, for each $k \in \N$, by \eqref{weighted20},
$$
\frac{L_{k, \alpha}}{2^{k\left(\frac{q}{p}-\frac{n}{\alpha} \right)}}=\frac{\left(\int_\SSS \left|2^{\frac{kq}{p}}W_{2^k}(z)\right|^\alpha d\sigma(\xi) \right)^{\frac{1}{\alpha}}}{2^{k\left(\frac{q}{p}-\frac{n}{\alpha} \right)}} \ge C^{\frac{1}{p}}(p, n) 2^{\frac{kn}{\alpha}}.
$$
Hence, we have
$$
\sup_{k \ge 0} \frac{L_{k, \alpha}}{2^{k\left(\frac{q}{p}-\frac{n}{\alpha} \right)}}=\infty,
$$
which implies that $f_1 \notin H^\alpha_{\frac{q}{p}-\frac{n}{\alpha}}$.

{\it Case II: $s=1$.} Take any $\varepsilon \in \left(0, \frac{np}{\alpha}\right)$ and consider the series $f_2(z)=\sum\limits_{k=0}^\infty 2^{\frac{k(q-\varepsilon)}{p}} W_{2^k}(z)$. On one hand,  we have
$$
\sum_{k=0}^\infty \frac{1}{2^{kq}} \cdot \left|2^{\frac{k(q-\varepsilon)}{p}}\right|^p=\sum_{k=0}^\infty \frac{1}{2^{k\varepsilon}}<\infty,
$$
which by  Corollary \ref{RWsequence}, implies that $f_2 \in \calN(p, q, 1)$. However, for each $k \in \N$, by \eqref{weighted20},
$$
\frac{L_{k, \alpha}}{2^{k\left(\frac{q}{p}-\frac{n}{\alpha} \right)}}=\frac{\left(\int_\SSS \left| 2^{\frac{k(q-\varepsilon)}{p}} W_{2^k}(z)\right|^\alpha d\sigma(\xi) \right)^{\frac{1}{\alpha}}}{2^{k\left(\frac{q}{p}-\frac{n}{\alpha} \right)}} \ge C^{\frac{1}{p}}(p, n) 2^{k \left(\frac{n}{\alpha}-\frac{\varepsilon}{p}\right)}.
$$
Hence, we have
$$
\sup_{k \ge 0} \frac{L_{k, \alpha}}{2^{k\left(\frac{q}{p}-\frac{n}{\alpha} \right)}}=\infty,
$$
which implies that $f_2 \notin H^\alpha_{\frac{q}{p}-\frac{n}{\alpha}}$.

{\it Case III: $0<s<1$.} Since $\alpha<\frac{p}{1-s}$, we have $\frac{s-1}{p}+\frac{1}{\alpha}>0$.  Take any $\varepsilon \in \left(0,  n(s-1)+\frac{np}{\alpha} \right)$ and consider the series $f_3(z)=\sum\limits_{k=0}^\infty 2^{\frac{k(ns+q-n-\varepsilon)}{p}} W_{2^k}(z)$. On one hand, we have
$$
\sum_{k=0}^\infty \frac{1}{2^{k(ns+q-n)}} \left|2^{\frac{k(ns+q-n-\varepsilon)}{p}}\right|^p=
\sum_{k=0}^\infty \frac{1}{2^{k\varepsilon}}<\infty,
$$
which by  Corollary \ref{RWsequence}, implies that $f_3 \in \calN(p, q, s)$. However, for each $k \in \N$, by \eqref{weighted20},

$$
\frac{L_{k, \alpha}}{2^{k\left(\frac{q}{p}-\frac{n}{\alpha} \right)}}=\frac{\left(\int_\SSS \left|2^{\frac{k(ns+q-n-\varepsilon)}{p}} W_{2^k}(z)\right|^\alpha d\sigma(\xi) \right)^{\frac{1}{\alpha}}}{2^{k\left(\frac{q}{p}-\frac{n}{\alpha} \right)}} \ge C^{\frac{1}{p}}(p, n) 2^{k \left(\frac{n(s-1)}{p}+\frac{n}{\alpha}-\frac{\varepsilon}{p}\right)}.
$$
Hence, we have
$$
\sup_{k \ge 0} \frac{L_{k, \alpha}}{2^{k\left(\frac{q}{p}-\frac{n}{\alpha} \right)}}=\infty,
$$
which implies that $f_3 \notin H^\alpha_{\frac{q}{p}-\frac{n}{\alpha}}$.
\end{proof}


\medskip

\subsection{Hadamard products}


An important application of those Carleson property of $\calN(p, q, s)$-type spaces is to study Hadamard product in them.

We first set up some basic notations which shall be used in this subsection. For $z=(z_1, \dots, z_n) \in \C^n, \eta=(\eta_1, \dots, \eta_n) \in \Z^n_+$, $\xi=(\xi_1, \dots, \xi_n) \in \C^n$, we let
\begin{eqnarray*}
&& \overline{z}=(\overline{z_1}, \dots, \overline{z_n}), \quad z^\eta=z_1^{\eta_1} \cdots z_n^{\eta_n}, \\
&& |\eta|=\eta_1+\dots+\eta_n, \quad \eta!=\eta_1! \cdot \eta_n!, \\
&& \partial_j=\frac{\partial}{\partial_j},  \ 1 \le j \le n, \quad \partial^{\eta}=\partial_1^{\eta_1} \dots \partial_n^{\alpha_n}.
 \end{eqnarray*}
The Bloch space on the unit ball, denoted by $\calB$, is defined as the space of $H(\B)$ for which
$$
\|f\|_{\calB}=|f(0)|+\sup_{z \in \B} (1-|z|^2)|\nabla f(z)|<\infty.
$$

Since $\B$ is a complete Rheinhardt domain in $\C^n$, i.e. $z \in \B$ implies $z \cdot \xi=(z_1\xi_1, \dots, z_n\xi_n) \in \B$ for every $\xi \in \overline{U}$, where $U$ is the unit polydisc in $\C^n$. Then any $f \in H(\B)$ has a unique power series
$$
f(z)=\sum_\eta a_\eta z^\eta, \quad  z \in \B,
$$
with $a_\eta=\frac{\partial^\eta f(0)}{\eta !}, \eta \in \Z^n_+$. So $H(\B)$ may be regarded as a space of multi-index sequence $\{a_\eta\}$.

For $f(z)=\sum_\eta a_\eta z^\eta, g(z)=\sum_\eta b_\eta z^\eta \in H(\B)$ and $d>0$, the $d$-Hadamard product of $f$ and $g$ is defined as follows
 (see, e.g., \cite{BL, LW}).
$$
(f*g)_d(z)=\sum_\eta \omega_\eta(d) a_\eta b_\eta z^\eta,
$$
where
$$
\omega_\eta (d)=\frac{\eta! \Gamma(n+d)}{\Gamma(n+d+|\eta|)}, \quad \eta \in \Z^n_+.
$$

Let us denote $H(U)$ the collection of all holomorphic functions in $U$ and $H^\infty(U)$ the Banach space consisting of all bounded holomorphic functions in $U$. The interesting feature of this product is that it is lying not only in $H(\B)$ but also in $H(U)$ and for $f, g \in H(\B)$ and any $0 \le r<1$,
\begin{equation} \label{Hintegral}
(f*g)_d(rz)=\langle f_r, g^*_{\overline{z}} \rangle_d:=\int_\B f(rw) g(z \cdot \overline{w})dV_{d-1}(w),
\end{equation}
for any $z \in U$, where
$$
f_r(z)=f(rz), \quad  g^*(z)=\overline{g(\overline{z})} \quad \textrm{and} \quad f_\xi(z)=f(\xi \cdot z), \ \xi \in \B,
$$
(see, e.g., \cite[Proposition 3.2]{BL}).

The following lemma proved in \cite[Proposition 3.3]{BL} plays an important role in this subsection.

\begin{lem}  \label{Hproduct}
Let $d>0$, $1 \le p_1, p_2, p_3 \le \infty$ with $1+\frac{1}{p_3}=\frac{1}{p_1}+\frac{1}{p_2}$, and let $f, g \in H(\B)$. Then $(f*g)_d \in H(U)$ with $$\|(f*g)_d\|_{p_3, d-1} \le \|f\|_{p_1, d-1} \|g\|_{p_2, d-1}.$$
 Moreover, if also $p_3=\infty$, i.e. $p_2=p'_1=p_1/(p_1-1)$, then $\|(f*g)_d\|_\infty \le \|f\|_{p_1, {d-1}} \|g\|_{p'_1, {d-1}}$. In particular, $(f*g)_d \in H^\infty (U)$ if $f \in A^{p_1}_{d-1}$ and $g \in A^{p'_1}_{d-1}$.
\end{lem}

\begin{prop}
Let $p, r \ge 1, q, d>0$ and $s>\max\left\{0, 1-\frac{q}{n}\right\}$. Then
\begin{enumerate}
\item $(f*g)_d \in \calN(p, q, s)$ if $f \in A^r_{d-1}$ and $g \in A^{r'}_{d-1}$, where $r>1$ and $r'=\frac{r}{r-1}$;
\item $(f*g)_d \in \calN(p, q, s)$ if $f \in A^1_{d-1}$ and $g \in \calB$.
\end{enumerate}
\end{prop}

\begin{proof} (1) Since $s>1-\frac{q}{n}$, by Proposition \ref{prop02} and Lemma \ref{Hproduct}, we have
$$
\|(f*g)_d\| \lesssim \|(f*g)_d\|_\infty \le \|f\|_{r, d-1} \|g\|_{r', d-1}<\infty,
$$
which implies the desired result.

(2) By the proof of \cite[Theorem 1]{LW}, we have
$$
|(f*g)_d(z)| \lesssim \|g\|_\calB \|f\|_{1, d-1},
$$
which, again, by the condition $s>1-\frac{q}{n}$, implies
$$
\|(f*g)_d\| \lesssim \|(f*g)_d\|_\infty \lesssim \|g\|_\calB \|f\|_{1, d-1}.
$$
Hence the desired result follows.
\end{proof}

By using the embedding relation between $\calN(p, q, s)$ and $A^k_\rho$, we have the following result.

\begin{prop} \label{NpqsHadamard}
Let $p, q, s, \rho$ satisfy the conditions in Proposition \ref{Bergmanembedding}. Then for $r_1, r_2 \ge 1$ with satisfying
$$
1+\frac{q}{p(n+\rho)}=\frac{1}{r_1}+\frac{1}{r_2},
$$
we have $(f*g)_\rho \in \calN(p, q, s)$ if $f \in A^{r_1}_{\rho-1}$ and $g \in A^{r_2}_{\rho-1}$.
\end{prop}

\begin{proof}
By Proposition \ref{Bergmanembedding} and Lemma \ref {Hproduct}, we have
$$
\|(f*g)_\rho\| \lesssim \|(f*g)_\rho\|_{\frac{p(n+\rho)}{q}, (\rho-1)} \le \|f\|_{r_1, (\rho-1)} \|g\|_{r_2, (\rho-1)},
$$
which implies the desired result.
\end{proof}

\begin{cor}
Let $p \ge 1, q>0$ and $s>\max\left\{0, 1-\frac{q}{n}\right\}$. Then $(f*g)_{q+ns-n} \in \calN(p, q, s)$ if $f \in \calN(p, q, s)$ and $g \in A_{q+ns-n-1}^{\frac{p(q+ns)}{p(q+ns)-ns}}$.
\end{cor}

\begin{proof}
By Proposition \ref{NpqsHadamard}, we have
\begin{eqnarray*}
\|(f*g)_{(q+ns-n)}\|%
&\lesssim& \|f\|_{p, (q+ns-n-1)} \|g\|_{\frac{p(q+ns)}{p(q+ns)-ns}, q+ns-n-1}\\
&\lesssim& \|f\| \|g\|_{\frac{p(q+ns)}{p(q+ns)-ns}, q+ns-n-1}.
\end{eqnarray*}
Hence, we get the desired result.
\end{proof}

The following lemma gives an estimation of the term $M_\alpha(r, (f*g)_d)$, which gives us another description of Hadamard products via the embedding relation between $\calN(p, q, s)$ and $H^\alpha_\beta$.

\begin{lem} \label{Malphaesti}
Let $\alpha \ge 1, r \in [0, 1),  d>0$ and $f, g \in H(\B)$. Then
$$
M_\alpha(r, (f*g)_d) \le  M_\alpha(\sqrt{r}, g)\|f\|_{1, d-1}.
$$
\end{lem}

\begin{proof}
Without the loss of generality, we assume that $\|f\|_{1, d-1}<\infty$. Thus, by \eqref{Hintegral}, \cite[Proposition 3.1, (i)]{BL} and the fact that the integral mean of subharmonic function over sphere is an increasing function of multi-radius, we have
\begin{eqnarray*}
M_\alpha(r, (f*g)_d)%
&=& \left( \int_\SSS |(f*g)_d(r\xi)|^\alpha d\sigma(\xi) \right)^{\frac{1}{\alpha}}\\
&=& \left( \int_\SSS |(f*g)_d(\sqrt{r} \cdot \sqrt{r}\xi)|^\alpha d\sigma(\xi) \right)^{\frac{1}{\alpha}}\\
&=& \left( \int_\SSS \left| \int_\B f(\sqrt{r} w) g(\sqrt{r}\xi \cdot \overline{w})dV_{d-1}(w) \right|^\alpha d\sigma(\xi) \right)^{\frac{1}{\alpha}}\\
&\le&  \int_\B  |f(\sqrt{r} w)|\left( \int_\SSS | g(\sqrt{r}\xi \cdot \overline{w})|^\alpha d\sigma(\xi) \right)^{\frac{1}{\alpha}}dV_{d-1}(w)\\
&& (\textrm{By Minkowski's inequality})\\
&\le& \int_\B  |f(\sqrt{r} w)|\left( \int_\SSS | g(\sqrt{r}\xi)|^\alpha d\sigma(\xi) \right)^{\frac{1}{\alpha}}dV_{d-1}(w)\\
&& (\textrm{Since} \ |g(z)|^\alpha \ \textrm{is subharmonic})\\
&\le&  M_\alpha(\sqrt{r}, g) \|f\|_{1, d-1}.
\end{eqnarray*}
\end{proof}

\begin{prop}
Let $d>0$ and $p, q, s, \alpha$ satisfy the conditions in Proposition \ref{Hardyembedding}. Then we have $(f*g)_d \in \calN(p, q, s)$ if $f \in A^1_{d-1}$ and $g \in H^\alpha_{\frac{q}{p}-\frac{n}{\alpha}}$.
\end{prop}

\begin{proof}
For any $f \in A^1_{d-1}$ and $g \in H^\alpha_{\frac{q}{p}-\frac{n}{\alpha}}$, by Proposition  \ref{Hardyembedding} and Lemma \ref{Malphaesti}, we have
\begin{eqnarray*}
\|(f*g)_d\|%
&\lesssim& \sup_{0<r<1} (1-r)^{\frac{q}{p}-\frac{n}{\alpha}} M_\alpha(r, (f*g)_d) \\
&\le& \|f\|_{1, d-1} \cdot \sup_{0<r<1} (1-r)^{\frac{q}{p}-\frac{n}{\alpha}} M_\alpha(\sqrt{r}, g) \\
&\lesssim& \|f\|_{1, d-1} \cdot \sup_{0<r<1} (1-\sqrt{r})^{\frac{q}{p}-\frac{n}{\alpha}} M_\alpha(\sqrt{r}, g) \\
&=& \|f\|_{1, d-1} \|g\|_{H^\alpha_{\frac{q}{p}-\frac{n}{\alpha}}},
\end{eqnarray*}
which implies $(f*g)_d \in \calN(p, q, s)$.
\end{proof}


\subsection{Random power series}


A second application of the above characterization of $\calN(p, q, s)$-type spaces by Carleson measure is to study the random power series. The behavior of the random power series on $Q_s$ spaces was studied in \cite{LO2}.

Let $\{\varepsilon_\alpha(w)\}$ be a Bernoulli sequence of random variables on a probability space $(\Omega, \mathcal A, P)$. In particular,  this sequence is independent, and each $\varepsilon_\alpha(w)$ takes the values $1$ and $-1$ with probability $\frac{1}{2}$ each. A well-known example of such a Bernoulli sequence is the Rademacher functions, which are defined as
$$
\{r_j(t)\}_{j \in \N}=\{\textrm{sgn}(\sin(2^j \pi t))\}_{j \in \N}, \quad t \in [0,1].
$$
It is easy to check that the $r_j$'s are mutually independent random variables on $[0, 1]$. We refer the readers to the excellent book \cite{LG} for some details for Rademacher functions. For $f$ a holomorphic function in $\B$ with Taylor expansion $f(z)=\sum_\alpha a_\alpha z^\alpha$, the randomization of $f$ is defined as
$$
f_\omega(z)=\sum_\alpha \varepsilon_\alpha(\omega) a_\alpha z^\alpha.
$$
The following result gives a sufficient condition for $f_\omega$ belonging to the $\calN(p, q, s)$-type spaces.

\begin{prop}
Let $p \ge 1, q>0, s>\max\left\{0, 1-\frac{q}{n} \right\}$, $\rho$ be some positive number satisfying
\[ \begin{cases}
\max\{0, q-n\}<\rho<\frac{q+ns-n}{1-s}, & 0<s<1;\\
\rho>\max\{0, q-n\}, & s \ge 1
\end{cases} \]
and $k=\frac{p(n+\rho)}{q}$. Let further, $f(z)=\sum\limits_\alpha a_\alpha z^\alpha \in H(\B)$. If
$$
\{|a_\alpha| \omega_{\alpha, k} \}_\alpha \in \ell^{\min\{2, k\}},
$$
then $f_\omega \in \calN(p, q, s)$ for almost every $\omega \in \Omega$, where for each multi-index $\alpha$,
$$
\omega_{\alpha, k}=\left(\int_\SSS |\xi^\alpha|^k d\sigma(\xi) \right)^{\frac{1}{k}}.
$$
\end{prop}

\begin{proof}
 By Proposition \ref{Bergmanembedding}, it suffices to show that $f_w \in A^k_{\rho-1}$ for almost every $\omega \in \Omega$, that is, it suffices to show that
\begin{equation} \label{randompower001}
P\left\{ \omega \in \Omega: \int_0^1 M_k^k(r, f_\omega) (1-r)^{\rho-1}dr<\infty \right\}=1.
\end{equation}

Indeed, by Fubini's theorem, we have
\begin{eqnarray*}
&& E\left( \int_0^1 M_k^k(r, f_\omega) (1-r)^{\rho-1}dr \right)\\ \nonumber
&=&\int_0^1  (1-r)^{\rho-1} \left[\int_\Omega \int_\SSS |f_\omega(r\xi)|^kd\sigma(\xi)dP\right]dr\\ \nonumber
&=&\int_0^1 (1-r)^{\rho-1} \left[\int_\SSS \int_\Omega  |f_\omega(r\xi)|^kdPd\sigma(\xi)\right]dr\\ \nonumber
&=&\int_0^1 (1-r)^{\rho-1} \left[\int_\SSS \int_\Omega  \left| \sum_\alpha \varepsilon_\alpha(\omega) a_\alpha r^{|\alpha|} \xi^\alpha    \right|^kdPd\sigma(\xi)\right]dr\\ \nonumber
&\lesssim& \int_0^1 (1-r)^{\rho-1} \left[ \int_\SSS \left( \sum_\alpha |a_\alpha|^2r^{2|\alpha|} |\xi^{\alpha}|^2 \right)^{\frac{k}{2}}d\sigma(\xi) \right]dr, \nonumber
\end{eqnarray*}
where the last inequality follows from Khintchine's inequality.\\

\textit{Case I: $0<k \le2$.} Since $\frac{k}{2} \le 1$, we have
\begin{eqnarray*}
&& E\left( \int_0^1 M_k^k(r, f_\omega) (1-r)^{\rho-1}dr \right)\\
&\lesssim & \int_0^1 (1-r)^{\rho-1} \left[ \int_\SSS \left( \sum_\alpha |a_\alpha|^k r^{k|\alpha|} |\xi^\alpha|^k \right) d\sigma \right] dr\\
&\lesssim& \sum_\alpha |a_\alpha|^k \omega_{\alpha, k}^k\\
&<&\infty.
\end{eqnarray*}

\textit{Case II: $k>2$.}  By Minkowski's inequality, we have
\begin{eqnarray*}
 && E\left( \int_0^1 M_k^k(r, f_\omega) (1-r)^{\rho-1}dr \right)\\
&\lesssim& \int_0^1 (1-r)^{\rho-1} \left[ \sum_\alpha \left( \int_\SSS |a_\alpha|^kr^{k|\alpha|} |\xi^\alpha|^kd\sigma(\xi) \right)^{\frac{2}{k}} \right]^{\frac{k}{2}}dr\\
&= & \int_0^1 (1-r)^{\rho-1} \left( \sum_\alpha |a_\alpha|^2r^{2|\alpha|} \omega_{\alpha, k}^2 \right)^{\frac{k}{2}} dr\\
&\le& \left( \sum_\alpha |a_\alpha|^2 \omega_{\alpha, k}^2 \right)^{\frac{k}{2}}<\infty.
\end{eqnarray*}

Thus, for both cases, we have
$$
E\left( \int_0^1 M_k^k(r, f_\omega) (1-r)^{\rho-1}dr \right)<\infty,
$$
which implies \eqref{randompower001} clearly.
\end{proof}

Conversely, we have the following result.

\begin{prop}
Let $p \ge 1, q>0$ and $s>\max\left\{0, 1-\frac{q}{n} \right\}$. Let further, $f(z)=\sum\limits_\alpha a_\alpha z^\alpha \in H(\B)$. If there exists some $a_0, \varepsilon>0$ such that for any $a>a_0$, the following decay condition
$$
P\{\omega: \|f_\omega\|>a\} \lesssim a^{-1-\varepsilon}
$$
holds, then
$$
\left\{|a_\alpha| w_{\alpha, p} B^{\frac{1}{p}}(2n+p|\alpha|, q+ns-n)\right\}_\alpha \in \ell^\infty.
$$
Here $B(\cdot, \cdot)$ is the Beta function.
\end{prop}

\begin{proof}
By Remark \ref{Bergremark}, it is known that $\calN(p, q, s) \subseteq A^p_{q+ns-n-1}$ and hence $f_w \in A^p_{q+ns-n-1}$ for almost every $\omega \in \Omega$. Thus, for any $a>a_0$, we have
$$
P\left\{w: \int_0^1 r^{2n-1} (1-r)^{q+ns-n-1} M_p^p(r, f_\omega) dr>a \right\} \lesssim a^{-1-\varepsilon},
$$
which clearly implies $E(f_\omega)<\infty$ by writing the expectation into an integration with respect to distribution function. Thus, for each $\alpha$, we have
\begin{eqnarray*}
\infty%
&>& E \left(\int_0^1 r^{2n-1} (1-r)^{q+ns-n-1} M_p^p(r, f_\omega) dr \right) \\
&=& \int_0^1 r^{2n-1}(1-r)^{q+ns-n-1} \left[ \int_\SSS \int_\Omega |f_\omega(r\xi)|^pdPd\sigma(\xi) \right] dr \\
&=& \int_0^1 r^{2n-1}(1-r)^{q+ns-n-1} \left[\int_\SSS \int_\Omega  \left| \sum_\alpha \varepsilon_\alpha(\omega) a_\alpha r^{|\alpha|} \xi^\alpha \right|^p dPd\sigma(\xi)\right]dr\\
&\gtrsim& \int_0^1 r^{2n-1}(1-r)^{q+ns-n-1}  \left[ \int_\SSS \left( \sum_\alpha |a_\alpha|^2r^{2|\alpha|} |\xi^{\alpha}|^2 \right)^{\frac{p}{2}}d\sigma(\xi) \right]dr\\
&& \quad (\textrm{by Khintchine's inequality}) \\
&\ge& \int_0^1 r^{2n-1} (1-r)^{q+ns-n-1} \left[ \int_\SSS |a_\alpha|^p r^{p|\alpha|} |\xi^\alpha|^p d\sigma(\xi) \right] dr \\
&=& |a_\alpha|^p w_{\alpha, p}^p B(2n+p|\alpha|, q+ns-n),
\end{eqnarray*}
which implies the desired result.
\end{proof}


\medskip

\section{Characterizations of $\calN(p, q, s)$-type spaces}


Recall that by Proposition \ref{Carleson01}, $f $ belongs to $\calN(p, q, s)$ is equivalent to $d\mu=d\mu_{f, p, q, s}(z)=|f(z)|^p(1-|z|^2)^{q+ns}d\lambda(z)$ is an $(ns)$-Carleson measure. In this section, we extend this result and establish several characterizations of the $\calN(p, q, s)$-norm.


\subsection{Various derivative  characterizations}


Let us recall several notations first. For $f \in H(\B), z \in \B$. Let
$$
\nabla f(z)=\left(\frac{\partial f}{\partial z_1}(z), \dots, \frac{\partial f}{\partial z_n}(z)\right)
$$
denote the complex gradient of $f$ and let $\widetilde {\nabla}f$ denote the invariant gradient of $\B$, i.e., $(\widetilde {\nabla}f)(z)=\nabla(f \circ \Phi_z)(0)$. Moreover, we write
$$
Rf(z)=\sum_{k=1}^n z_k \frac{\partial f}{\partial z_k}(z)
$$
as the radial derivative of $f$ (see, e.g., \cite{Zhu}) and for $1 \le i, j \le n$,
$$
T_{i, j}f(z)=\bar{z_j} \frac{\partial f}{\partial z_i}-\bar{z_i}\frac{\partial f}{\partial z_j}
$$
as the tangential derivative of $f$ (see, e.g., \cite{MJ}). We need the following lemmas.

\begin{lem} \label{Charac01}
Let $\xi \in \SSS$ and $0<\del<1$. Then there exists some $M>0$, which is independent of $\del$, such that
$$
\bigcup_{z \in Q_\del(\xi)} D \left(z, \frac{1}{4} \right) \subseteq Q_{M\del}(\xi).
$$
\end{lem}

\begin{proof}
Take $w \in \bigcup\limits_{z \in Q_r(\xi)} D\left(z, \frac{1}{4}\right)$, which implies $w \in D\left(z, \frac{1}{4}\right)$ for some $z \in Q_\del(\xi)$. Since $|\Phi_w(z)|<\frac{1}{4}$, by \cite[Proposition 1.21 and Lemma 2.20]{Zhu}, there exists some $M'>0$ independent of $z$ and $w$, such that
\begin{equation} \label{Charac03}
\frac{1}{M'} \le \frac{1-|w|^2}{1-|z|^2} \le M'.
\end{equation}
Hence, we have
\begin{eqnarray*}
|1-\langle w, \xi\rangle|^{1/2}%
&\le& |1-\langle w, z \rangle|^{1/2}+|1-\langle z, \xi \rangle|^{1/2} \\
&\le& \del^{1/2}+\left( \frac{(1-|z|^2)(1-|w|^2)}{1-|\Phi_w(z)|^2}\right)^{\frac{1}{4}}\\
&\le& \del^{1/2}+2M'^{1/4}\del^{1/2}.
\end{eqnarray*}
The result follows by taking $M=(1+2M'^{1/4})^2$.
\end{proof}

\begin{lem} \label{Charac02}
For $f \in H(\B)$, there exists a constant $C>0$, such that
$$
|Rf(z)| \le \frac{C}{(1-|z|^2)^{1/2}} \int_{D\left(z, \frac{1}{4}\right)} \sum_{i<j} |T_{i, j}f(w)|d\lambda(w), \ \forall z \in \B.
$$
\end{lem}

\begin{proof}
The proof of this lemma is a simple modification of \cite[Lemma 2]{ZH} and hence we omit it here.
\end{proof}

We have the following result.

\begin{thm} \label{Carlesoncharac}
Let $f \in H(\B)$ and $p \ge 1, q>0$ and $s>\max\left\{0, 1-\frac{q}{n}\right\}$. The following statements are equivalent:
\begin{enumerate}
\item $f \in \calN(p, q, s)$ or equivalently, $d\mu_1=|f(z)|^p(1-|z|^2)^{q+ns}d\lambda(z)$ is an $(ns)$-Carleson measure;
\item $d\mu_2=|\nabla f(z)|^p(1-|z|^2)^{p+q+ns}d\lambda(z)$ is an $(ns)$-Carleson measure;
\item $d\mu_3=|\widetilde{\nabla} f(z)|^p(1-|z|^2)^{q+ns}d\lambda(z)$ is an $(ns)$-Carleson measure;
\item $d\mu_4=|Rf(z)|^p(1-|z|^2)^{p+q+ns}d\lambda(z)$ is an $(ns)$-Carlson measure;
\item $d\mu_5=\left(\sum\limits_{i<j} |T_{i, j}f(z)|^p\right)(1-|z|^2)^{\frac{p}{2}+q+ns}d\lambda(z)$ is an $(ns)$-Carleson measure.
\end{enumerate}
\end{thm}

\begin{proof}
Note that the equivalence between (1) and (2) follows from \cite[Theorem 3.2]{XZ} and \cite[Theorem 45]{ZZ}. Moreover, since $$(1-|z|^2)|Rf(z)| \le (1-|z|^2)|\nabla f(z)| \le \widetilde{\nabla} f(z)|, z \in \B,$$
  it is clear that $(3) \Longrightarrow (2) \Longrightarrow (4)$. Furthermore, the identity
$$
|z|^2|\widetilde{\nabla}f(z)|^2=(1-|z|^2)\left(\left(1-|z|^2\right)|Rf(z)|^2+\sum_{i<j}|T_{i, j}f(z)|^2\right)
$$
implies that (3) follows from (4) and (5). Therefore, it suffices to show the equivalence between (4) and (5).

$ (5) \Longrightarrow (4).$ Suppose $d\mu_5$ is an $(ns)$-Carleson measure. First we note that for any $z \in \B$,
\begin{equation} \label{Charac04}
\int_{D\left(z, \frac{1}{4}\right)} d\lambda(w)<K,
\end{equation}
for some $K$ independent of the choice of $z$. Indeed, by \cite[2.2.7]{Rud}, we have
\begin{eqnarray*}
\int_{D\left(z, \frac{1}{4}\right)} d\lambda(w)&= &
\int_{D\left(z, \frac{1}{4}\right)} \frac{1}{(1-|w|^2)^{n+1}}dV(w)\\
&\simeq &\frac{1}{(1-|z|^2)^{n+1}} \int_{D\left(z, \frac{1}{4}\right)}dV(w) \quad (\textrm{by} \ \eqref{Charac03}) \\
&\simeq &\frac{1}{(1-|z|^2)^{n+1}} \cdot (1-|z|^2)^{n+1}=1.
\end{eqnarray*}
By Lemmas \ref{Charac01} and \ref{Charac02}, \eqref{Charac03} and \eqref{Charac04}, for any $\xi \in \SSS$ and $\del \in (0, 1)$,
\begin{eqnarray*}
&&\int_{Q_\del(\xi)} |Rf(z)|^p(1-|z|^2)^{p+q+ns}d\lambda(z) \\
&\lesssim &\int_{Q_\del(\xi)} \left( \int_{D\left(z, \frac{1}{4}\right)} \sum_{i<j} |T_{i, j}f(w)|d\lambda(w) \right)^p (1-|z|^2)^{\frac{p}{2}+q+ns}d\lambda(z)\\
&\lesssim& \int_{Q_\del(\xi)} \sum_{i<j} \left(\int_{D\left(z, \frac{1}{4}\right)} |T_{i, j}f(w)|d\lambda(w)\right)^p (1-|z|^2)^{\frac{p}{2}+q+ns}d\lambda(z)
\end{eqnarray*}
\begin{eqnarray*}
&\lesssim & \int_{Q_\del(\xi)} \left(\sum_{i<j} \int_{D\left(z, \frac{1}{4}\right)} |T_{i, j}f(w)|^pd\lambda(w)\right) (1-|z|^2)^{\frac{p}{2}+q+ns}d\lambda(z)\\
&\le & \int_\B \sum_{i<j} |T_{i, j}f(w)|^p \chi_{\bigcup_{z \in Q_\del(\xi)} D \left(z, \frac{1}{4} \right)}(w)\times \\
&& \left(\int_{Q_\del(\xi)} \chi_{ D \left(z, \frac{1}{4} \right)}(z)(1-|z|^2)^{\frac{p}{2}+q+ns}d\lambda(z)\right) d\lambda(w)\\
&\simeq& \int_\B (1-|w|^2)^{\frac{p}{2}+q+ns}\sum_{i<j} |T_{i, j}f(w)|^p \chi_{\bigcup_{z \in Q_\del(\xi)} D \left(z, \frac{1}{4} \right)}(w) \times\\
\\
&& \left(\int_{Q_\del(\xi)} \chi_{ D \left(z, \frac{1}{4} \right)}(z)d\lambda(z)\right) d\lambda(w)\\
&\lesssim & \int_\B (1-|w|^2)^{\frac{p}{2}+q+ns}\sum_{i<j} |T_{i, j}f(w)|^p \chi_{\bigcup_{z \in Q_\del(\xi)} D \left(z, \frac{1}{4} \right)}(w) d\lambda(w)\\
&\le& \int_{Q_{M\del}(\xi)} (1-|w|^2)^{\frac{p}{2}+q+ns}\sum_{i<j} |T_{i, j}f(w)|^p d\lambda(w)\\
&\lesssim& \del^{ns} \quad (\textrm{since} \ d\mu_5 \ \textrm{is an $(ns)$-Carleson measure}).
\end{eqnarray*}
Hence, we get the desired result.

$ (4) \Longrightarrow (5).$ Suppose $d\mu_4$ is an $(ns)$-Carleson measure. From
$$
f(z)-f(0)=\int_0^1 \frac{d}{dt}f(tz)dt=\int_0^1 \frac{Rf(tz)}{t}dt,
$$
we see that for $1 \le i, j \le n$,
$$
T_{i, j}f(z)=\int_0^1 \frac{T_{i, j}(Rf(tz))}{t}dt=\int_0^1 \frac{(T_{i, j}Rf)(tz)}{t}dt.
$$
Hence, it suffices to prove for each $1 \le i, j \le n$,
$$
\left| \int_0^1 \frac{(T_{i, j}Rf)(tz)}{t}dt \right|^p (1-|z|^2)^{\frac{p}{2}+q+ns}d\lambda(z)
$$
is an $(ns)$-Carleson measure.

Note that by \cite[Corollary 5.24]{Zhu} and the fact that  $\frac{p}{2}+q+ns-n-1>-1$, we have, for any $\xi \in \SSS$ and $0<\del<1$,
\begin{eqnarray*}
&&\int_{Q_r(\xi)} \left| \int_0^{1/2} \frac{(T_{i, j}Rf)(tz)}{t}dt\right|^p(1-|z|^2)^{\frac{p}{2}+q+ns}d\lambda(z)\\
&\lesssim& \int_{Q_r(\xi)} (1-|z|^2)^{\frac{p}{2}+q+ns-n-1}dV(z)  \simeq  r^{\frac{p}{2}+q+ns} \le r^{ns}
\end{eqnarray*}
Thus, we only need to show for $1 \le i, j \le n$,
$$
\left( \int_{1/2}^1 |(T_{i, j}Rf)(tz)|dt \right)^p (1-|z|^2)^{\frac{p}{2}+q+ns}d\lambda(z)
$$
is an $(ns)$-Carleson measure.

By the proof of \cite[Theorem 1]{ZH}, for any $\gamma \ge 0$, we have
\begin{equation} \label{Charac05}
 \int_{1/2}^1 |(T_{i, j}Rf)(tw)|dt \lesssim \int_{\B} \frac{(1-|z|^2)^\gamma |Rf(z)|}{|1-\langle z, w \rangle|^{n+\gamma+\frac{1}{2}}}dV(z), \ w \in \B.
\end{equation}

Now we consider two cases.\\

\textit{Case I: $p>1$.}  Let $p'$ be the conjugate of $p$. Take and fix two positive numbers $\gamma$ and $\rho$, such that
$$
\max\left\{0, 1-\frac{p'}{2}\right\}<2p'\rho<1, \gamma>\max\left\{(p+p')\rho, p+q+ns-n-1-p\rho \right\}.
$$
Then for $w \in\B$, by \cite[Theorem 1.12]{Zhu}, we have
\begin{eqnarray*}
&&\left(\int_\B \frac{(1-|z|^2)^\gamma |Rf(z)|}{|1-\langle z, w \rangle|^{n+\gamma+\frac{1}{2}}}dV(z) \right)^p\\
&= & \left(\int_\B \frac{(1-|z|^2)^{\frac{\gamma}{p}} (1-|z|^2)^{\frac{\gamma}{p'}}|Rf(z)|}{|1-\langle z, w \rangle|^{n+\gamma+\frac{1}{2}}}dV(z) \right)^p\\
&=&\left(\int_\B \frac{(1-|z|^2)^{\frac{\gamma}{p}+\rho} (1-|z|^2)^{\frac{\gamma}{p'}-\rho}|Rf(z)|}{|1-\langle z, w \rangle|^{\frac{n+\gamma}{p}-\rho} |1-\langle z, w \rangle|^{\frac{n+\gamma}{p'}+\rho+\frac{1}{2}}}dV(z) \right)^p\\
&\le &\left( \int_\B \frac{(1-|z|^2)^{\gamma+p\rho}|Rf(z)|^p}{|1-\langle z, w \rangle|^{n+\gamma-p\rho}}dV(z)\right)  \left( \int_\B \frac{(1-|z|^2)^{\gamma-p'\rho}}{|1-\langle z, w \rangle|^{n+\gamma+p'\rho+\frac{p'}{2}}}dV(z)\right)^{p/p'}\\
&\simeq &\left( \int_\B \frac{(1-|z|^2)^{\gamma+p\rho}|Rf(z)|^p}{|1-\langle z, w \rangle|^{n+\gamma-p\rho}}dV(z)\right)   \left((1-|w|^2) ^{1-2p'\rho-\frac{p'}{2}}\right)^{p-1} \\
&=&\left( \int_\B \frac{(1-|z|^2)^{\gamma+p\rho}|Rf(z)|^p}{|1-\langle z, w \rangle|^{n+\gamma-p\rho}}dV(z)\right)   (1-|w|^2) ^{\frac{p}{2}-2p\rho-1}  .
\end{eqnarray*}

Thus, for any $\xi \in \SSS$ and $rl \in (0, 1)$,
\begin{eqnarray*}
&& \int_{Q_r(\xi)} \left( \int_{1/2}^1 |(T_{i, j}Rf)(tw)|dt \right)^p (1-|w|^2)^{\frac{p}{2}+q+ns}d\lambda(w) \\
&\lesssim & \int_{Q_r(\xi)} \left( \int_{\B} \frac{(1-|z|^2)^\gamma |Rf(z)|}{|1-\langle z, w \rangle|^{n+\gamma+\frac{1}{2}}}dV(z) \right)^p (1-|w|^2)^{\frac{p}{2}+q+ns}d\lambda(w) \\
&\lesssim& \int_{Q_r(\xi)} \left( \int_\B \frac{(1-|z|^2)^{\gamma+p\rho}|Rf(z)|^p}{|1-\langle z, w \rangle|^{n+\gamma-p\rho}}dV(z)\right)  (1-|w|^2) ^{p+q+ns-2p\rho-1}d\lambda(w)\\
&= &\int_{Q_r(\xi)} \left( \int_{Q_{2r}(\xi)} \frac{(1-|z|^2)^{\gamma+p\rho}|Rf(z)|^p}{|1-\langle z, w \rangle|^{n+\gamma-p\rho}}dV(z)\right) (1-|w|^2) ^{p+q+ns-2p\rho-1}d\lambda(w)\\
&&   +\sum_{j=1}^\infty  \int_{Q_r(\xi)} \bigg\{ \left( \int_{Q_{2^{j+1}r}(\xi)\backslash Q_{2^jr}(\xi)} \frac{(1-|z|^2)^{\gamma+p\rho}|Rf(z)|^p}{|1-\langle z, w \rangle|^{n+\gamma-p\rho}}dV(z)\right) \\
&&  \quad(1-|w|^2) ^{p+q+ns-2p\rho-1}d\lambda(w)\bigg\}\\
&=& I_1+I_2.
\end{eqnarray*}

\textbf{$\bullet$  \ Estimation of $I_1$.}

 Since $2p'\rho<1$ and $s>1-\frac{q}{n}$, it follows that
$$
p+q+ns-n-2p\rho-2>p-2p\rho-2>-1.
$$
Hence, by \cite[Theorem 1.12]{Zhu} and Fubini's theorem, we have
\begin{eqnarray*}
I_1%
&=& \int_{Q_{2r}(\xi)} (1-|z|^2)^{\gamma+p\rho}|Rf(z)|^p \\
&& \left(\int_{Q_r(\xi)}\frac{(1-|w|^2)^{p+q+ns-n-2p\rho-2}}{|1-\langle z, w \rangle|^{n+t-p\rho}}dV(w) \right) dV(z)\\
&\le& \int_{Q_{2r}(\xi)} (1-|z|^2)^{\gamma+p\rho}|Rf(z)|^p \\
&& \left(\int_\B\frac{(1-|w|^2)^{p+q+ns-n-2p\rho-2}}{|1-\langle z, w \rangle|^{n+t-p\rho}}dV(w) \right) dV(z)\\
&\simeq& \int_{Q_{2r}(\xi)} |Rf(z)|^p(1-|z|^2)^{\gamma+p\rho+p+q+ns-\gamma-p\rho-n-1}dV(z) \\
&=& \int_{Q_{2r}(\xi)} |Rf(z)|^p(1-|z|^2)^{p+q+ns}d\lambda(z)\\
&\lesssim&  r^{ns}.
\end{eqnarray*}

\textbf{$\bullet$  Estimation of $I_2$.}

First we note that for $w \in Q_r(\xi)$ and $$z \in Q_{2^{j+1}r}(\xi) \backslash Q_{2^jr}(\xi), j \in \N,
$$
 we have
\begin{eqnarray*}
|1-\langle z, w \rangle|^{1/2}%
&\ge& |1-\langle z, \xi \rangle|^{1/2}-|1-\langle w, \xi\rangle|^{1/2}\\
&\ge& (2^{j/2}-1) r^{1/2},
\end{eqnarray*}
which implies for $w \in Q_r(\xi)$ and $z \in Q_{2^{j+1}r}(\xi) \backslash Q_{2^jr}(\xi), j \in \N$, we have
$|1-\langle z, w \rangle| \ge \frac{2^jr}{100}.$ Since $\gamma+p\rho>p+q+ns-n-1$, we put
$\beta=\gamma+p\rho-p-q-ns+n+1>0,$ and hence
$n+\gamma-p\rho=p+q+ns+\beta-2p\rho-1.$
Moreover, since $2p'\rho<1$, we have
$2p\rho+1-p-q<0.$

Thus, using the fact that $p+q+ns-2p\rho-n-2>-1$ and \cite[Corollary 5.24]{Zhu}, it follows that
\begin{eqnarray*}
I_2%
&=& \sum_{j=1}^\infty  \bigg\{ \int_{Q_r(\xi)}  \int_{Q_{2^{j+1}r}(\xi)\backslash Q_{2^jr}(\xi)} \frac{(1-|z|^2)^\beta}{|1-\langle z, w \rangle|^\beta} \times\\
&&  \frac{(1-|z|^2)^{p+q+ns-n-1}|Rf(z)|^p}{|1-\langle z, w \rangle|^{p+q+ns-2p\rho-1}}dV(z)   (1-|w|^2)^{p+q+ns-2p\rho-1}d\lambda(w) \bigg\}\\
&\lesssim& \sum_{j=1}^\infty  \bigg\{ \int_{Q_r(\xi)} \left( \int_{Q_{2^{j+1}r}(\xi)\backslash Q_{2^jr}(\xi)} \frac{(1-|z|^2)^{p+q+ns-n-1}|Rf(z)|^p}{|1-\langle z, w \rangle|^{p+q+ns-2p\rho-1}}dV(z)\right) \\
&& \quad \quad \quad (1-|w|^2) ^{p+q+ns-2p\rho-1}d\lambda(w) \bigg\}\\
&\lesssim& \sum_{j=1}^\infty \bigg\{(2^j r)^{2p\rho+1-p-q-ns} \int_{Q_r(\xi)} (1-|w|^2)^{p+q+ns-2p\rho-1}\\
&& \quad \quad \quad \quad  \left(\int_{Q_{2^{j+1}r}(\xi)} (1-|z|^2)^{p+q+ns}|Rf(z)|^p d\lambda(z)  \right)d\lambda(w)\bigg\}\\
&\lesssim& \sum_{j=1}^\infty \left\{ (2^j r)^{2p\rho+1-p-q-ns} (2^{j+1}r)^{ns} \int_{Q_r(\xi)} (1-|w|^2)^{p+q+ns-2p\rho-1}d\lambda(z) \right\}\\
&\simeq& \sum_{j=1}^\infty (2^j r)^{2p\rho+1-p-q-ns} (2^{j+1}r)^{ns} r^{p+q+ns-2p\rho-1}\\
&\simeq& \left(\sum_{j=1}^\infty 2^{j(2p\rho+1-p-q)}\right) r^{ns} \lesssim r^{ns}.
\end{eqnarray*}

Finally, combining the estimations of $I_1$ and $I_2$, we get the desired result.\\

\textit{Case II: $p=1$.}

Take and fix a positive number $\gamma$, such that $\gamma>q+ns-n,$ and put $\beta=\gamma-q-ns+n$, which implies that
$$n+\gamma+\frac{1}{2}=\beta+q+ns+\frac{1}{2}.$$

 For any $\xi \in \SSS$ and $r \in (0, 1)$, we have
\begin{eqnarray*}
&&\int_{Q_r(\xi)} \left(\int_{1/2}^1 |(T_{i, j}Rf)(tw)|dt\right) (1-|w|^2)^{\frac{1}{2}+q+ns}d\lambda(w)\\
&\lesssim& \int_{Q_r(\xi)} \left(\int_\B \frac{(1-|z|^2)^\gamma |Rf(z)|}{|1-\langle z,w \rangle|^{n+\gamma+\frac{1}{2}}}dV(z) \right)(1-|w|^2)^{\frac{1}{2}+q+ns}d\lambda(w)\\
&= &\int_{Q_r(\xi)} (1-|w|^2)^{\frac{1}{2}+q+ns} \left(\int_{Q_{2r}(\xi)} \frac{(1-|z|^2)^\gamma |Rf(z)|}{|1-\langle z,w \rangle|^{n+\gamma+\frac{1}{2}}}dV(z)\right) d\lambda(w) \\
&&   + \sum_{j=1}^\infty \int_{Q_r(\xi)} \bigg\{ \left( \int_{Q_{2^{j+1}r}(\xi)\backslash Q_{2^jr}(\xi)} \frac{(1-|z|^2)^\gamma |Rf(z)|}{|1-\langle z,w \rangle|^{n+\gamma+\frac{1}{2}}}dV(z)\right) \\
&&  \quad (1-|w|^2)^{\frac{1}{2}+q+ns} d\lambda(w) \bigg\}\\
&=& J_1+J_2.
\end{eqnarray*}

\textbf{$\bullet$ Estimation of $J_1$.}

By Fubini's theorem and \cite[Theorem 1.12]{Zhu}, we have
\begin{eqnarray*}
J_1%
&=& \int_{Q_r(\xi)} (1-|w|^2)^{q+ns-n-\frac{1}{2}} \left(\int_{Q_{2r}(\xi)} \frac{(1-|z|^2)^\gamma |Rf(z)|}{|1-\langle z,w \rangle|^{n+\gamma+\frac{1}{2}}}dV(z)\right) dV(w) \\
&\le& \int_{Q_{2r}(\xi)}|Rf(z)|(1-|z|^2)^\gamma  \left( \int_\B \frac{(1-|w|^2)^{q+ns-n-\frac{1}{2}}}{|1-\langle z, w\rangle|^{n+\gamma+\frac{1}{2}}}dV(w) \right) dV(z)\\
&\simeq& \int_{Q_{2r}(\xi)}|Rf(z)|(1-|z|^2)^\gamma (1-|z|^2)^{q+ns-n-\gamma}dV(z)\\
&=& \int_{Q_{2r}(\xi)} |Rf(z)|(1-|z|^2)^{1+q+ns}d\lambda(z) \lesssim r^{ns}.
\end{eqnarray*}

\textbf{$\bullet$ Estimation of $J_2$.}

By our choice of $\gamma$ and \cite[Corollary 5.24]{Zhu}, we have
\begin{eqnarray*}
J_2%
&=& \sum_{j=1}^\infty \int_{Q_r(\xi)}  \int_{Q_{2^{j+1}r}(\xi)\backslash Q_{2^jr}(\xi)}  \frac{(1-|z|^2)^\beta}{|1-\langle z, w \rangle|^\beta} \times\\
&& \frac{(1-|z|^2)^{q+ns-n} |Rf(z)|}{|1-\langle z,w \rangle|^{q+ns+\frac{1}{2}}}dV(z)  (1-|w|^2)^{\frac{1}{2}+q+ns} d\lambda(w)  \\
&\lesssim& \sum_{j=1}^\infty \int_{Q_r(\xi)}   \int_{Q_{2^{j+1}r}(\xi)\backslash Q_{2^jr}(\xi)} \frac{(1-|z|^2)^{q+ns-n} |Rf(z)|}{|1-\langle z,w \rangle|^{q+ns+\frac{1}{2}}}dV(z)  \\
&& \quad \quad \quad (1-|w|^2)^{\frac{1}{2}+q+ns} d\lambda(w)  \\
&\lesssim& \sum_{j=1}^\infty \bigg\{(2^j r)^{-q-ns-\frac{1}{2}} \int_{Q_r(\xi)} (1-|w|^2)^{\frac{1}{2}+q+ns} \\
&& \quad \quad \quad \left(\int_{Q_{2^{j+1}r}(\xi)} |Rf(z)|(1-|z|^2)^{1+q+ns}d\lambda(z)\right)d\lambda(w)\bigg\} \\
&\lesssim& \sum_{j=1}^\infty (2^jr)^{-q-ns-\frac{1}{2}} (2^{j+1}r)^{ns} \int_{Q_r(\xi)}(1-|w|^2)^{\frac{1}{2}+q+ns}d\lambda(w)\\
&\simeq& \sum_{j=1}^\infty (2^jr)^{-q-ns-\frac{1}{2}} (2^{j+1}r)^{ns}  r^{\frac{1}{2}+q+ns} \lesssim r^{ns}.
\end{eqnarray*}
The desired result follows from the estimations on $J_1$ and $J_2$.
\end{proof}

Correspondingly, we have the following result for $\calN^0(p, q, s)$ spaces.

\begin{thm} \label{vanishingeqnorm}
Let $f \in H(\B), p \ge 1, q>0$ and $s>\max\left\{0, 1-\frac{q}{n}\right\}$. The following statements are equivalent:
\begin{enumerate}
\item $f \in \calN^0(p, q, s)$ or equivalently, $d\mu_1=|f(z)|^p(1-|z|^2)^{q+ns}d\lambda(z)$ is a vanishing $(ns)$-Carleson measure;
\item $d\mu_2=|\nabla f(z)|^p(1-|z|^2)^{p+q+ns}d\lambda(z)$ is a vanishing $(ns)$-Carleson measure;
\item $d\mu_3=|\widetilde{\nabla} f(z)|^p(1-|z|^2)^{q+ns}d\lambda(z)$ is a vanishing $(ns)$-Carleson measure;
\item $d\mu_4=|Rf(z)|^p(1-|z|^2)^{p+q+ns}d\lambda(z)$ is a vanishing $(ns)$-Carlson measure;
\item $d\mu_5=\left(\sum\limits_{i<j} |T_{i, j}f(z)|^p\right)(1-|z|^2)^{\frac{p}{2}+q+ns}d\lambda(z)$ is a vanishing $(ns)$-Carleson measure.
\end{enumerate}
\end{thm}

\begin{proof}
The proof of this theorem is a simple modification of Theorem \ref{Carlesoncharac} and  hence we omit it here.
\end{proof}

Moreover, combining Theorem \ref{Carlesoncharac}, \cite[Theorem 45]{ZZ} and \cite[Theorem 3.2]{XZ}, we have the following characterizations of $\calN(p, q, s)$-norm. In particular, we have:

\begin{cor} \label{equivalentnorm}
Let $f \in H(\B)$, $p \ge 1, q>0$ and $s>\max\left\{0, 1-\frac{q}{n}\right\}$. The following quantities are equivalent.
\begin{enumerate}
\item
$$
I_1=\sup_{a \in \B} \int_\B |f(z)|^p(1-|z|^2)^q(1-|\Phi_a(z)|^2)^{ns}d\lambda(z);
$$
\item
$$
I_2=|f(0)|^p+\sup_{a \in \B} \int_\B |\nabla f(z)|^p(1-|z|^2)^{p+q}(1-|\Phi_a(z)|^2)^{ns}d\lambda(z);
$$
\item
$$
I_3=|f(0)|^p+\sup_{a \in \B} \int_\B |R f(z)|^p(1-|z|^2)^{p+q}(1-|\Phi_a(z)|^2)^{ns}d\lambda(z);
$$
\item
$$
I_4=|f(0)|^p+\sup_{a \in \B} \int_\B |\widetilde {\nabla}f(z)|^p(1-|z|^2)^q(1-|\Phi_a(z)|^2)^{ns}d\lambda(z);
$$
\item
$$
I_5=|f(0)|^p+\sup_{a \in \B} \int_\B \left(\sum\limits_{i<j} |T_{i, j}f(z)|^p\right)(1-|z|^2)^{\frac{p}{2}+q}(1-|\Phi_a(z)|^2)^{ns} d\lambda(z).
$$
\end{enumerate}
\end{cor}

\begin{rem} \label{NpqsFpqsID} From the above results, it is easy to see that
$$
\calN(p, q, s)=F(p, p+q-n-1, ns)
$$
and
$$
\calN^0(p, q, s)=F_0(p, p+q-n-1, ns).
$$
An alternative proof of the equivalence of $(1), (2), (3)$ and $(4)$ in Corollary \ref{equivalentnorm} can be found in \cite{ZHC}, which depends on a careful estimation of the quantity
$$
I_{w, a}=\int_\B \frac{(1-|z|^2)^\del}{|1-\langle z, w \rangle|^t |1-\langle z, a \rangle|^r}dV(z),
$$
where $w, a \in \B$, $\del>-1$, $t, r \ge 0$ and $r+t-\del>n+1$. However, their approach does not cover the last quantity $I_5$.
\end{rem}


\medskip

\subsection{Korenblum's inequality of $\calN(p, q, s)$-type spaces}


We denote $f_r(z)=f(rz)$ for $f$ being a holomorphic function and $0 \le r<1$. Recall that the Korenblum's inequality usually refers to the following inequality:
$$
\|g_r\|_{BMOA} \le \|g\|_{\calB} \sqrt{|\log(1-r^2)|},
$$
where $g \in H(\D)$, BMOA is the space of bounded mean oscillation of holomorphic functions and $\calB$ is the Bloch space. By \cite{KO}, the above estimation is sharp.

As an application of our main result in this section, we study the Korenblum's inequality for $\calN(p, q, s)$-type spaces.

The following description via higher radial derivative is straightforward from Corollary \ref{equivalentnorm}.

\begin{cor} \label{higherequivnorm}
Let $f \in H(\B)$, $p \ge 1, q>0$, $s>\max\left\{0, 1-\frac{q}{n}\right\}$ and $m \in \N$. Then $f \in \calN(p, q, s)$ if and only if
$$
\sup_{a \in \B} \int_\B |R^mf(z)|^p(1-|z|^2)^{mp+q}(1-|\Phi_a(z)|^2)^{ns}d\lambda(z)<\infty.
$$
\end{cor}

We need the following preliminaries. Let $\alpha>0$. The $\alpha$-Bloch space $\calB^\alpha$ is the space of all   $f\in H(\B)$ such that $\sup_{a \in \B} (1-|z|^2)^\alpha|Rf(z)|<\infty$ and the little $\alpha$-Bloch space $\calB^\alpha_0$ consists of those  $f\in H(\B)$ satisfying $
\lim_{|z| \to 1}(1-|z|^2)^\alpha|Rf(z)|=0.$
It is well-known that $\calB^\alpha$ becomes a Banach space with the norm $$\|f\|_{\calB^\alpha}=|f(0)|+\sup_{z \in \B} (1-|z|^2)^\alpha|Rf(z)|. $$ We denote
$$
K_1=|f(0)|+\sup_{z \in \B} (1-|z|^2)^\alpha |\nabla f(z)|
$$
and
$$
K_2=|f(0)|+\sup_{z \in \B} (1-|z|^2)^{\alpha-1} |\widetilde{\nabla}f(z)|.
$$
It is known that when $\alpha>0$, $K_1$ and $\|f\|_{\calB^\alpha}$ are equivalent and when $\alpha>\frac{1}{2}$, the same is true for both $K_1$ and $K_2$ (see, e.g., \cite[Theorem 7.1 and Theorem 7.2]{Zhu}.

The $\alpha$-Bloch space $\calB^\alpha$ has a close relationship with the Bergman-type space $A^{-p}(\B)$. It is a classical result that for $p>0$, we have (see, e.g., \cite{Zhu})
$$
A^{-p}(\B)=\calB^{p+1}.
$$
 However, we did not find a reference for the proof of the above result,  and hence we give its proof here for the readers' convenience.

\begin{lem} \label{BlochBertype}
Suppose $p>0$. Then $f$ is in $\calB^{p+1}$ if and only if $|f(z)|(1-|z|^2)^p$ is bounded in $\B$; $f$ is in $\calB^{p+1}_0$ if and only if $(1-|z|^2)^p|f(z)| \to 0$ as $|z| \to 1^{-}$.
\end{lem}

\begin{proof}
$(i)$ \textit{Necessity.} Let $f \in \calB^{p+1}$. Take and fix some $\alpha>p+1$. It is clear that $Rf \in A^1_\alpha$. Hence, by the proof of \cite[Theorem 2.16]{Zhu} and \cite[Proposition 1.4.10]{Rud}, we have
\begin{eqnarray*}
|f(z)-f(0)|%
&\lesssim& \int_\B \frac{(1-|w|^2)^\alpha |Rf(w)|dV(w)}{|1-\langle z, w\rangle|^{n+\alpha}}\\
&\le& \|f\|_{\calB^{p+1}} \int_\B \frac{(1-|w|^2)^{\alpha-p-1}}{|1-\langle z, w \rangle|^{n+1+(\alpha-p-1)+p}}dV(w)\\
&\simeq& \frac{\|f\|_{\calB^{p+1}}}{(1-|z|^2)^p},
\end{eqnarray*}
which implies the desired claim.

\textit{Sufficiency.}
If $(1-|z|^2)^p|f(z)| \le M$ for some constant $M>0$, then by \cite[Theorem 2.2]{Zhu},
$$
f(z)=\int_\B \frac{f(w)}{(1-\langle z, w \rangle)^{n+1+p}}dV_p(w).
$$
Thus, by \cite[Proposition 1.4.10]{Rud}, we get
\begin{eqnarray*}
|Rf(z)|%
&=&(n+p+1) \left|\int_\B \frac{f(w)}{(1-\langle z, w\rangle)^{n+2+p}} \left(\sum_{k=1}^n z_k \bar{w_k}\right) dV_p(w)\right|\\
&\lesssim& \int_\B \frac{|f(w)|(1-|w|^2)^p}{|1-\langle z, w \rangle|^{n+2+p}}dV(w)\\
&\le& M \int_\B \frac{1}{|1-\langle z, w \rangle|^{n+1+p+1}}dV(w)\\
&\lesssim& \frac{1}{(1-|z|^2)^{p+1}},
\end{eqnarray*}
as desired.

$(ii)$ The second assertion follows from $(i)$ and the fact that $\calB^{p+1}_0$ and $A^{-p}_0(\B)$ are the closure of the polynomials in $\calB^{p+1}$ and $A^{-p}(\B)$ respectively.
\end{proof}

\begin{thm}
Let $p \ge 1, q>0$, $s>\max\left\{0, 1-\frac{q}{n} \right\}$ and $0 \le r<1$. Then for any $f\in H(\B)$, we have
\[ \|f_r\| \le \begin{cases}
C|f|_{\frac{q}{p}} \cdot \frac{1}{(1-r^2)^{\frac{q}{p}}} & s>1\\
 C|f|_{\frac{q}{p}}\cdot  \frac{1}{(1-r^2)^{\frac{2(n+1-ns)}{p}}} & s \le 1,
\end{cases} \]
where $C$ is some constant independent of $r$.
\end{thm}

\begin{proof}
When $s>1$, it is clear that the $\calN(p, q, s)$-norm is equivalent to the $A^{-\frac{q}{p}}(\B)$-norm, from which, the desired result follows trivially. Now we assume $s \le 1$.

Take and fix the smallest $m \in \N$ such that $mp+q-n-1>0$. First note that an easy calculation shows that $R(f_r)(z)=Rf(rz)=(Rf)_r(z), \forall z \in \B$, which implies
\begin{equation} \label{Koren0001}
R^m(f_r)(z)=(R^mf)_r(z), \quad \forall z \in \B.
\end{equation}
Next, for any $a \in \B$, apply the maximum modulus principle in $D(ar, r)$ to the function $z \mapsto \frac{z}{\Phi_a\left(\frac{\Phi_{ra}(z)}{r}\right)}$, it follows that
\begin{equation} \label{Koren0002}
\left|\Phi_a\left(\frac{\Phi_{ra}(z)}{r}\right) \right| \ge |z|,
\end{equation}
since $\Phi_a$ maps $\SSS$ homeomorphically to itself for each $a \in \B$.

Moreover, for any fixed $a \in \B$, we claim that
\begin{equation} \label{Koren0003}
\int_{D(ar, r)} (1-|z|^2)^{ns-n-1}dV(z) \le C \cdot \frac{r^{2n}}{(1-r^2)^{2(n+1-ns)}}.
\end{equation}
where $C$ is independent of both $a$ and $r$. Indeed, for any $z \in D(ar, r)$, we have $|\Phi_{ra}(z)|<r$, which implies $1-|\Phi_{ra}(z)|^2>1-r^2$, that is
$$
\frac{(1-|ra|^2)(1-|z|^2)}{|1-\langle z, ra \rangle|^2}>1-r^2.
$$
Thus, for $z \in D(ra, r)$, we have
$$
1-|z|^2 \ge \frac{(1-r^2)|1-\langle z, ra \rangle|^2}{1-|ra|^2} \gtrsim (1-r^2)(1-|ra|^2).
$$
Thus, by \cite[Lemma 1.23]{Zhu} and the fact that $ns-n-1<0$, we have
\begin{eqnarray*}
&& \int_{D(ar, r)} (1-|z|^2)^{ns-n-1}dV(z)\\
&\lesssim& (1-r^2)^{ns-n-1}(1-|ra|^2)^{ns-n-1} V(D(ar, r))\\
& =&(1-r^2)^{ns-n-1}(1-|ra|^2)^{ns} \cdot \frac{r^{2n}}{(1-r^4|a|^2)^{n+1}}\\
&\le& r^{2n}(1-r^2)^{ns-n-1}(1-r^4|a|^2)^{ns-n-1}\\
&\lesssim & \frac{r^{2n}}{(1-r^2)^{2(n+1-ns)}}.
\end{eqnarray*}
Therefore, by \eqref{Koren0001}, \eqref{Koren0002} and \eqref{Koren0003}, we have
\begin{eqnarray*}
&&\|f_r\|^p= \sup_{a \in \B} \int_\B |R^m(f_r)(z)|^p(1-|z|^2)^{mp+q}(1-|\Phi_a(z)|^2)^{ns}d\lambda(z) \\
&= &\sup_{a \in \B} \int_\B|R^mf (rz)|^p (1-|z|^2)^{mp+q-n-1}(1-|\Phi_a(z)|^2)^{ns}dV(z)\\
&=&\sup_{a \in \B} \int_{|w|<r} |R^mf(w)|^p \left(1-\left|\frac{w}{r}\right|^2\right)^{mp+q-n-1}\left(1-\left|\Phi_a\left( \frac{w}{r} \right)\right|^2\right)^{ns} \frac{dV(w)}{r^{2n}}\\
&&   \quad  (\textrm{change variable with} \ w=rz) \\
& =&\sup_{a \in \B} \int_{D(ar, r)} |R^mf \circ \Phi_{ra}(u)|^p \left(1-\left| \frac{\Phi_{ar}(u)}{r} \right|^2 \right)^{mp+q-n-1}\\
&&   \quad \quad \quad  \left(1-\left|\Phi_a \left( \frac{\Phi_{ra}(u)}{r} \right) \right|^2 \right)^{ns} \left( \frac{1-|ra|^2}{|1-\langle u, ra \rangle|^2} \right)^{n+1} \frac{dV(u)}{r^{2n}}\\
&&  \quad  (\textrm{change variable with} \ u=\Phi_{ar}(w))
\end{eqnarray*}
\begin{eqnarray*}
&& \le \sup_{a \in \B} \int_{D(ar, r)} |R^mf \circ \Phi_{ra}(u)|^p \left(1-\left|\Phi_{ar}(u) \right|^2 \right)^{mp+q-n-1}\\
&& \quad \quad \quad \quad  (1-|u|^2)^{ns} \left( \frac{1-|ra|^2}{|1-\langle u, ra \rangle|^2} \right)^{n+1} \frac{dV(u)}{r^{2n}}\\
&& \lesssim |f|_{\frac{q}{p}}^p \sup_{a \in \B} \int_{D(ar, r)} \frac{(1-|u|^2)^{ns-n-1}}{r^{2n}}dV(u)\\
&& \quad \quad (\textrm{by \cite[Lemma 15]{ZZ} and Corollary \ref{BlochBertype}}) \\
&& \lesssim |f|_{\frac{q}{p}}^p \cdot \frac{1}{(1-r^2)^{2(n+1-ns)}},
\end{eqnarray*}
which implies the desired result.
\end{proof}


\medskip

\subsection{Derivative-free, mixture and oscillation characterizations}


As a second application of Corollary \ref{equivalentnorm}, we study some other new derivative-free, mixture and oscillation characterizations to $\calN(p,q, s)$-spaces, whose idea comes from \cite{LH}.

We need the following lemma, which was proved in \cite{PZ}.

\begin{lem} \label{mixedlem}
Suppose $\alpha>-1, p>0, 0 \le \beta<p+2$ and $f \in H(\B)$. Then $f \in A^p_\alpha$ if and only if
$$
K(f)=\int_\B |f(z)|^{p-\beta} |\widetilde{\nabla}f(z)|^\beta dV_\alpha(z)<\infty.
$$
Moreover, the quantities $\|f\|_{p, \alpha}^p$ and $|f(0)|^p+K(f)$ are comparable for $f \in H(\B)$.
\end{lem}

\begin{thm} \label{mix01}
Suppose $f \in H(\B), p \ge 1, q>0, s>\max\left\{0, 1-\frac{q}{n}\right\}, 0 \le \beta<p+2$ and $\alpha>q+ns-n-1$. Then $f \in \calN(p, q, s)$ if and only if
$$
M=\sup_{a \in \B} \int_\B \int_\B \frac{|f(z)-f(w)|^{p-\beta}}{|1-\langle z, w\rangle|^{2(n+1+\alpha)}} |\widetilde{\nabla}f(z)|^\beta(1-|w|^2)^q $$
$$(1-|\Phi_a(w)|^2)^{ns}dV_\alpha(w)dV_\alpha(z)<\infty.
$$
\end{thm}

\begin{proof} {\it Sufficiency.} Suppose $M<\infty$. Note that by $$|\widetilde{\nabla}(f \circ \Phi_w)(z)|=|\widetilde{\nabla} f(\Phi_w(z))|$$
 and $\alpha>q+ns-n-1>-1$, we have
\begin{equation} \label{Charac07}
M=\sup_{a \in \B} \int_\B (1-|w|^2)^q(1-|\Phi_a(w)|^2)^{ns} \left( \int_\B |F_w(z)|^{p-\beta} |\widetilde{\nabla}F_w(z)|^\beta dV_\alpha(z) \right) d\lambda(w),
\end{equation}
where $F_w=f \circ \Phi_w-f(w), w \in \B$. By Lemma \ref{mixedlem},
$$
M \simeq \sup_{a \in \B} \int_\B (1-|w|^2)^q(1-|\Phi_a(w)|^2)^{ns} \left( \int_\B |F_w(z)|^pdV_\alpha(z) \right) d\lambda(w).
$$
Note that by \cite[Lemma 2.4]{Zhu}, we have, for any $w \in \B$,
\begin{eqnarray*}
|\widetilde{\nabla}f (w)|^p%
&=&|\nabla (f \circ \Phi_w)(0)|^p=|\nabla (f \circ \Phi_w-f(w))(0)|^p \\
&\lesssim& \int_\B |f \circ \Phi_w(z)-f(w)|^pdV_\alpha(z)=\int_\B |F_w(z)|^pdV_\alpha(z).
\end{eqnarray*}
Hence
$$
\infty>M \gtrsim  \sup_{a \in \B} \int_\B |\widetilde{\nabla}f(w)|^p (1-|w|^2)^q(1-|\Phi_a(w)|^2)^{ns}  d\lambda(w),
$$
which, by Corollary \ref{equivalentnorm}, implies $f \in \calN(p, q, s)$.

{\it Necessity.} Suppose $f \in \calN(p, q, s)$. First by Lemma \ref{mixedlem}, we see that the quantities
$$
\int_\B |F_w(z)|^{p-\beta} |\widetilde{\nabla}F_w(z)|^\beta dV_\alpha(z) \ \textrm{and} \ \int_\B |\widetilde{\nabla}F_w(z)|^pdV_\alpha(z)
$$
are comparable. Hence, by \eqref{Charac07}, we have
\begin{eqnarray*}
M%
&=&\int_\B (1-|w|^2)^q(1-|\Phi_a(w)|^2)^{ns} \left( \int_\B |F_w(z)|^{p-\beta} |\widetilde{\nabla}F_w(z)|^\beta dV_\alpha(z) \right) d\lambda(w)\\
&\lesssim& \int_\B \left( \int_\B |\widetilde{\nabla}F_w(z)|^pdV_\alpha(z) \right) (1-|w|^2)^q(1-|\Phi_a(w)|^2)^{ns}d\lambda(w)\\
&=& \int_\B \left( \int_\B |\widetilde{\nabla}f(\Phi_w(u))|^pdV_\alpha(u) \right) (1-|w|^2)^q(1-|\Phi_a(w)|^2)^{ns}d\lambda(w)\\
&=& \int_\B   \left( \int_\B |\widetilde{\nabla} f(z)|^p(1-|\Phi_w(z)|^2)^{n+1+\alpha}d\lambda(z)\right) \\
&& \quad \quad \quad \quad \quad (1-|w|^2)^q(1-|\Phi_a(w)|^2)^{ns}d\lambda(w)  \\
&\le& \int_\B |\widetilde{\nabla}f(z)|^p(1-|z|^2)^p(1-|\Phi_a(z)|^2)^{ns}d\lambda(z) \cdot I,
\end{eqnarray*}
where
\begin{eqnarray*}
I%
&=&\sup_{a, z \in \B} \int_\B \frac{(1-|w|^2)^q(1-|\Phi_a(w)|^2)^{ns}}{(1-|z|^2)^q(1-|\Phi_a(z)|^2)^{ns}} (1-|\Phi_w(z)|^2)^{n+1+\alpha}d\lambda(w)\\
&=& \sup_{a, z \in \B} \int_\B\frac{(1-|w|^2)^q(1-|\Phi_a(w)|^2)^{ns}}{(1-|z|^2)^q(1-|\Phi_a(z)|^2)^{ns}} (1-|\Phi_z(w)|^2)^{n+1+\alpha}d\lambda(w)\\
&=& \sup_{a, z \in \B} \int_\B \frac{(1-|\Phi_z(u)|^2)^q(1-|\Phi_a(\Phi_z(u))|^2)^{ns}}{(1-|z|^2)^q(1-|\Phi_a(z)|^2)^{ns}}(1-|u|^2)^{n+1+\alpha}d\lambda(u)\\
&\simeq& \sup_{a, z \in \B} \int_\B \frac{(1-|u|^2)^q (1-|u|^2)^{ns}}{|1-\langle u, a\rangle|^{2q} |1-\langle \Phi_a(z), u \rangle|^{2ns}}dV_\alpha(u) \\
& \le& \sup_{a, z \in \B} \bigg\{ \left(\int_\B \frac{(1-|u|^2)^{q+ns}}{|1-\langle u,a \rangle|^{2(q+ns)}}dV_\alpha(u) \right)^{\frac{q}{q+ns}} \\
&& \quad \quad \quad \cdot \left(\int_\B \frac{(1-|u|^2)^{q+ns}}{|1-\langle u,\Phi_z(a) \rangle|^{2(q+ns)}}dV_\alpha(u) \right)^{\frac{ns}{q+ns}} \bigg\} \\
&<& \infty.
\end{eqnarray*}
Here, we use the fact that $q+ns-n-1-\alpha<0$ in the last inequality.

Thus, we get
$$\int_\B (1-|w|^2)^q(1-|\Phi_a(w)|^2)^{ns} \left( \int_\B |F_w(z)|^{p-\beta} |\widetilde{\nabla}F_w(z)|^\beta dV_\alpha(z) \right) d\lambda(w) \\
 \lesssim  \|f\|^p,$$
 which implies the desired result.
\end{proof}

In particular, taking $\beta=0$, we get the following result.

\begin{thm} \label{mix02}
Suppose $f \in H(\B), p \ge 1, q>0, s>\max\left\{0, 1-\frac{q}{n}\right\}$ and $\alpha>q+ns-n-1$. Then $f \in \calN(p, q, s)$ if and only if
$$
 \sup_{a \in \B} \int_\B \int_\B \frac{|f(z)-f(w)|^p}{|1-\langle z, w\rangle|^{2(n+1+\alpha)}} (1-|w|^2)^q(1-|\Phi_a(w)|^2)^{ns}dV_\alpha(w)dV_\alpha(z)<\infty.
$$
\end{thm}

We also have the following description.

\begin{thm} \label{finter}
Suppose $f \in H(\B), p \ge 1, q>0, s>\max\left\{0, 1-\frac{q}{n}\right\}$ and $0<r<1$. Then the following statements are equivalent:
\begin{enumerate}
\item $f \in \calN(p, q, s)$;
\item \[ \begin{split}
\sup_{a \in \B} \int_\B \bigg( &\frac{1}{V(D(z, r))}  \int_{D(z, r)}  |f(z)-f(w)|(1-|z|^2)^{\frac{q}{2p}}(1-|w|^2)^{\frac{q}{2p}}\\
                                              &(1-|\Phi_a(z)|^2)^{\frac{ns}{2p}} (1-|\Phi_a(w)|^2)^{\frac{ns}{2p}}dV(w)\bigg)^pd\lambda(z)<\infty;
\end{split} \]
\item \[ \begin{split}
\sup_{a \in \B} \int_\B \big( &\sup_{w \in D(z, r)}  |f(z)-f(w)|(1-|z|^2)^{\frac{q}{2p}}(1-|w|^2)^{\frac{q}{2p}}\\
                                            & (1-|\Phi_a(z)|^2)^{\frac{ns}{2p}}(1-|\Phi_a(w)|^2)^{\frac{ns}{2p}}\big)^pd\lambda(z)<\infty;
\end{split} \]
\item There exists some $c$ satisfying $1<c<\frac{1}{r}$, such that
\[ \begin{split}
\sup_{a \in \B} \int_\B &\bigg( \frac{1}{V(D(z, cr))}  \int_{D(z, cr)}  |f(z)-f(w)|^p|(1-|z|^2)^{\frac{q}{2}}(1-|w|^2)^{\frac{q}{2}}\\
                                    &(1-|\Phi_a(z)|^2)^{\frac{ns}{2}} (1-|\Phi_a(w)|^2)^{\frac{ns}{2}}dV(w)\bigg)d\lambda(z)<\infty;
\end{split} \]
\end{enumerate}
\end{thm}

\begin{proof}
 (3) $\Longrightarrow$ (2).  This implication is obvious.

 (2) $\Longrightarrow$ (1).  When $z \in D(a, r), a, z \in \B$, we have (see, e.g., \cite{Zhu})
\begin{equation} \label{Charac11}
(1-|z|^2)^{n+1} \simeq (1-|a|^2)^{n+1} \simeq |1-\langle a, z \rangle|^{n+1} \simeq V(D(a, r))
\end{equation}
as well as (see, e.g., \cite[(2.20)]{Zhu})
\begin{equation} \label{Charac12}
|1-\langle z, u \rangle| \simeq |1-\langle a, u \rangle|, \ \forall u \in \B.
\end{equation}

By the inequality
$$
(1-|z|^2)|\nabla f(z)| \lesssim \frac{1}{V(D(z, r))} \int_{D(z, r)} |f(z)-f(w)|dV(w), \forall z \in \B,
$$
(see, e.g., \cite{LH}), we have
\begin{eqnarray*}
&&|\nabla f(z)|^p (1-|z|^2)^{p+q} (1-|\Phi_a(z)|^2)^{ns}\\
&\lesssim &\left( \frac{1}{V(D(z, r))} \int_{D(z, r)} |f(z)-f(w)| (1-|z|^2)^{\frac{q}{p}}(1-|\Phi_a(z)|^2)^{\frac{ns}{p}}dV(w) \right)^p \\
&\simeq& \bigg( \frac{1}{V(D(z, r))}  \int_{D(z, r)}  |f(z)-f(w)|(1-|z|^2)^{\frac{q}{2p}} (1-|w|^2)^{\frac{q}{2p}}\\
&&  \quad (1-|\Phi_a(z)|^2)^{\frac{ns}{2p}} (1-|\Phi_a(w)|^2)^{\frac{ns}{2p}}dV(w)\bigg)^p.
\end{eqnarray*}
Integrating with respect to $z$ over $\B$ on both sides and taking the supremum over $a$, we get
$$
\sup_{a \in \B} \int_\B |\nabla f(z)|^p (1-|z|^2)^{p+q} (1-|\Phi_a(z)|^2)^{ns}d\lambda(z),
$$
which implies $f \in \calN(p, q, s)$.

 (1) $\Longrightarrow$ (4).
Indeed, for this assertion, we can show that for each $1 \le c < \frac{1}{r}$, (4) is satisfied. Take and fix some $c \in \left(1, \frac{1}{r}\right)$. From Lemma \ref{mixedlem} and the fact that
$$
|\widetilde{\nabla}f(z)|^2=(1-|z|^2)(|\nabla f(z)|^2-|Rf(z)|^2),
$$
we have
$$
\int_{D(z, cr)} |f(w)|^pdV(w) \lesssim \int_{D(z, cr)} (1-|w|^2)^p |\nabla f(w)|^pdV(w)+|f(z)|^p.
$$
Hence, we have
\begin{eqnarray*}
&&\int_\B \bigg( \frac{1}{V(D(z, cr))}  \int_{D(z, cr)}  |f(z)-f(w)|^p|(1-|z|^2)^{\frac{q}{2}}(1-|w|^2)^{\frac{q}{2}}\\
&&   \quad (1-|\Phi_a(z)|^2)^{\frac{ns}{2}} (1-|\Phi_a(w)|^2)^{\frac{ns}{2}}dV(w)\bigg)d\lambda(z)\\
&\lesssim& \int_\B \bigg( \frac{1}{V(D(z, cr))} \int_{D(z, cr)} (1-|w|^2)^{p+q}|\nabla f(w)|^p (1-|\Phi_a(w)|^2)^{ns}dV(w) \bigg)d\lambda(z)\\
&\lesssim & \int_\B \int_\B \chi_{D(z, cr)}(w)|\nabla f(w)|^p(1-|w|^2)^{p+q} (1-|\Phi_a(w)|^2)^{ns}d\lambda(w)d\lambda(z)\\
& =&\int_\B \chi_{D(w, cr)}(z) \left(\int_\B |\nabla f(w)|^p (1-|w|^2)^{p+q}(1-|\Phi_a(w)|^2)^{ns}d\lambda(w)\right)d\lambda(z) \\
&\lesssim & \|f\|^p<\infty.
\end{eqnarray*}

 (4) $\Longrightarrow$ (3).
Suppose there exists some $c$ satisfying (4). For any $f \in H(\B)$, by the subharmonicity and \cite[Proposition 1.21 and Lemma 2.20]{Zhu}, for any $z \in \B$ and $w \in D(z, r)$, {\it i.e.} $|\Phi_z(w)|<r$, we have
\begin{eqnarray*}
&&|f(z)-f(w)|^p=|(f \circ\Phi_z)(\Phi_z(w))-(f\circ \Phi_z)(0)|^p\\
&\lesssim& \int_{\{u \in \B: |u-\Phi_z(w)|<(c-1)r\}} |(f \circ \Phi_z)(u)-(f \circ \Phi_z)(0)|^pdV(u)\\
&\le& \int_{|u| \le cr} |(f\circ\Phi_z)(u)-f(z)|^pdV(u)\\
&& \quad (\textrm{since} \ |\Phi_z(w)|<r)\\
&= & \int_{|\Phi_z(\zeta)|<cr} |f(\zeta)-f(z)|^p \frac{(1-|z|^2)^{n+1}}{|1-\langle z, \zeta \rangle|^{2(n+1)}}dV(\zeta)\\
&&  \quad (\textrm{change variables with} \ \zeta=\Phi_z(u)) \\
&\simeq& \frac{1}{V(D(z, cr)} \int_{D(z, cr)} |f(\zeta)-f(z)|^pdV(\zeta).
\end{eqnarray*}
Hence, by the above inequality, we have
\begin{eqnarray*}
&&  \int_\B \big( \sup_{w \in D(z, r)}  |f(z)-f(w)|(1-|z|^2)^{\frac{q}{2p}}(1-|w|^2)^{\frac{q}{2p}}\\
&&  \quad \quad \quad \quad (1-|\Phi_a(z)|^2)^{\frac{ns}{2p}}(1-|\Phi_a(w)|^2)^{\frac{ns}{2p}}\big)^pd\lambda(z) \\
&\simeq& \int_\B (1-|z|^2)^{\frac{q}{p}}(1-|\Phi_a(z)|^2)^{\frac{ns}{p}} \left(\sup_{w \in D(z, r)} |f(z)-f(w)|^p\right)d\lambda(z)
\end{eqnarray*}
\begin{eqnarray*}
&\lesssim& \int_\B \frac{1}{V(D(z, cr))} \int_{D(z, cr)} |f(\zeta)-f(z)|^p (1-|z|^2)^{\frac{q}{p}}\times\\
&& (1-|\Phi_a(z)|^2)^{\frac{ns}{p}}dV(\zeta) d\lambda(z)\\
&\simeq& \int_\B \bigg( \frac{1}{V(D(z, cr))}  \int_{D(z, cr)}  |f(z)-f(\zeta)|^p|(1-|z|^2)^{\frac{q}{2}}(1-|\zeta|^2)^{\frac{q}{2}}\\
&&   \quad \quad \quad \quad  (1-|\Phi_a(z)|^2)^{\frac{ns}{2}} (1-|\Phi_a(\zeta)|^2)^{\frac{ns}{2}}dV(\zeta)\bigg)d\lambda(z)\\
&<&\infty,
\end{eqnarray*}
which implies the desired result.
\end{proof}

From  Theorem \ref{finter}, \eqref{Charac11} and \eqref{Charac12}, we easily get the following result.

\begin{thm} \label{finter01}
Suppose $f \in H(\B), p \ge 1, q>0, s>\max\left\{0, 1-\frac{q}{n}\right\}$ and $0<r<1$. Then the following statements are equivalent:
\begin{enumerate}
\item $f \in \calN(p, q, s)$;
\item
\[ \begin{split}
\sup_{a \in \B} \int_\B \bigg( &\frac{1}{V(D(z, r))}  \int_{D(z, r)} |f(z)-f(w)|(1-|z|^2)^{\frac{q}{p}}\\
                                               &(1-|\Phi_a(z)|^2)^{\frac{ns}{p}}dV(w)\bigg)^pd\lambda(z)<\infty;
\end{split} \]
\item
\[ \begin{split}
\sup_{a \in \B} \int_\B \big( &\sup_{w \in D(z, r)}  |f(z)-f(w)|(1-|z|^2)^{\frac{q}{p}}(1-|\Phi_a(z)|^2)^{\frac{ns}{p}}\big)^pd\lambda(z)<\infty;
\end{split} \]
\item There exists some $c$ satisfying $1<c<\frac{1}{r}$, such that
\[ \begin{split}
\sup_{a \in \B} \int_\B &\bigg( \frac{1}{V(D(z, cr))}  \int_{D(z, cr)}  |f(z)-f(w)|^p|(1-|z|^2)^q\\
                                    &(1-|\Phi_a(z)|^2)^{ns}\bigg)d\lambda(z)<\infty.
\end{split} \]
\end{enumerate}
\end{thm}

\begin{rem}
Our earlier estimates in Theorems \ref{mix01},  \ref{mix02},   \ref{finter} and  \ref{finter01} are pointwise estimates with respect to $a \in B$, so if we replace $\sup_{a \in \B}$ and $<\infty$ by $\lim_{|a| \to 1}$ and $=0$, respectively, we obtain the corresponding characterizations of $\calN^0(p, q, s)$. Hence, we omit the details of the proof.
\end{rem}


\medskip

\section{Atomic decomposition and Gleason's problem for $\calN(p, q, s)$-type spaces}


In this section, we will focus on the decomposition of functions in $\calN(p, q, s)$-type spaces, which is an important concept and is a useful tool in studying such kind of function spaces.


\subsection{Atomic decomposition}


First, we give some preliminaries. Recall that for $z, w \in \B$, the Bergman metric can be written as (see, e.g. \cite[Proposition 1.21]{Zhu})
$$
\beta(z, w)=\frac{1}{2} \log \frac{1+|\Phi_z(w)|}{1-|\Phi_z(w)|}.
$$
  Moreover, for $r>0$ and $z \in \B$, the set
$E(z, r)=\{w \in \B: \beta(z, w)<r\}$ is a Bergman metric ball at $z$. Note that by a simple calculation, we have
$$
|\Phi_z(w)|=\tanh \beta(z, w),  \ z, w \in \B,
$$
which implies that $E(z, r)=D(z, \tanh r)$.

\begin{lem} \cite[Theorem 2.23]{Zhu} \label{decom01}
There exists a positive integer $N$ such that for any $0<r \le 1$ we can find a sequence $\{a_k\}$ in $\B$ with the following properties:
\begin{enumerate}
\item $\B=\bigcup_k E(a_k ,r)$;
\item The sets $E(a_k, r/4)$ are mutually disjoint;
\item Each point $z \in \B$ belongs to at most $N$ of the sets $E(a_k, 4r)$.
\end{enumerate}
\end{lem}

\begin{lem} \cite[Lemma 2.28]{Zhu} \label{decom02}
Take and fix a sequence $\{a_k\}$ chosen according to Lemma \ref{decom01} with $r$ the separation constant. Then for each $k \ge 1$ there exists a Borel set $E_k$ satisfying the following conditions:
\begin{enumerate}
\item $E(a_k, r/4) \subset E_k \subset E(a_k, r)$ for every $k$;
\item $E_k \cap E_j=\emptyset$ for $k \neq j$;
\item $\B=\bigcup_k E_k$.
\end{enumerate}
\end{lem}

\begin{lem} \label{decom03}
Suppose $p \ge 1, q>0, s>\max\left\{1, 1-\frac{q}{n}\right\}$ and $\{a_k\} \subset \B$ is a chosen sequence according to Lemma \ref{decom01} with the separation constant $r \in (0, 1]$. Then
$$
d\mu_1=\sum_k |c_k|^p (1-|a_k|^2)^{q+ns}\del_{a_k} dV(z)
$$
is an $(ns)$-Carleson measure if and only if
$$
d\mu_2=\sum_k \frac{|c_k|^p}{(1-|a_k|^2)^p} (1-|z|^2)^{p+q+ns}\chi_k(z) d\lambda(z)
$$
is an $(ns)$-Carleson measure, where $\chi_k$ is the characteristic function of $E(a_k, r)$.
\end{lem}

\begin{proof}
For any $a \in \B$, by \cite[Lemma 2.20, (2.20)]{Zhu} and \cite[2.2.7]{Rud}, we have
\begin{eqnarray*}
&&\int_\B \left(\frac{1-|a|^2}{|1-\langle z, a \rangle|^2}\right)^{ns}d\mu_2(z)\\
&= &\int_\B \left(\frac{1-|a|^2}{|1-\langle z, a \rangle|^2}\right)^{ns} \left(\sum_k \frac{|c_k|^p}{(1-|a_k|^2)^p} (1-|z|^2)^{p+q+ns}\chi_k(z)\right) d\lambda(z)\\
&= &\sum_k \frac{|c_k|^p}{(1-|a_k|^2)^p} \int_{D(a_k, r)} \frac{(1-|a|^2)^{ns}(1-|z|^2)^{p+q+ns-n-1}}{|1-\langle z,a \rangle|^{2ns}}dV(z)\\
&\simeq& \sum_k \frac{|c_k|^p}{(1-|a_k|^2)^{n+1-q-ns}} \cdot \frac{(1-|a|^2)^{ns}}{|1-\langle a_k, a \rangle|^{2ns}} V(D(a_k, r))
 \end{eqnarray*}
\begin{eqnarray*}
&\simeq& \sum_k |c_k|^p (1-|a_k|^2)^{q+ns} \cdot \frac{(1-|a|^2)^{ns}}{|1-\langle a_k, a \rangle|^{2ns}}\\
& =&\int_\B \left( \frac{1-|a|^2}{|1-\langle z, a\rangle|^2}\right)^{ns}d\mu_1(z).
\end{eqnarray*}
The desired result follows from \cite[Theorem 45]{ZZ}.
\end{proof}

Fix a parameter $b>n$ and let $\alpha=b-(n+1)$. We also fix a sequence $\{a_k\}$ chosen according to Lemma \ref{decom01} with separation constant $r$ and a sequence of Borel measurable sets $\{E_k\}$ with each $E_k$ satisfying the condition in Lemma \ref{decom02}. Recall that the operator $T$ associated to $\{a_k\}$ is as follows:
\begin{equation} \label{decomeq01}
T(f)(z)=\int_\B \frac{(1-|w|^2)^{b-n-1}}{|1-\langle z, w \rangle|^b}f(w)dV(w),
\end{equation}
where $f$ is some Lebesgue measurable function.

Moreover, take a finer ``lattice"  $\{a_{kj}\}$ with separation constant $\gamma$ in the Bergman metric than $\{a_k\}$ with a sequence of Borel measurable sets $\{E_{kj}\}$ chosen according to \cite[Page 64]{Zhu} and define
\begin{equation} \label{decomeq02}
S(f)(z)=\sum_{k, j} \frac{V_\alpha(E_{kj})f(a_{kj})}{(1-\langle z, a_{kj} \rangle)^b},
\end{equation}
where $f \in H(\B)$. Note that $\{a_{kj}\}$ also satisfies the conditions in Lemma \ref{decom01}. We refer the reader to the excellent book \cite{Zhu} for the detailed information about such a decomposition of $\mathbb{B}$ into Bergman metric balls.

The following result indicates a deep relationship between $T$ and $S$.

\begin{lem} \cite[Lemma 3.22]{Zhu}\label{decom10}
There exists a constant $C>0$, independent of the separation constant $r$ for $\{a_k\}$ and the separation constant $\gamma$ for $\{a_{kj}\}$, such that
$$
|f(z)-S(f)(z)| \le C\sigma T(|f|)(z)
$$
for all $f \in H(\B)$ and $z \in \B$, where $\sigma=\gamma+\frac{\tanh(\gamma)}{\tanh(r)}$.
\end{lem}

 We have the following result considering the behaviour of $T$, which follows the methods in \cite[Theorem 5.26]{Zhu} and \cite[Lemma 1]{RO}.

\begin{lem}  \label{decom04}
Let $f$ be some Lebesgue measurable function. Suppose $p \ge 1, q>0, s>\max\left\{0, 1-\frac{q}{n} \right\}$ and $t>n-p-q-ns$. Then we have
\begin{enumerate}
\item
If $p>1, b>\frac{n+1}{p'}+\frac{q+ns+t}{p}+1$ and $|f(z)|^p(1-|z|^2)^{p+q+ns+t}d\lambda(z)$ is an $(ns)$-Carleson measure, where $p'$ is the conjugate of $p$, then
$$
|T(f)(z)|^p(1-|z|^2)^{p+q+ns+t}d\lambda(z)
$$
is also an $(ns)$-Carleson measure.

\item
If $p=1, b>1+q+ns+t$ and $|f(z)|(1-|z|^2)^{1+q+ns+t}d\lambda(z)$ is an $(ns)$-Carleson measure, then
$$
 |T(f)(z)|(1-|z|^2)^{1+q+ns+t}d\lambda(z)
$$
is also an $(ns)$-Carleson measure.
\end{enumerate}
\end{lem}

\begin{proof}
The proof of (2) is a simple modification of (1), which turns out to be much more easier than (1), and hence we omit it here.

For any $\xi \in \SSS$ and $\del \in (0, 1)$, we have
\begin{eqnarray*}
&&\frac{1}{\del^{ns}} \int_{Q_\del(\xi)} |T(f)(z)|^p(1-|z|^2)^{p+q+ns+t}d\lambda(z)\\
&\le& \frac{1}{\del^{ns}} \int_{Q_\del(\xi)} \left( \int_\B \frac{(1-|w|^2)^{b-n-1}}{|1-\langle z, w \rangle|^b}|f(w)|dV(w) \right)^p (1-|z|^2)^{p+q+ns+t}d\lambda(z)\\
&=&\frac{1}{\del^{ns}} \int_{Q_\del(\xi)} \left( \left[ \int_{Q_{2\del}(\xi)}+\sum_{j=1}^\infty \int_{A_j}\right] \frac{(1-|w|^2)^{b-n-1}}{|1-\langle z, w \rangle|^b}|f(w)|dV(w) \right)^p\\
&&  (1-|z|^2)^{p+q+ns+t}d\lambda(z)\\
&\le& \frac{2^{p-1}}{\del^{ns}}\int_{Q_\del(\xi)} \left( \int_{Q_{2\del}(\xi)} \frac{(1-|w|^2)^{b-n-1}}{|1-\langle z, w \rangle|^b}|f(w)|dV(w) \right)^p \\
&& (1-|z|^2)^{p+q+ns+t}d\lambda(z)\\
&&   + \frac{2^{p-1}}{\del^{ns}}\int_{Q_\del(\xi)} \left(\sum_{j=1}^\infty \int_{A_j} \frac{(1-|w|^2)^{b-n-1}}{|1-\langle z, w \rangle|^b}|f(w)|dV(w) \right)^p \\
&& (1-|z|^2)^{p+q+ns+t}d\lambda(z)\\
&=& I_1+I_2,
\end{eqnarray*}
where $A_j=\left\{w \in \B: 2^j\del \le |1-\langle w, \xi \rangle|<2^{j+1}\del \right\}, i=1, 2, \dots$.

\textbf{$\bullet$ Estimation of $I_1$.}

Consider the integral operator $M$ induced by $K(z, w)$,
$$
Mh(z)=\int_\B K(z, w)h(w)dV(w), z \in \B,
$$
where
$$
K(z, w)=\frac{(1-|w|^2)^{b-\frac{n+1}{p'}-\frac{q+ns+t}{p}-1}(1-|z|^2)^{1+\frac{q+ns+t-n-1}{p}}}{|1-\langle z, w \rangle|^b}, \ z, w \in \B.
$$
We claim that $M$ is a bounded operator on $L^p(\B, dV)$. Indeed, consider the function $g(z)=(1-|z|^2)^{-\frac{1}{p+p'}}$. On one hand, we have
\begin{eqnarray*}
&& \int_\B K(z, w)g^{p'}(w)dV(w)\\
& = &(1-|z|^2)^{1+\frac{q+ns+t-n-1}{p}}\int_\B \frac{(1-|w|^2)^{b-\frac{n+1}{p'}-\frac{q+ns+t}{p}-1-\frac{p'}{p+p'}}}{|1-\langle z, w \rangle|^b} dV(w)\\
&= & (1-|z|^2)^{1+\frac{q+ns+t-n-1}{p}}\int_\B \frac{(1-|w|^2)^{b-\frac{n+1}{p'}-\frac{q+ns+t}{p}-1-\frac{p'}{p+p'}}}{|1-\langle z, w \rangle|^{n+1+\left(b-\frac{n+1}{p'}-\frac{q+ns+t}{p}-1-\frac{p'}{p+p'}\right)+c_1}} dV(w),
\end{eqnarray*}
where $c_1=1+\frac{q+ns+t-n-1}{p}+\frac{p'}{p+p'}$.  Note that by our assumption,
$$
b-\frac{n+1}{p'}-\frac{q+ns+t}{p}-1-\frac{p'}{p+p'}>-1
$$
and
$$
c_1=1+\frac{q+ns+t-n-1}{p}+\frac{p'}{p+p'}>0,
$$
and hence by \cite[Proposition 1.4.10]{Rud},
\begin{eqnarray*}
&&\int_\B K(z, w)g^{p'}(w)dV(w)\\
&\lesssim & (1-|z|^2)^{1+\frac{q+ns+t-n-1}{p}} \cdot (1-|z|^2)^{-1-\frac{q+ns+t-n-1}{p}-\frac{p'}{p+p'}}=g^{p'}(z).
\end{eqnarray*}

On the other hand, we have
\begin{eqnarray*}
&&\int_\B K(z, w) g^p(z)dV(z)\\
&= &(1-|w|^2)^{b-\frac{n+1}{p'}-\frac{q+ns+t}{p}-1}\int_\B \frac{(1-|z|^2)^{1+\frac{q+ns+t-n-1}{p}-\frac{p}{p+p'}}}{|1-\langle z, w \rangle|^b}dV(z)\\
&=&(1-|w|^2)^{b-\frac{n+1}{p'}-\frac{q+ns+t}{p}-1}\int_\B \frac{(1-|z|^2)^{1+\frac{q+ns+t-n-1}{p}-\frac{p}{p+p'}}}{|1-\langle z, w \rangle|^{n+1+\left(1+\frac{q+ns+t-n-1}{p}-\frac{p}{p+p'}\right)+c_2}}dV(z),
\end{eqnarray*}
where $c_2=b+\frac{p}{p+p'}-\frac{n+1}{p'}-1-\frac{q+ns+t}{p}$. By our assumption again, it follows that
$$
1+\frac{q+ns+t-n-1}{p}-\frac{p}{p+p'}>-1
$$
and
$$
c_2=b+\frac{p}{p+p'}-\frac{n+1}{p'}-1-\frac{q+ns+t}{p}>0.
$$
Thus, we have
\begin{eqnarray*}
&&\int_\B K(z, w) g^p(z)dV(z)\\
&\lesssim& (1-|w|^2)^{b-\frac{n+1}{p'}-\frac{q+ns+t}{p}-1} \cdot (1-|w|^2)^{-b-\frac{p}{p+p'}+\frac{n+1}{p'}+1+\frac{q+ns+t}{p}}=g^p (w).
\end{eqnarray*}

The boundedness of $M$ on $L^p(\B, dV)$ is  clear by Schur's test (see, e.g., \cite[Theorem 2.9]{Zhu}). Put
$$
h(w)=|f(w)| (1-|w|^2)^{1+\frac{q+ns+t-n-1}{p}}\chi_{Q_{2\del}(\xi)}(w), \ w \in \B.
$$
It is easy to see $h \in L^p(\B, dV)$ since $|f(z)|^p(1-|z|^2)^{p+q+ns+t}d\lambda(z)$ is an $(ns)$-Carleson measure. Moreover, we have
$$
\frac{1}{(2\del)^{ns}} \|h\|_{L^p}^p=\frac{1}{(2\del)^{ns}} \int_{Q_{2\del}(\xi)}|f(w)|^p(1-|w|^2)^{p+q+ns+t}d\lambda(w)<\infty.
$$
Therefore,
\begin{eqnarray*}
I_1%
&\le& \frac{2^{p-1}}{\del^{ns}} \int_\B \left( \int_\B K(z, w)h(w)dV(w) \right)^pdV(z)\\
&=& \frac{2^{p-1}}{\del^{ns}} \|Mh\|_{L^p}^p \lesssim \frac{1}{\del^{ns}} \|h\|_{L^p}^p<\infty.
\end{eqnarray*}

\textbf{$\bullet$ Estimation of $I_2$.}

Note that for $z \in Q_\del(\xi)$ and $w \in A_j$, we have
$$
|1-\langle z, w \rangle|^{\frac{1}{2}} \ge |1-\langle \xi, w\rangle|^{\frac{1}{2}}-|1-\langle \xi, z \rangle|^{\frac{1}{2}}> \frac{1}{2} (\sqrt{2}-1) 2^{\frac{j}{2}}\del^{\frac{1}{2}}.
$$
Moreover, for each $j \ge 1$, consider the term
$$
I_{3, j}=\int_{Q_{2^{j+1}\del}(\xi)} |f(w)|(1-|w|^2)^{b-n-1}dV(w).
$$
Using H\"older's inequality and \cite[Corollary 5.24]{Zhu}, we have
\begin{eqnarray*}
I_{3, j}%
&\le& \left( \int_{Q_{2^{j+1}\del}(\xi)} |f(w)|^p(1-|w|^2)^{p+q+ns+t}d\lambda(z) \right)^{\frac{1}{p}} \\
&& \times \left( \int_{Q_{2^{j+1}\del}(\xi)} (1-|w|^2)^{p'(b-n-1)-p'\left(1+\frac{q+ns+t-n-1}{p}\right)}dV(w) \right)^{\frac{1}{p'}}\\
&\lesssim& (2^{j+1}\del)^{b-1-\frac{q+ns+t}{p}} (2^{j+1}\del)^{\frac{ns}{p}}\\
&& \times \left( \frac{1}{(2^{j+1}\del)^{ns}} \int_{Q_{2^{j+1}\del}(\xi)} |f(w)|^p(1-|w|^2)^{p+q+ns+t}d\lambda(z) \right)^{\frac{1}{p}}\\
&\lesssim& (2^{j+1}\del)^{b-1-\frac{q+t}{p}}.
\end{eqnarray*}

Thus,  we have
\begin{eqnarray*}
I_2%
&\lesssim& \frac{1}{\del^{ns}} \int_{Q_\del(\xi)} \left( \sum_{j=1}^\infty \frac{1}{(2^j\del)^b} \int_{Q_{2^{j+1}\del}(\xi)} |f(w)|(1-|w|^2)^{b-n-1}dV(w) \right)^p\\
&& \quad \quad (1-|z|^2)^{p+q+ns+t-n-1}dV(z)\\
&\simeq& \del^{p+q+t} \left(\sum_{j=1}^\infty \frac{1}{(2^j\del)^b} \int_{Q_{2^{j+1}\del}(\xi)} |f(w)|(1-|w|^2)^{b-n-1}dV(w) \right)^p\\
&\lesssim& \del^{p+q+t} \left( \sum_{j=1}^\infty \frac{1}{(2^j \del)^b} (2^{j+1}\del)^{b-1-\frac{q+t}{p}} \right)^p \lesssim \left( \sum_{j=1}^\infty \frac{1}{2^{j\left(1+\frac{q+t}{p}\right)}} \right)^p<\infty.
\end{eqnarray*}

Finally, combining the estimations on $I_1$ and $I_2$, it is clear that for any $\xi \in \SSS$ and $\del \in (0, 1)$,
$$
\frac{1}{\del^{ns}} \int_{Q_\del(\xi)} |T(f)(z)|^p(1-|z|^2)^{p+q+ns+t}d\lambda(z)<\infty,
$$
which implies $|T(f)(z)|^p(1-|z|^2)^{p+q+ns+t}d\lambda(z)$ is an $(ns)$-Carleson measure.
\end{proof}

We are now ready to establish our main results in this section.

\begin{thm} \label{Atomicdecom}
Suppose $p \ge 1, q>0, s>\max\left\{0, 1-\frac{q}{n} \right\}$ and
\[ b>\begin{cases}
\frac{n+1}{p'}+\frac{q+ns}{p}, & p>1;\\
q+ns, & p=1.
\end{cases} \]

\begin{enumerate}
\item Let $\{a_k\}$ be a sequence satisfying the conditions in Lemma \ref{decom01} with the separation constant $r \in (0, 1)$. If $\{c_k\}$ is a sequence such that the measure $\sum_k |c_k|^p(1-|a_k|^2)^{q+ns}\del_{a_k}$ is an $(ns)$-Carleson measure, then the function
$$
f(z)=\sum_k c_k \left( \frac{1-|a_k|^2}{1-\langle z, a_k \rangle} \right)^b
$$
belongs to $\calN(p, q, s)$.
\item There exists a sequence $\{a_k\}$ in $\B$ such that $\calN(p, q, s)$-type spaces consists exactly of function of the form
$$
f(z)=\sum_k c_k \left( \frac{1-|a_k|^2}{1-\langle z, a_k \rangle} \right)^b,
$$
where the sequence $\{c_k\}$ has the property that $\sum_k |c_k|^p(1-|a_k|^2)^{q+ns}\del_{a_k}$ is an $(ns)$-Carleson measure.
\end{enumerate}
\end{thm}

\begin{proof}
Without the loss of generality, we may assume $p>1$.

(1) For each $k \ge 1$, let $E_k=E\left(a_k, \frac{r}{4}\right)$. Consider the function
$$
u(z)=\sum_{k=1}^\infty \frac{|c_k|\chi_k(z)}{1-|a_k|^2},
$$
where $\chi_k$ is the characteristic function of the set $E_k$. By Lemma \ref{decom01}, the sets $E_k$ are mutually disjoint, the measure $|u(z)|^p(1-|z|^2)^{p+q+ns}d\lambda(z)$ can be written as
\begin{eqnarray*}
&&\left| \sum_{k=1}^\infty \frac{|c_k|\chi_k(z)}{1-|a_k|^2}\right|^p (1-|z|^2)^{p+q+ns}d\lambda(z)\\
&= &\sum_{k=1}^\infty \frac{|c_k|^p}{(1-|a_k|^2)^p} (1-|z|^2)^{p+q+ns}\chi_k(z)d\lambda(z),
\end{eqnarray*}
which by Lemma \ref{decom03}, is an $(ns)$-Carleson measure.

Let $T$ be the operator with the parameter $b+1$. By Lemma \ref{decom04}, since $|u(z)|^p(1-|z|^2)^{p+q+ns}d\lambda(z)$ is an $(ns)$-Carleson measure, it follows that $|T(u)(z)|^p(1-|z|^2)^{p+q+ns}d\lambda(z)$ is also an $(ns)$-Carleson measure.

From the proof of \cite[Lemma 5.28]{Zhu}, we know that $|Rf(z)| \lesssim T(u)(z)$, which implies that
$$
|Rf(z)|^p(1-|z|^2)^{p+q+ns}d\lambda(z)
$$
is also an $(ns)$-Carleson measure. The desired result follows from Theorem \ref{Carlesoncharac}.

(2) Let $X$ be the function space consist $f \in H(\B)$ satisfying
$$
 \|f\|_X^p=\sup_{0<\del<1, \xi \in \SSS} \frac{1}{\del^{ns}} \int_{Q_\del(\xi)} |f(z)|^p(1-|z|^2)^{p+q+ns}d\lambda(z)<\infty.
$$
It is easy to check that $X$ becomes a Banach space when equipped with the norm $\|\cdot\|_X$ (see, e.g., \cite[Page 185]{Zhu}). Hence, by Lemma \ref{decom04}, the operator $T$ defined in \eqref{decomeq01} with parameter $b+1$ is bounded on $X$.

Take and fix a sequence $\{b_k\}$ satisfying the condition in Lemma \ref{decom01} with separation constant $r$ and a finer ``lattice" $\{b_{kj}\}$ with separation constant $\gamma$ satisfying that
$$
C\sigma \|T\|<1,
$$
where $C$ and $\sigma$ are defined in Lemma \ref{decom10}. Let $S$ be the linear operator defined in \eqref{decomeq02} with the parameter $b+1$ associated to $\{b_{kj}\}$. Thus, by Lemma \ref{decom10} again, we can get
$$
\|f-Sf\|_X \le C\sigma \|T(|f|)\|_X \le C\sigma \|T\| \|f\|_X<\|f\|_X,
$$
which implies the operator $I-S$ is bounded on $X$ with operator norm strictly less than $1$, where $I$ is the identity operator. Hence, by \cite[Theorem 1.5.2]{AW}, $S$ is invertible on $X$.

Fix $f \in \calN(p, q, s)$ and let $g=R^{\alpha, 1}f$, where $\alpha=b-(n+1)$ and $R^{\alpha, 1}$ is a linear partial differential operator of order $1$ (see, e.g., \cite[Proposition 1.15]{Zhu}). By \cite[Proposition 1.15]{Zhu} and Theorem \ref{Carlesoncharac}, $g \in X$.

Since $S$ is invertible on $X$, there exists a function $h \in X$ such that $g=Sh$. Thus $g$ admits the representation
$$
g(z)=\sum_{k, j} \frac{V_\beta(E_{kj})h(b_{kj})}{(1-\langle z, b_{kj}\rangle)^{b+1}}.
$$
where $\beta=(b+1)-(n+1)=b-n$. Applying the inverse of $R^{\alpha, 1}$ to both sides with \cite[Proposition 1.14]{Zhu}, we obtain
$$
f(z)=\sum_{k, j}  \frac{V_\beta(E_{kj})h(b_{kj})}{(1-\langle z, b_{kj}\rangle)^b}.
$$
Let
$$
c_{kj}=\frac{V_\beta(E_{kj})h(b_{kj})}{(1-|b_{kj}|^2)^b}, \ k \ge 1, 1 \le j \le J,
$$
where $J$ is some integer depending on $\gamma$ (see, e.g., \cite[Page 64]{Zhu}) and hence we can write
$$
f(z)=\sum_{k, j}  c_{kj} \left( \frac{1-|b_{kj}|^2}{1-\langle z, b_{kj} \rangle} \right)^b.
$$
It remains for us to show that the measure
$$
\sum_{k, j} |c_{kj}|^p(1-|b_{kj}|^2)^{q+ns}\del_{b_{kj}}
$$
is an $(ns)$-Carleson measure. Since
$$
V_\beta(E_{kj}) \le V_\beta(E_k) \simeq (1-|b_k|^2)^{n+1+\beta}=(1-|b_k|^2)^{b+1} \simeq (1-|b_{kj}|^2)^{b+1},
$$
where the last estimation follows from the fact that $b_{kj} \in E(b_k, r), 1 \le j \le J$, it suffices to show that the measure
$$
d\mu=\sum_{k, j} (1-|b_{kj}|^2)^{p+q+ns}|h(b_{kj})|^p\del_{b_{kj}}
$$
is an $(ns)$-Carleson measure.

By Lemma \ref{decom02}, we know that $E(b_{kj}, \frac{\gamma}{4})$ are mutually disjoint. Using \cite[Lemma 2.20, (2.20) and Lemma 2.24]{Zhu}, we have for any $a \in \B$,
\begin{eqnarray*}
&&\int_\B \left(\frac{1-|a|^2}{|1-\langle z, a \rangle|^2} \right)^{ns}d\mu(z)\\
&= &\sum_{k, j} \left( \frac{1-|a|^2}{|1-\langle b_{kj}, a\rangle|^2} \right)^{ns} |h(b_{kj})|^p (1-|b_{kj}|^2)^{p+q+ns}\\
&\lesssim & \sum_{k, j}  \left( \frac{1-|a|^2}{|1-\langle b_{kj}, a\rangle|^2} \right)^{ns} \int_{E(b_{kj}, \frac{\gamma}{4})} |h(z)|^p(1-|z|^2)^{p+q+ns}d\lambda(z)\\
&\simeq & \sum_{k, j}  \int_{E(b_{kj}, \frac{\gamma}{4})}\left(\frac{1-|a|^2}{|1-\langle z, a \rangle|^2}\right)^{ns}|h(z)|^p(1-|z|^2)^{p+q+ns}d\lambda(z)\\
&\le& \int_\B \left(\frac{1-|a|^2}{|1-\langle z, a \rangle|^2}\right)^{ns}|h(z)|^p(1-|z|^2)^{p+q+ns}d\lambda(z)\\
&<&\infty,
\end{eqnarray*}
where the last inequality follows from the fact that $f \in X$ and \cite[Theorem 45]{ZZ}. Using \cite[Theorem 45]{ZZ} again, we conclude that the measure $\mu$ is an $(ns)$-Carleson measure. The proof is complete.
\end{proof}

Similarly, we have the following description for $\calN^0(p, q, s)$ with regarding its atomic decomposition. First we observe that we have the following ``little" version for Lemma \ref{decom04}.

\begin{lem}  \label{decom06}
Let $f$ be some Lebesgue measurable function. Suppose $p \ge 1, q>0, s>\max\left\{0, 1-\frac{q}{n} \right\}$ and $t>n-p-q-ns$. Then we have
\begin{enumerate}
\item
If $p>1, b>\frac{n+1}{p'}+\frac{q+ns+t}{p}+1$ and $|f(z)|^p(1-|z|^2)^{p+q+ns+t}d\lambda(z)$ is a vanishing $(ns)$-Carleson measure, where $p'$ is the conjugate of $p$, then
$$
|T(f)(z)|^p(1-|z|^2)^{p+q+ns+t}d\lambda(z)
$$
is also a vanishing $(ns)$-Carleson measure.

\item
If $p=1, b>1+q+ns+t$ and $|f(z)|(1-|z|^2)^{1+q+ns+t}d\lambda(z)$ is a vanishing $(ns)$-Carleson measure, then
$$
 |T(f)(z)|(1-|z|^2)^{1+q+ns+t}d\lambda(z)
$$
is also a vanishing $(ns)$-Carleson measure.
\end{enumerate}
\end{lem}

\begin{proof}
We would only consider the case $p>1$ again. By our assumption, for any $\varepsilon>0$, there exists a $\del_0>0$, such that the estimate
\begin{equation} \label{decomeq111}
\frac{1}{\del^{ns}}\int_{Q_\del(\xi)} |f(z)|^p(1-|z|^2)^{p+q+ns+t}d\lambda(z)<\varepsilon
\end{equation}
holds uniformly for $\xi \in \SSS$ when $\del<\del_0$. From the proof of Lemma \ref{decom04}, we have
$$
I_1 \lesssim \frac{1}{\del^{ns}} \int_{Q_{2\del}(\xi)}|f(z)|^p(1-|z|^2)^{p+q+ns+t}d\lambda(z),
$$
which, combining with \eqref{decomeq111}, implies when $\del<\frac{\del_0}{2}$, we have
$$
I_1 \lesssim 2^{ns}\varepsilon.
$$
Now we estimate $I_2$. For the chosen $\varepsilon$, there exists a $J_0 \in \N$, such that
$$
\sum_{j=J_0+1}^\infty \frac{1}{2^{j \left(1+\frac{q+1}{p} \right)}}<\varepsilon^{1/p}.
$$
From the proof of Lemma \ref{decom04}, we have
\begin{eqnarray*}
I_2%
&\lesssim& \frac{1}{\del^{ns}} \int_{Q_\del(\xi)} \left( \sum_{j=1}^\infty \frac{1}{(2^j\del)^b} \int_{Q_{2^{j+1}\del}(\xi)}|f(w)|(1-|w|^2)^{b-n-1}dV(w) \right)^p\\
&& \quad \quad (1-|z|^2)^{p+q+ns+t-n-1}dV(z)\\
&\lesssim& \frac{1}{\del^{ns}} \int_{Q_\del(\xi)} \left( \sum_{j=1}^{J_0} \frac{1}{(2^j\del)^b} \int_{Q_{2^{j+1}\del}(\xi)}|f(w)|(1-|w|^2)^{b-n-1}dV(w) \right)^p\\
&& \quad \quad (1-|z|^2)^{p+q+ns+t-n-1}dV(z)\\
&& + \frac{1}{\del^{ns}} \int_{Q_\del(\xi)} \left( \sum_{j=J_0+1}^\infty \frac{1}{(2^j\del)^b} \int_{Q_{2^{j+1}\del}(\xi)}|f(w)|(1-|w|^2)^{b-n-1}dV(w) \right)^p\\
&& \quad \quad (1-|z|^2)^{p+q+ns+t-n-1}dV(z)\\
&=&J_1+J_2.
\end{eqnarray*}

\textbf{$\bullet$ Estimation of $J_1$.}

 Take $\del<\frac{\del_0}{2^{J_0+1}}$. Then by \eqref{decomeq111} and the estimation of $I_{3, j}, j \ge 1$, we have for $j=1, 2, \dots, J_0$,
\begin{eqnarray*}
I_{3, j}%
&\lesssim& (2^{j+1}\del)^{b-1-\frac{q+t}{p}}  \left( \frac{1}{(2^{j+1}\del)^{ns}} \int_{Q_{2^{j+1}\del}(\xi)} |f(w)|^p(1-|w|^2)^{p+q+ns+t}d\lambda(z) \right)^{\frac{1}{p}}\\
&\le&  (2^{j+1}\del)^{b-1-\frac{q+t}{p}} \varepsilon^{\frac{1}{p}}.
\end{eqnarray*}
Thus, we see that
\begin{eqnarray*}
J_1%
&\lesssim& \frac{\varepsilon}{\del^{ns}} \int_{Q_\del(\xi)} \left(\sum_{j=1}^{J_0} \frac{1}{(2^j\del)^{1+\frac{q+t}{p}}}\right)^p (1-|z|^2)^{p+q+ns+t-n-1}dV(z)\\
&\lesssim& \varepsilon \cdot \left(\sum_{j=1}^{J_0} \frac{1}{2^{j\left(1+\frac{q+t}{p}\right)}}\right)^p \lesssim \varepsilon.
\end{eqnarray*}

\textbf{$\bullet$ Estimation of $J_2$.}

By the proof in Lemma \ref{decom04}, it is easy to see that
$$
J_2 \lesssim \left( \sum_{j=J_0+1}^\infty \frac{1}{2^{j\left(1+\frac{q+t}{p}\right)}} \right)^p<\varepsilon.
$$	

Hence, it follows that $I_2 \lesssim \varepsilon$, which implies the desired result.
\end{proof}

By using the last lemma, we get the following theorem considering the atomic decomposition on $\calN^0(p, q, s)$-type space, whose proof is straightforward from Theorem \ref{Atomicdecom}, and hence we omit the detail here.

\begin{thm} \label{Atomicdecom01}
Suppose $p \ge 1, q>0, s>\max\left\{0, 1-\frac{q}{n} \right\}$ and
\[ b>\begin{cases}
\frac{n+1}{p'}+\frac{q+ns}{p}, & p>1;\\
q+ns, & p=1.
\end{cases} \]

\begin{enumerate}
\item Let $\{a_k\}$ be a sequence satisfying the conditions in Lemma \ref{decom01} with the separation constant $r \in (0, 1)$. If $\{c_k\}$ is a sequence such that the measure $\sum_k |c_k|^p(1-|a_k|^2)^{q+ns}\del_{a_k}$ is a vanishing $(ns)$-Carleson measure, then the function
$$
f(z)=\sum_k c_k \left( \frac{1-|a_k|^2}{1-\langle z, a_k \rangle} \right)^b
$$
belongs to $\calN^0(p, q, s)$.
\item There exists a sequence $\{a_k\}$ in $\B$ such that $\calN^0(p, q, s)$-type spaces consists exactly of function of the form
$$
f(z)=\sum_k c_k \left( \frac{1-|a_k|^2}{1-\langle z, a_k \rangle} \right)^b,
$$
where the sequence $\{c_k\}$ has the property that $\sum_k |c_k|^p(1-|a_k|^2)^{q+ns}\del_{a_k}$ is a vanishing $(ns)$-Carleson measure.
\end{enumerate}
\end{thm}

We already see that the atomic decomposition on $\calN(p, q, s)$-type space is closely related to the choice of $\{a_k\}$. We say that a sequence $\{a_k\}$ of distinct points with in $\B$ is an \emph{$r$-lattice} in the Bergman metric if it satisfies Lemma \ref{decom01} with separation constant $r$. We need the following lemma.

\begin{lem} \label{atomicsample}
Let $\{z_n\}$ be a sequence of points on $\B$, $\alpha>-1$ and $f \in A^1_\alpha(\B) $. Let $\{a_n\}$ be an $r$-lattice in the Bergman metric. Then there exists a constant $C_1>0$ depending only on $r$ and $\alpha$ so that
\begin{equation} \label{sample001}
\|f\|_{1, \alpha} \ge C_1 \sum_n (1-|a_n|^2)^{n+1+\alpha}|f(a_n)|.
\end{equation}
Furthermore, there exists some $r_0>0$ and a constant $C_2>0$ depending only on $r$ and $\alpha$ so that
\begin{equation} \label{sample002}
\|f\|_{1, \alpha} \le C_2 \sum_n (1-|a_n|^2)^{n+1+\alpha}|f(a_n)|,
\end{equation}
if $0<r<r_0$.
\end{lem}

\begin{proof}
The first part of the above lemma is a particular case of \cite[Lemma 1.5]{JMT} and the second part is proved in \cite[Theorem 2]{luk}.
\end{proof}

Using the above lemma, we have the following result.

\begin{thm}
Let $p \ge 1$, $q>0$, $s>\max\left\{0, 1-\frac{q}{n}\right\}$, $\alpha>-1$ and $\{a_n\}$ be an $r$-lattice in the Bergman metric in $\B$. Then for any $\{c_n\} \in \ell^\infty$,
\begin{equation} \label{sample003}
f(z)=\sum_n c_n \cdot \left( \frac{1-|a_n|^2}{1-\langle z, a_n\rangle}\right)^{n+1+\alpha}
\end{equation}
belongs to $\calN(p, q, s)$. Moreover, there is an $r_0>0$ such that every $f \in \calN(p, q, s)$ has the form \eqref{sample003} for some $\{c_n\} \in \ell^\infty$ if $r<r_0$.
\end{thm}

\begin{proof}
Let $\{a_n\}$ be an $r$-lattice in the Bergman metric in $\B$. Moreover, since $s>1-\frac{q}{n}$, there exists some $k_0 \in \left(0, \frac{q}{p}\right]$, such that
$$
s>1-\frac{q-k_0p}{n},
$$
which, by Proposition \ref{Npqsbigs}, implies $A^{-k_0}(\B) \subseteq \calN(p, q, s)$.

Then $T$, defined as follows, is a bounded linear operator from $A^1_\alpha(\B)$ to $\ell^1$,
$$
T f=\{(T f)_n\}=\{(1-|a_n|^2)^{n+1+\alpha} f(a_n)\}, \ f \in A^1_\alpha(\B),
$$
where the boundedness of $T$ is due to \eqref{sample001} under $\{z_n\}$ being an $r$-lattice. Thus $T^*$, by \cite[Theorem 7.6]{Zhu}, the adjoint operator of $T$ is a bounded linear operator from $\ell^\infty (=(\ell^1)^*)$ to $A^{-k_0} (=\calB^{k_0+1}=(A^1_\alpha)^*)$, where $\calB^{k_0+1}$ is the $(k_0+1)$-Bloch space. Moreover, $T^*$ can be written as follows.
$$
\langle Tf, y \rangle=\langle f, T^*y \rangle, \quad f \in A^1_\alpha(\B),  \ y \in \ell^\infty,
$$
where the $\langle \cdot, \cdot \rangle$ is just the usual inner product between $\ell^1$ and $\ell^\infty$.

To compute $T^*$, we take
$$
y=e_n,  \quad (e_n)_m=
\begin{cases}
1, &m=n \\
0, &m \neq n.
\end{cases}
$$
So, for $f \in A^1_\alpha(\B)$,
\[ \begin{split}
\langle T f, e_n\rangle=(T f)_n&=(1-|a_n|^2)^{n+1+\alpha} f(a_n)\\
                                                 &=(1-|a_n|^2)^{n+1+\alpha} \langle f, K_{a_n}\rangle_{\alpha+k_0},
\end{split} \]
where $K_{a_n}(z)=\frac{1}{(1-\langle z, a_n \rangle)^{n+1+\alpha}}$ is the reproducing kernel for $A^1_\alpha(\B)$ and $\langle \cdot, \cdot  \rangle_{\alpha+k_0}$ is the integral pair defined in \cite[Theorem 7.6]{Zhu}. Hence
$$
T^*e_n=(1-|z_n|^2)^{n+1+\alpha} K_{a_n}(z)
$$
and
$$
T^*y=\sum_{n} c_n \cdot \left( \frac{1-|a_n|^2}{1-\langle z, a_n \rangle} \right)^{n+1+\alpha}, \quad \textrm{for} \ y=\{c_n\} \in \ell^\infty,
$$
i.e. the function in the form \eqref{sample003} is in $A^{-k_0}$. Since $A^{-k_0} \subseteq \calN(p, q, s)$, which implies the desired result.

Now we turn to show the second part. Noting that by \cite[Theorem 7.6]{Zhu} again, $(A^1_\alpha)^*=\calB^{\frac{q}{p}+1}=A^{-\frac{q}{p}}(\B)$, we also can regard $T^*$ as a bounded linear operator from $\ell^\infty$ to $A^{-\frac{q}{p}}(\B)$, where in the sequel, we denote $T^*$ as $S^*$.

In fact, it is only necessary to claim $S^*$ to be surjective. However, $S^*$ is onto if and only if $S: A^1_\alpha(\B) \mapsto \ell^1$ is bounded below (see, e.g., \cite[Theorem A, Page 194]{DS}). By Lemma \ref{atomicsample}, there exists an $r_0>0$ such that $S$ is bounded below if $\{a_n\}$ is an $r$-lattice with $0<r<r_0$, that is to say, there is an $r_0>0$, such that every $f \in A^{-\frac{q}{p}}(\B)$ has the form \ref{sample003} for some $\{c_n\} \in \ell^\infty$. Finally, by Proposition \ref{boundaryineq}, we note that $\calN(p, q, s) \subseteq A^{-\frac{q}{p}}(\B)$. The proof is complete.
\end{proof}

By modifying the proof of the above theorem a bit, we easily get the following corollary.

\begin{cor}
Let $p>0$, $\alpha>-1$ and $\{a_n\}$ be an $r$-lattice in the Bergman metric in $\B$. Then for any $\{c_n\} \in \ell^\infty$,
\begin{equation} \label{sample004}
f(z)=\sum_n c_n \cdot \left( \frac{1-|a_n|^2}{1-\langle z, a_n\rangle}\right)^{n+1+\alpha}
\end{equation}
belongs to $A^{-p}(\B)$. Moreover, there is an $r_0>0$ such that every $f \in A^{-p}(\B)$ has the form \eqref{sample004} for some $\{c_n\} \in \ell^\infty$ if $r<r_0$.
\end{cor}


\medskip

\subsection{The solvability of Gleason's problem}


In the second half of this section, we study the Gleason's problem on $\calN(p, q, s)$-type spaces, which, finally turns out to be an application of Lemma \ref{decom04}.

Let $X$ be a space of holomorphic functions on a domain $\Omega$ in $\C^n$. The Gleason's problem for $X$ is the following statement: Given $a \in \Omega, f \in X$, and $f(a)=0$, do there exist functions $f_1, \dots, f_n$ in $X$ such that
$$
f(z)=\sum_{k=1}^n (z_k-a_k)f_k(z)
$$
for all $z$ in $\Omega$?

The case when $X$ is the Bergman space or the Bloch space was considered in \cite{Zhu3}. In the present survey, we focus on the case $X=\calN(p, q, s)$ and $\Omega=\B$. We have the following results.

\begin{prop} \label{Gleprop01}
Suppose $p \ge 1, q>0$ and $s>\max\left\{0, 1-\frac{q}{n} \right\}$. Then there exist bounded linear operators $A_1, \dots, A_n$ on $\calN(p, q, s)$ such that
$$
f(z)-f(0)=\sum_{k=1}^n z_kA_kf(z), \quad \forall f \in \calN(p, q, s), z \in \B.
$$
Here $$
A_k  f = \int_0^1 \frac{\partial f}{\partial z_k}(tz)dt.
$$
\end{prop}

\begin{proof}
Note that for any $f$ holomorphic in $\B$, we have
$$
f(z)-f(0)=\sum_{k=1}^n z_k \int_0^1 \frac{\partial f}{\partial z_k}(tz)dt.
$$
Thus, it suffices to show that the operator $A_k$ is bounded on $\calN(p, q, s)$ for each $k \in \{1, \dots, n\}$. Without the loss of generality, we may assume $p>1$ since the proof for the case $p=1$ follows exactly the same line as the case $p>1$.

Take any $f \in \calN(p, q, s)$, and hence $f \in A^{-\frac{q}{p}}(\B)$. Moreover, take and fix $\alpha>\max\left\{\frac{q}{p}-1, \frac{q+ns-n-1}{p}\right\}$. It is easy to see that $f \in A_\alpha^1$, and hence by \cite[Theorem 2.2]{Zhu}
$$
f(z)=\int_\B \frac{f(w)dV_\alpha(w)}{(1-\langle z, w \rangle)^{n+1+\alpha}}, \quad \forall z \in \B,
$$
which implies that
$$
\frac{\partial f}{\partial z_k}(z)=C(\alpha) \int_\B \frac{\bar{w_k}(1-|w|^2)^\alpha f(w)dV(w)}{(1-\langle z, w\rangle)^{n+2+\alpha}},
$$
where $C(\alpha)$ is some constant which only depends on $\alpha$. Then we have
\begin{eqnarray*}
|A_kf(z)|%
 &=& \left| \int_0^1 \frac{\partial f}{\partial z_k}(tz)dt \right|\\
&\simeq& \left| \int_0^1 \left( \int_\B \frac{\bar{w_k}(1-|w|^2)^\alpha f(w)dV(w)}{(1-t\langle z, w\rangle)^{n+2+\alpha}}\right) dt \right| \\
&=& \left| \int_\B \bar{w_k} (1-|w|^2)^\alpha f(w) \left( \int_0^1 \frac{1}{(1-t\langle z, w \rangle)^{n+2+\alpha}}dt \right)dV(w) \right|\\
&\simeq& \left| \int_\B \frac{\bar{w_k}(1-|w|^2)^\alpha f(w)}{(1-\langle z, w \rangle)^{n+1+\alpha}} \cdot \frac{ 1-(1-\langle z, w \rangle)^{n+1+\alpha}}{\langle z, w \rangle} dV(w) \right|.
\end{eqnarray*}
Note that
$$
\frac{ 1-(1-\langle z, w \rangle)^{n+1+\alpha}}{\langle z, w \rangle}
$$
is a polynomial in $z$ and $\bar{w}$. Thus, we have
$$
|A_kf(z)| \lesssim \int_\B \frac{(1-|w|^2)^\alpha |f(w)|}{|1-\langle z, w \rangle|^{n+1+\alpha}}dV(w).
$$
Now in Lemma \ref{decom04}, put $t=-p$. Then by the choice of $\alpha$, we have
$$
|A_kf(z)|^p(1-|z|^2)^{q+ns}d\lambda(z)
$$
is an $(ns)$-Carleson measure, which, by Proposition \ref{Carleson01}, implies the desired result.
\end{proof}

\begin{cor} \label{Glecor1}
Suppose $p \ge 1, q>0$ and $s>\max\left\{0, 1-\frac{q}{n} \right\}$. Then there exist bounded linear operators $A_1, \dots, A_n$ on $\calN(p, q, s)$ such that
$$
f(z)=\sum_{k=1}^n (z_k-b_k)A_kf(z), \quad z \in \B
$$
for all $f \in \calN(p, q, s)$ with $f(b)=0, b=(b_1, \dots, b_n) \in \B$.
\end{cor}

\begin{proof}
Note that
\begin{eqnarray*}
f(z)%
&=&f(z)-f(b)=\int_0^1 \left( \frac{d}{dt} f(b+t(z-b)\right)dt\\
&=& \sum_{k=1}^n (z_k-b_k) \int_0^1 \frac{\partial f}{\partial z_k}(b+t(z-b))dt=\sum_{k=1}^n(z_k-b_k)A_kf(z).
\end{eqnarray*}
A similar argument as in Proposition \ref{Gleprop01} with choosing the same $\alpha$ shows that
\begin{eqnarray*}
&&|A_kf(z)| \\
&\le& \int_\B \frac{(1-|w|^2)^\alpha |f(w)|}{|1-\langle z, w\rangle|^{n+1+\alpha}} \cdot \left|\frac{(1-\langle b, w \rangle)^{n+1+\alpha}-(1-\langle z, w \rangle)^{n+1+\alpha}}{\langle b, w \rangle- \langle z, w \rangle}\right| dV(z)\\
&\lesssim& \int_\B \frac{(1-|w|^2)^\alpha |f(w)|}{|1-\langle z, w \rangle|^{n+1+\alpha}}dV(w),
\end{eqnarray*}
since the quantity
$$
\frac{(1-x)^{n+1+\alpha}-(1-y)^{n+1+\alpha}}{x-y}
$$
is clearly bounded when $|x|, |y| \le 2$. The rest of the proof follows exactly the same as Proposition \ref{Gleprop01}.
\end{proof}

\begin{cor}
Suppose $p \ge 1, q>0$, $s>\max\left\{0, 1-\frac{q}{n} \right\}$ and $m \in \N$. Then there exist bounded linear operators $A_\alpha (|\alpha|=m)$ on $\calN(p, q, s)$ such that
$$
f(z)=\sum_{|\alpha|=m} (z-b)^\alpha A_\alpha f(z), \quad z \in \B
$$
for all $f \in \calN(p, q, s)$ with $(D^\beta f)(b)=0, b=(b_1, \dots, b_n) \in \B$, where $\beta$ is the multi-index satisfying $|\beta|<m$. Moreover, we have
\begin{equation} \label{Gleeq001}
A_\alpha f(z)=C_\alpha \int_0^1 (1-t)^m (D^\alpha f)(b+t(z-b))dt, \quad \forall z \in \B.
\end{equation}
where $C_\alpha$ is some constant only depending on $\alpha$.
\end{cor}

\begin{proof}
We prove the result by induction. The base case is exactly Corollary \ref{Glecor1}. Suppose the statement is valid when $|\alpha|=m$. Take any $f \in \calN(p, q, s)$ with $(D^\beta f)(b)=0, |\beta|<m+1$, then by induction assumption and the fact that $A_\alpha f(b)=0, \forall |\alpha|=m$ (this follows from \eqref{Gleeq001}), we have
\begin{eqnarray*}
f(z)%
&=&\sum_{|\alpha|=m} (z-b)^\alpha A_\alpha f(z)\\
&=& \sum_{|\alpha|=m} (z-b)^\alpha \left( A_\alpha f(b)+ \sum_{k=1}^n(z_k-b_k) A_k (A_\alpha f)(z) \right)\\
&=& \sum_{|\alpha|=m} \sum_{k=1}^n (z-b)^\alpha (z_k-b_k) A_k (A_\alpha f)(z),
\end{eqnarray*}
which clearly can be written into the form of
$$
\sum\limits_{|\gamma|=m+1} (z-b)^\gamma A_\gamma f(z).
$$
The boundedness of $A_\gamma$ follows from the facts that $A_\gamma$ can be written as a finite sum
$$
\sum_{|k|=1, |\alpha|=m, \alpha+e_k=\gamma} A_k \circ A_\alpha
$$
and both $A_k$ and $A_\alpha$ are bounded on $\calN(p, q, s)$.

To prove \eqref{Gleeq001}, it suffices to show that for each $k$ and $\alpha$, the operator $A_k \circ A_\alpha$ is also of this form. Indeed, we have
\begin{eqnarray*}
A_k(A_\alpha f)(z)&=& \int_0^1 \frac{\partial (A_\alpha f)}{\partial z_k} (b+v(z-b))dv\\
& =&C_\alpha \int_0^1 \int_0^1 t (1-t)^m (D^{\alpha+e_k} f)(b+tv(z-b))dtdv\\
& = &C_\alpha \int_0^1 \int_0^t (1-t)^m (D^{\alpha+e_k} f)(b+u(z-b))dudt\\
&&  (\textrm{change variables with} \ t=t, u=tv)\\
&= & C_\alpha \int_0^1 (D^{\alpha+e_k} f)(b+u(z-b)) \left( \int_u^1 (1-t)^m dt \right) du\\
&&  (\textrm{by Fubini's theorem})\\
& =&C_\alpha' \int_0^1 (1-u)^{m+1} (D^\gamma f)(b+u(z-b))du,
\end{eqnarray*}
which implies the desired result.
\end{proof}


\medskip

\section{Distance between $A^{-\frac{q}{p}}(\B)$ spaces and $\calN(p, q, s)$-type spaces}


Recall that from Theorem \ref{boundaryineq}, we have, for $p \ge 1$ and $q, s>0$,
$$
\calN(p, q, s) \subseteq A^{-\frac{q}{p}}(\B).
$$
A natural question can be asked is: for any $f \in A^{-\frac{q}{p}}(\B)$, what can we say about the distance between $f$ and $\calN(p, q, s)$ with regarding $\calN(p, q, s)$ as a subspace of $A^{-\frac{q}{p}}(\B)$? In this section, we will focus on this question.

We denote the distance in $A^{-\frac{q}{p}}(\B)$ of $f$ to $\calN(p, q, s)$ by $d(f, \calN(p, q, s) )$, that is
$$
d(f, \calN(p, q, s) )=\inf_{g \in \calN(p, q ,s)} |f-g|_{\frac{q}{p}},
$$
where $|\cdot|_{\frac{q}{p}}$ is the norm defined on $A^{-\frac{q}{p}}(\B)$. Moreover, for $f \in H(\B)$ and $\varepsilon>0$, let
$$
\Omega_{\varepsilon}(f)=\{z \in \B: |f(z)|(1-|z|^2)^{\frac{q}{p}} \ge \varepsilon\}.
$$

We have the following result.

\begin{thm} \label{distance01}
Suppose $p \ge 1, q>0$, $s>\max\left\{0, 1-\frac{q}{n}\right\}$  and $f \in A^{-\frac{q}{p}}(\B)$. Then the following quantities are equivalent:
\begin{enumerate}
\item $d_1=d(f, \calN(p, q, s) )$;
\item $d_2=\inf\{\varepsilon: \chi_{\Omega_\varepsilon(f)}(z) (1-|z|^2)^{ns}d\lambda(z) \ \textrm{is an $(ns)$-Carleson measure}\}$;
\item
$$
d_3=\inf\left\{\varepsilon: \sup_{a \in \B} \int_{\Omega_\varepsilon(f)} |f(z)|^p(1-|z|^2)^q(1-|\Phi_a(z)|^2)^{ns}d\lambda(z)<\infty\right\}.
$$
\end{enumerate}
\end{thm}

\begin{proof}  (1) $d_1 \lesssim d_2$.

Without the loss of generality, we may assume that $p>1$. Let $\varepsilon$ be a positive number such that $\chi_{\Omega_\varepsilon(f)}(z)(1-|z|^2)^{ns}d\lambda(z)$ is an $(ns)$-Carleson measure. Since $f \in A^{-\frac{q}{p}}(\B)$, we have
$$
\sup_{z \in \B} |f(z)|(1-|z|^2)^{\frac{q}{p}}<\infty.
$$
Take and fix some $\alpha>\max\left\{\frac{q}{p}-1, \frac{q+ns-n-1}{p}\right\}$. It is easy to see that $f \in A^1_\alpha$, and hence by \cite[Theorem 2.2]{Zhu}, we have
$$
f(z)=\int_\B \frac{f(w)}{(1-\langle z, w\rangle)^{n+1+\alpha}}dV_\alpha(w).
$$
Let
$$
f_1(z)=\int_{\Omega_\varepsilon(f)}\frac{f(w)}{(1-\langle z, w\rangle)^{n+1+\alpha}}dV_\alpha(w)
$$
and
$$
f_2(z)=\int_{\B \backslash \Omega_\varepsilon(f)}\frac{f(w)}{(1-\langle z, w\rangle)^{n+1+\alpha}}dV_\alpha(w).
$$
It is clear that both $f_1$ and $f_2$ are holomorphic functions in $\B$ and $f(z)=f_1(z)+f_2(z)$. We have the following claims.

\textbf{$\bullet \  |f_2|_{\frac{q}{p}} \le C\varepsilon$ for some constant $C>0$.}

For any $z \in \B$, by \cite[Proposition 1.4.10]{Rud}, we have
\begin{eqnarray*}
|f_2(z)|%
&\le& \int_{\B \backslash \Omega_\varepsilon(f)}\frac{|f(w)|}{|1-\langle z, w\rangle|^{n+1+\alpha}}dV_\alpha(w)\\
&\lesssim& \varepsilon \int_\B \frac{(1-|w|^2)^{\alpha-\frac{q}{p}}}{|1-\langle z, w\rangle|^{n+1+\alpha}}dV_\alpha(w)\\
&\lesssim& \varepsilon (1-|z|^2)^{-\frac{q}{p}},
\end{eqnarray*}
which implies the desired claim.

\textbf{$\bullet \  f_1 \in \calN(p, q, s)$.}

By Theorem \ref{Carlesoncharac}, it suffices to show that $|f_1(z)|^p(1-|z|^2)^{q+ns}d\lambda(z)$ is a $(ns)$-Carleson measure.
Note that for $z \in \B$, we have
\begin{eqnarray*}
|f_1(z)|%
&\lesssim& \int_{\Omega_\varepsilon(f)} \frac{|f(w)|(1-|w|^2)^{\frac{q}{p}}(1-|w|^2)^{\alpha-\frac{q}{p}}dV(w)}{|1-\langle z, w \rangle|^{n+1+\alpha}}\\
&\lesssim& \int_{\Omega_\varepsilon(f)} \frac{(1-|w|^2)^{\alpha-\frac{q}{p}}dV(w)}{|1-\langle z, w \rangle|^{n+1+\alpha}}\\
&=& \int_\B \frac{(1-|w|^2)^{\alpha}}{|1-\langle z,w \rangle|^{n+1+\alpha}} \cdot \frac{\chi_{\Omega_\varepsilon(f)}(w)}{(1-|w|^2)^{\frac{q}{p}}}dV(w).
\end{eqnarray*}
Let
$$
g(w)= \frac{\chi_{\Omega_\varepsilon(f)}(w)}{(1-|w|^2)^{\frac{q}{p}}}
$$
and hence we have
$$
|g(w)|^p(1-|w|^2)^{q+ns}d\lambda(w)=\chi_{\Omega_\varepsilon(f)}(w) (1-|w|^2)^{ns}d\lambda(w),
$$
which is an $(ns)$-Carleson measure by our assumption. Now in Theorem \ref{decom04}, let $t=-p$ and $b=n+1+\alpha$, it is easy to check that
$$
t=-p>n-p-q-ns
$$
and
$$
 b=n+1+\alpha>\frac{n+1}{p'}+\frac{q+ns}{p}.
$$
Hence, the operator $T$ with parameter $n+1+\alpha$ sends the $(ns)$-Carleson measure
$$
|g(z)|^p(1-|z|^2)^{q+ns}d\lambda(z)
$$
to
\begin{eqnarray*}
&&|Tg(z)|^p(1-|z|^2)^{q+ns}d\lambda(z)\\
&=&\left|\int_\B \frac{(1-|w|^2)^{\alpha}}{|1-\langle z,w \rangle|^{n+1+\alpha}} \cdot \frac{\chi_{\Omega_\varepsilon(f)}(w)}{(1-|w|^2)^{\frac{q}{p}}}dV(w)\right|^p(1-|z|^2)^{q+ns}d\lambda(z),
\end{eqnarray*}
which, by Theorem \ref{decom04}, is also an $(ns)$-Carleson measure. Finally, since
$$
|f_1(z)|\lesssim  \int_\B \frac{(1-|w|^2)^{\alpha}}{|1-\langle z,w \rangle|^{n+1+\alpha}} \cdot \frac{\chi_{\Omega_\varepsilon(f)}(w)}{(1-|w|^2)^{\frac{q}{p}}}dV(w),
$$
it follows that $|f_1(z)|^p(1-|z|^2)^{q+ns}d\lambda(z)$ is an $(ns)$-Carleson measure. Hence, the claim is proved.

Note that $f_1 \in A^{-\frac{q}{p}}(\B)$ since $\calN(p, q, s) \subseteq A^{-\frac{q}{p}}(\B)$. Thus, we have
$$
d_1=d(f, \calN(p, q, s) ) \le |f-f_1|_{\frac{q}{p}}=|f_2|_{\frac{q}{p}}\lesssim \varepsilon.
$$
Finally, by letting $\varepsilon$ tends to $d_2$, we get the desired result.

  (2)  $d_2 \le d_3$.

Take and fix an $\varepsilon>0$ such that
$$
\sup_{a \in \B} \int_{\Omega_\varepsilon(f)} |f(z)|^p(1-|z|^2)^q(1-|\Phi_a(z)|^2)^{ns}d\lambda(z)<\infty.
$$
Since $|f(z)|^p(1-|z|^2)^q \ge \varepsilon^p$ for $z \in \Omega_\varepsilon(f)$, it follows that
\begin{eqnarray*}
&&\sup_{a \in \B} \int_{\Omega_\varepsilon(f)} (1-|\Phi_a(z)|^2)^{ns}d\lambda(z)\\
&=&\sup_{a \in \B} \int_\B \left( \frac{1-|a|^2}{|1-\langle z, a \rangle|^2} \right)^{ns} \chi_{\Omega_\varepsilon(f)}(z)(1-|z|^2)^{ns}d\lambda(z)<\infty,
\end{eqnarray*}
which, by \cite[Theorem 45]{Zhu}, implies  $\chi_{\Omega_\varepsilon(f)}(z)(1-|z|^2)^{ns}d\lambda(z)$ is an $(ns)$-Carleson measure. Hence,
\[ \begin{split}
\bigg\{\varepsilon &: \sup_{a \in \B} \int_{\Omega_\varepsilon(f)} |f(z)|^p(1-|z|^2)^q(1-|\Phi_a(z)|^2)^{ns}d\lambda(z)<\infty \bigg\} \\
&\subseteq \left\{\varepsilon: \chi_{\Omega_\varepsilon(f)}(z) (1-|z|^2)^{ns}d\lambda(z) \ \textrm{is an $(ns)$-Carleson measure}\right\},
\end{split} \]
which implies $d_2 \le d_3$.

 (3)   $d_3 \le d_1$.

It suffices to show that
$$
(d_1, \infty) \subseteq \bigg\{\varepsilon : \sup_{a \in \B} \int_{\Omega_\varepsilon(f)} |f(z)|^p(1-|z|^2)^q(1-|\Phi_a(z)|^2)^{ns}d\lambda(z)<\infty \bigg\}.
$$
Take and fix any $\varepsilon>d_1$, then there exists a $f_1 \in \calN(p, q, s)$, such that
$$
|f-f_1|_\frac{q}{p}<\frac{d_1+\varepsilon}{2}.
$$
By triangle inequality, we have for $z \in \Omega_\varepsilon(f)$,
\begin{eqnarray*}
|f_1(z)|(1-|z|^2)^{\frac{q}{p}}%
&\ge& |f(z)|(1-|z|^2)^{\frac{q}{p}}-|f(z)-f_1(z)|(1-|z|^2)^{\frac{q}{p}}\\
&\ge& \varepsilon-\frac{d_1+\varepsilon}{2}=\frac{\varepsilon-d_1}{2}.
\end{eqnarray*}
Thus, we have
\begin{eqnarray*}
&&\sup_{a \in \B} \int_{\Omega_\varepsilon(f)} |f(z)|^p(1-|z|^2)^q(1-|\Phi_a(z)|^2)^{ns}d\lambda(z)\\
&\lesssim& \sup_{a \in \B} \int_{\Omega_\varepsilon(f)} (1-|\Phi_a(z)|^2)^{ns}d\lambda(z)\\
&\le& \frac{2^p}{(\varepsilon-d_1)^p} \sup_{a \in \B} \int_\B |f_1(z)|^p(1-|z|^2)^q(1-|\Phi_a(z)|^2)^{ns}d\lambda(z)\\
&<&\infty,
\end{eqnarray*}
which implies the desired inclusion. The proof is complete.
\end{proof}

Noting that by Corollary \ref{equivalentnorm}, we can express the $\calN(p, q, s)$-norm by using complex gradient and radial derivative respectively. By using these equivalent norms, we can express $d(f, \calN(p, q, s) )$ via different forms.

More precisely, from the view of Lemma \ref{BlochBertype}, it is clear that we can express $d(f, \calN(p, q, s) )$ as
$$
\inf_{g \in \calN(p, q, s)} \|f-g\|_{\calB^{\frac{q}{p}+1}}.
$$
Now for $f \in H(\B)$ and $\varepsilon>0$, we denote
$$
\widetilde{\Omega}_\varepsilon(f)=\{z \in \B: |Rf(z)|(1-|z|^2)^{\frac{q}{p}+1} \ge \varepsilon\}.
$$
By using the above expression, we have the following result.

\begin{thm} \label{distance02}
Suppose $p \ge 1, q>0$, $s>\max\left\{0, 1-\frac{q}{n}\right\}$  and $f \in A^{-\frac{q}{p}}(\B)=\calB^{\frac{q}{p}+1}$. Then the following quantities are equivalent:
\begin{enumerate}
\item $d_1=d(f, \calN(p, q, s) )$;
\item $d_4=\inf\{\varepsilon: \chi_{\widetilde{\Omega}_\varepsilon(f)}(z) (1-|z|^2)^{ns}d\lambda(z) \ \textrm{is an $(ns)$-Carleson measure}\}$;
\item
$$
d_5=\inf\left\{\varepsilon: \sup_{a \in \B} \int_{\widetilde{\Omega}_\varepsilon(f)} |Rf(z)|^p(1-|z|^2)^{p+q}(1-|\Phi_a(z)|^2)^{ns}d\lambda(z)<\infty\right\};
$$
\item
$$
d_6=\inf\left\{\varepsilon: \sup_{a \in \B} \int_{\widetilde{\Omega}_\varepsilon(f)} |\nabla f(z)|^p(1-|z|^2)^{p+q}(1-|\Phi_a(z)|^2)^{ns}d\lambda(z)<\infty\right\};
$$
\item
$$
d_7=\inf\left\{\varepsilon: \sup_{a \in \B} \int_{\widetilde{\Omega}_\varepsilon(f)} |\widetilde{\nabla} f(z)|^p(1-|z|^2)^q(1-|\Phi_a(z)|^2)^{ns}d\lambda(z)<\infty\right\}.
$$
\end{enumerate}
\end{thm}

\begin{proof}  (1)  $d_1 \lesssim d_4$.

The proof for this part is similar to the proof of $d_1 \lesssim d_2$ in Theorem \ref{distance01}. Again, we may assume that $p>1$. Let $\varepsilon$ be a positive number such that $\chi_{\widetilde{\Omega}_\varepsilon(f)}(z) (1-|z|^2)^{ns}d\lambda(z)$ is an $(ns)$-Carleson measure. Since $f \in \calB^{\frac{q}{p}+1}$, we have
$$
\sup_{z \in \B} |Rf(z)|(1-|z|^2)^{\frac{q}{p}+1}<\infty.
$$
Take and fix some $\alpha>\max\left\{\frac{q}{p}, \frac{q+ns-n-1}{p}+1 \right\}$. It is easy to see that $Rf(z) \in A^1_\alpha$, and hence by \cite[Theorem 2.2]{Zhu}, we have
$$
Rf(z)=\int_\B \frac{Rf(w)dV_\alpha(w)}{(1-\langle z, w \rangle)^{n+1+\alpha}}, \  z \in \B.
$$
Since $Rf(0)=0$, we have
$$
Rf(z)=\int_\B Rf(w) \left(\frac{1}{(1-\langle z, w \rangle)^{n+1+\alpha}}-1 \right)dV_\alpha(w), \  z \in \B.
$$
It follows that
$$
f(z)-f(0)=\int_0^1 \frac{Rf(tz)}{t}dt=\int_\B Rf(w)L(z, w)dV_\alpha(w),
$$
where the kernel
$$
L(z, w)=\int_0^1 \left(\frac{1}{(1-t\langle z, w \rangle)^{n+1+\alpha}}-1\right)\frac{dt}{t}.
$$
Let $f(z)=f_1(z)+f_2(z)$, where
$$
f_1(z)=f(0)+\int_{\widetilde{\Omega}_\varepsilon(f)} Rf(w)L(z, w)dV_\alpha(w)
$$
and
$$
f_2(z)=\int_{\B \backslash \widetilde{\Omega}_\varepsilon(f)} Rf(w)L(z, w)dV_\alpha(w)
$$
We have the following claims.

\textbf{$\bullet \ \|f_2\|_{\calB^{\frac{q}{p}+1}} \le C\varepsilon$ for some constant $C>0$.}

Since
$$
RL(z, w)=\int_0^1 \frac{(n+1+\alpha)\langle z , w\rangle dt}{(1-t\langle z, w\rangle)^{n+\alpha+2}}=\frac{1}{(1-\langle z, w\rangle)^{n+1+\alpha}}-1,
$$
we have
\begin{eqnarray*}
|Rf_2(z)|%
&=& \left| \int_{\B \backslash \widetilde{\Omega}_\varepsilon(f)} Rf(w)RL(z, w)dV_\alpha(w) \right|\\
&\lesssim& \varepsilon \int_{\B \backslash \widetilde{\Omega}_\varepsilon(f)}  (1-|w|^2)^{\alpha-\frac{q}{p}-1} \cdot \left(\frac{1}{(1-\langle z, w\rangle)^{n+1+\alpha}}-1\right)dV(w)\\
&\le& \varepsilon \cdot \left(\int_\B \frac{(1-|w|^2)^{\alpha-\frac{q}{p}-1}}{|1-\langle z, w \rangle|^{n+1+\alpha-\frac{q}{p}-1+\left(\frac{q}{p}+1\right)}}dV(w)+1\right)\\
&\lesssim&\frac{\varepsilon}{(1-|z|^2)^{\frac{q}{p}+1}},
\end{eqnarray*}
which implies the desired claim.

\textbf{$\bullet \ f_1 \in \calN(p, q, s)$.}

By Theorem \ref{Carlesoncharac}, it suffices to show that $|Rf_1(z)|^p(1-|z|^2)^{p+q+ns}d\lambda(z)$ is a $(ns)$-Carleson measure. Note that for $z \in \B$, we have
\begin{eqnarray*}
|Rf_1(z)|%
&=& \left| \int_{\widetilde{\Omega}_\varepsilon(f)} Rf(w)RL(z, w)dV_\alpha(w)\right|\\
&\le& \int_{\widetilde{\Omega}_\varepsilon(f)} |Rf(w)| \left(\frac{1}{|1-\langle z, w \rangle|^{n+1+\alpha}}+1\right)dV_\alpha(w)\\
&=& I_1+I_2,
\end{eqnarray*}
where
$$
I_1=\int_{\widetilde{\Omega}_\varepsilon(f)}  \frac{|Rf(w)|}{|1-\langle z, w \rangle|^{n+1+\alpha}}dV_\alpha(w)
$$
and
$$
I_2=\int_{\widetilde{\Omega}_\varepsilon(f)} |Rf(w)| dV_\alpha(w).
$$
First, we note that by our choice of $\alpha$, it follows that
\begin{eqnarray*}
I_2%
&\lesssim& \int_\B |Rf(w)|(1-|w|^2)^{\frac{q}{p}+1} (1-|w|^2)^{\alpha-\frac{q}{p}-1}dV(w)\\
&\lesssim& \int_\B (1-|w|^2)^{\alpha-\frac{q}{p}-1}dV(w) \lesssim 1.
\end{eqnarray*}
Next, we estimate $I_1$. Note that
\begin{eqnarray*}
I_1%
&\simeq& \int_{\widetilde{\Omega}_\varepsilon(f)}  \frac{|Rf(w)|(1-|w|^2)^{\frac{q}{p}+1}(1-|w|^2)^{\alpha-\frac{q}{p}-1}}{|1-\langle z, w \rangle|^{n+1+\alpha}}dV(w)\\
&\lesssim& \int_\B \frac{(1-|w|^2)^{\alpha}}{|1-\langle z, w \rangle|^{n+1+\alpha}} \cdot \frac{\chi_{\widetilde{\Omega}_\varepsilon(f)}(w)}{(1-|w|^2)^{\frac{q}{p}+1}}dV(w)
\end{eqnarray*}
Let
$$
g(w)=\frac{\chi_{\widetilde{\Omega}_\varepsilon(f)}(w)}{(1-|w|^2)^{\frac{q}{p}+1}}
$$
and hence we have
$$
|g(w)|^p(1-|w|^2)^{p+q+ns}d\lambda(w)=\chi_{\widetilde{\Omega}_\varepsilon(f)}(w)(1-|w|^2)^{ns}d\lambda(w),
$$
which is an $(ns)$-Carleson measure by our assumption. Now in Theorem \ref{decom04}, let $t=0$ and $b=n+1+\alpha$, it is easy to check that
$$
t=0>n-p-q-ns
$$
and
$$
b=n+1+\alpha>\frac{n+1}{p'}+\frac{q+ns}{p}+1.
$$
Hence, the operator $T$ with parameter $n+1+\alpha$ sends the $(ns)$-Carleson measure
$$
|g(z)|^p(1-|z|^2)^{p+q+ns}d\lambda(z)
$$
to
\begin{eqnarray*}
&&|Tg(z)|^p(1-|z|^2)^{p+q+ns}d\lambda(z)\\
&& \left| \int_\B \frac{(1-|w|^2)^{\alpha}}{|1-\langle z, w \rangle|^{n+1+\alpha}} \cdot \frac{\chi_{\widetilde{\Omega}_\varepsilon(f)}(w)}{(1-|w|^2)^{\frac{q}{p}+1}}dV(w)\right|^p(1-|z|^2)^{p+q+ns}d\lambda(z),
\end{eqnarray*}
which, by Theorem \ref{decom04}, is also an $(ns)$-Carleson measure. Finally, we have
$$
|Rf_1(z)| \le I_1+I_2 \lesssim \left| \int_\B \frac{(1-|w|^2)^{\alpha}}{|1-\langle z, w \rangle|^{n+1+\alpha}} \cdot \frac{\chi_{\widetilde{\Omega}_\varepsilon(f)}(w)}{(1-|w|^2)^{\frac{q}{p}+1}}dV(w)\right|+1
$$
and it follows that $|Rf_1(z)|^p(1-|z|^2)^{p+q+ns}d\lambda(z)$ is an $(ns)$-Carleson measure. Hence, the claim is proved.

Note that $f_1 \in A^{-\frac{q}{p}}(\B)=\calB^{\frac{q}{p}+1}$ since $\calN(p, q, s) \subseteq A^{-\frac{q}{p}}(\B)=\calB^{\frac{q}{p}+1}$. Thus, we have
$$
d_1=d(f, \calN(p, q, s) ) \lesssim \|f-f_1\|_{\calB^{\frac{q}{p}+1}}=\|f_2\|_{\calB^{\frac{q}{p}+1}}\lesssim \varepsilon.
$$
Finally, by letting $\varepsilon$ tends to $d_2$, we get the desired result.

 (2)  $d_4 \le d_5$.

The proof for this part is almost the same as the proof for $d_2 \le d_3$ in Theorem \ref{distance01} and hence we omit it here.

 (3)  $d_5 \le d_6 \le d_7$.

This assertion is follows by the   inequality (see, e.g., \cite[Lemma 2.14]{Zhu})
$$
|Rf(z)|(1-|z|^2) \le |\nabla f(z)|(1-|z|^2) \le |\widetilde{\nabla} f(z)|.
$$

 (4)  $d_7 \le d_1$.

It suffices to show that
$$
(d_1, \infty) \subset \left\{\varepsilon: \sup_{a \in \B} \int_{\widetilde{\Omega}_\varepsilon(f)} |\widetilde{\nabla} f(z)|^p(1-|z|^2)^q (1-|\Phi_a(z)|^2)^{ns}d\lambda(z)<\infty\right\}.
$$
Take and fix any $\varepsilon>d_1$, there exists a function $f_1 \in \calN(p, q, s)$ such that
$$
\|f-f_1\|_{\calB^{\frac{q}{p}+1}}<\frac{d_1+\varepsilon}{2}.
$$
Then, by triangle inequality, we have for $z \in \widetilde{\Omega}_\varepsilon(f)$,
\begin{eqnarray*}
|\widetilde{\nabla} f_1(z)|(1-|z|^2)^{\frac{q}{p}}%
&\ge& |Rf_1(z)|(1-|z|^2)^{\frac{q}{p}+1}\\
&\ge&|Rf(z)|(1-|z|^2)^{\frac{q}{p}+1}-|R(f-f_1)(z)|(1-|z|^2)^{\frac{q}{p}+1}\\
&\ge& \varepsilon-\frac{d_1+\varepsilon}{2}=\frac{\varepsilon-d_1}{2}.
\end{eqnarray*}
Thus, it follows that
\begin{eqnarray*}
&&\sup_{a \in \B} \int_{\widetilde{\Omega}_\varepsilon(f)} |\widetilde{\nabla} f(z)|^p(1-|z|^2)^q(1-|\Phi_a(z)|^2)^{ns}d\lambda(z)\\
&\lesssim& \sup_{a \in \B} \int_{\widetilde{\Omega}_\varepsilon(f)} (1-|\Phi_a(z)|^2)^{ns}d\lambda(z)\\
&\le&\left(\frac{2}{\varepsilon-d_1}\right)^p \sup_{a  \in \B} \int_\B |\nabla f_1(z)|^p(1-|z|^2)^{p+q}(1-|\Phi_a(z)|^2)^{ns}d\lambda(z)\\
&<&\infty,
\end{eqnarray*}
where in the last inequality, we use Corollary \ref{equivalentnorm} and in the first inequality, we use the fact that $1+\frac{q}{p}>\frac{1}{2}$ and hence by our previous remark,
$$
|f(0)|+\sup_{z \in \B} (1-|z|^2)^{\frac{q}{p}}|\widetilde{\nabla} f(z)|
$$
becomes an equivalent norm of $\calB^{\frac{q}{p}+1}$. Therefore we get the desired result.
\end{proof}

\begin{cor}
Suppose $p \ge 1, q>0$, $s>\max\left\{0, 1-\frac{q}{n}\right\}$  and $f \in A^{-\frac{q}{p}}(\B)$. Then the following conditions are equivalent:
\begin{enumerate}
\item $f$ is in the closure of $\calN(p, q, s)$ in $A^{-\frac{q}{p}}(\B)$.
\item $\chi_{\Omega_\varepsilon(f)}(z)(1-|z|^2)^{ns}d\lambda(z)$ is an $(ns)$-Carleson measure for every $\varepsilon>0$;
\item $\chi_{\widetilde{\Omega}_\varepsilon(f)}(z)(1-|z|^2)^{ns}d\lambda(z)$ is an $(ns)$-Carleson measure for every $\varepsilon>0$;
\item
$$
\sup_{a \in \B} \int_{\Omega_\varepsilon(f)} |f(z)|^p(1-|z|^2)^q(1-|\Phi_a(z)|^2)^{ns}d\lambda<\infty
$$
for every $\varepsilon>0$;
\item
$$
\sup_{a \in \B} \int_{\widetilde{\Omega}_\varepsilon(f)} |Rf(z)|^p(1-|z|^2)^{p+q}(1-|\Phi_a(z)|^2)^{ns}d\lambda(z)<\infty.
$$
for every $\varepsilon>0$;
\item
$$
\sup_{a \in \B} \int_{\widetilde{\Omega}_\varepsilon(f)} |\nabla f(z)|^p(1-|z|^2)^{p+q}(1-|\Phi_a(z)|^2)^{ns}d\lambda(z)<\infty
$$
for every $\varepsilon>0$;
\item
$$
\sup_{a \in \B} \int_{\widetilde{\Omega}_\varepsilon(f)} |\widetilde{\nabla} f(z)|^p(1-|z|^2)^q(1-|\Phi_a(z)|^2)^{ns}d\lambda(z)<\infty
$$
for every $\varepsilon>0$.
\end{enumerate}
\end{cor}

For the ``little-oh" version, we denote the distance in $A^{-\frac{q}{p}}(\B)$ of $f$ to $\calN^0(p, q, s)$ by $d(f, \calN^0(p, q, s))$, that
is
$$
d(f, \calN^0(p, q, s))=\inf_{g \in \calN^0(p, q, s)}|f-g|_{\frac{q}{p}}.
$$

We have the following result.

\begin{thm}
Suppose $p \ge 1, q>0$, $s>\max\left\{0, 1-\frac{q}{n}\right\}$  and $f \in A^{-\frac{q}{p}}(\B)$. Then the following conditions are equivalent:
\begin{enumerate}
\item $e_1=d(f, A^{-\frac{q}{p}}_0(\B))$;
\item $e_2=d(f, \calN^0(p, q, s))$;
\item
$$
e_3=\inf\{\varepsilon:\chi_{\Omega_\varepsilon(f)}(z)(1-|z|^2)^{ns}d\lambda(z) \ \textrm{is a vanishing $(ns)$-Carleson measure} \};
$$
\item
$$
e_4=\inf\{\varepsilon:\chi_{\widetilde{\Omega}_\varepsilon(f)}(z)(1-|z|^2)^{ns}d\lambda(z) \ \textrm{is a vanishing $(ns)$-Carleson measure} \};
$$
\item
$$
e_5=\inf\left\{\varepsilon: \lim_{|a| \to 1} \int_{\Omega_\varepsilon(f)} |f(z)|^p(1-|z|^2)^q(1-|\Phi_a(z)|^2)^{ns}d\lambda(z)=0 \right\};
$$
\item
$$
e_6=\inf\left\{\varepsilon: \lim_{|a| \to 1}  \int_{\widetilde{\Omega}_\varepsilon(f)} |Rf(z)|^p(1-|z|^2)^{p+q}(1-|\Phi_a(z)|^2)^{ns}d\lambda(z)=0 \right\};
$$
\item
$$
e_7=\inf\left\{\varepsilon: \lim_{|a| \to 1}  \int_{\widetilde{\Omega}_\varepsilon(f)} |\nabla f(z)|^p(1-|z|^2)^{p+q}(1-|\Phi_a(z)|^2)^{ns}d\lambda(z)=0\right\};
$$
\item
$$
e_8=\inf\left\{\varepsilon: \lim_{|a| \to 1}  \int_{\widetilde{\Omega}_\varepsilon(f)} |\widetilde{\nabla} f(z)|^p(1-|z|^2)^q(1-|\Phi_a(z)|^2)^{ns}d\lambda(z)=0\right\}.
$$
\end{enumerate}
\end{thm}

\begin{proof}
Since both $A^{-\frac{q}{p}}_0(\B)$ and $\calN^0(p, q, s)$ are the closure of all the polynomials in $A^{-\frac{q}{p}}(\B)$ and $\calN(p, q, s)$ respectively, the equivalence between $e_1$ and $e_2$ is obvious.

Moreover, in the proof of Theorems \ref{distance01} and \ref{distance02}, by interchanging the role of $d(f, \calN(p, q, s))$ to $d(f, \calN^0(p, q, s))$, $\sup_{a \in \B}$ to $\lim_{|a| \to 1}$ and applying Lemma \ref{decom06} instead of Lemma \ref{decom04}, we can get the equivalence of $e_2, e_3, \dots, e_8$.
\end{proof}


\medskip

\section{Riemann-Stieltjes operators and multipliers}


In this section, we will study the behavior of Riemann-Stieltjes operators on $\calN(p, q, s)$-type spaces, which can be interpreted as ``half" of the multiplication operator.

Precisely, let $g$ be a holomorphic function on $\B$. Denote the \emph{Riemann-Stieltjes operators $T_g$ and $L_g$ with symbol $g$} as
\[ \begin{split}
&T_g f(z)=\int_0^1 f(tz)Rg(tz)\frac{dt}{t}, \quad f \in H(\B),  z \in \B; \\
& L_g f(z)=\int_0^1 g(tz)Rf(tz)\frac{dt}{t}, \quad f \in H(\B), z \in \B.
\end{split}
\]
Clearly, the Riemann-Stieltjes operator $T_g$ can be viewed as a generalization of the well-known Ces\`aro operator. It is also easy to see that the multiplication operator $M_g$ are determined by
$$
M_gf(z)=g(z)f(z)=g(0)f(0)+T_gf(z)+L_gf(z), \quad  f \in H(\B), z \in \B.
$$
In general, these operators are usually referred as the integral operators, which have been studied under various settings (see, e.g., \cite{CLS, ZH2, SL, JX3}).

For the purpose to study the boundedness and compactness of the Riemann-Stieltjes operators on $\calN(p, q, s)$-spaces, we introduce the following \emph{non-isotropic tent type space $T_{m, l}^\infty$} of all $\mu$-measure functions $f$ on $\B$ satisfying
$$
\|f\|_{T_{m, l}^\infty(\mu)}=\sup_{\xi \in \SSS, \del>0} \left(\frac{1}{\del^{nl}} \int_{Q_\del(\xi)} |f|^md\mu\right)^{\frac{1}{m}}<\infty,
$$
where $Q_\del(\xi)=\{z \in \B: |1-\langle z, \xi \rangle|<\del\}$ for $\xi \in \SSS$ and $\del>0$. The tent-type space is a very powerful tool in studying some deep properties of function spaces, for example, in Section 6, we take the advantage of this tent-type space to study the atomic decomposition of $\calN(p, q, s)$-spaces with $d\mu=(1-|z|^2)^{p+q+ns}d\lambda(z)$. Finally, based on the setting at the beginning of Section 4, we denote
$$
\|\mu\|_{\calC \calM_p}=\sup_{\xi \in \SSS, \del>0} \frac{\mu(Q_\del(\xi))}{\del^{p}}.
$$


\medskip

\subsection{Embedding theorem of $\calN(p, q, s)$-spaces into the tent space}


We need the following well-known lemma (see, e.g., \cite[Theorem 45]{ZZ}).

\begin{lem} \label{lemSec802}
Suppose $n+1+\alpha>0$ and $\mu$ is a positive Borel measure on $\B$. Then the following conditions are equivalent.
\begin{enumerate}
\item[(a)] There exists a constant $C>0$ such that
$$
\mu(Q_r(\xi)) \le Cr^{n+1+\alpha}
$$
for all $\xi \in \SSS$ and all $r>0$.
\item[(b)] For each $s>0$ there exists a constant $C>0$ such that
\begin{equation} \label{CarlesonSec801}
\int_\B \frac{(1-|z|^2)^sd\mu(w)}{|1-\langle z, w \rangle|^{n+1+\alpha+s}} \le C
\end{equation}
for all $z \in \B$.
\item[(c)] For some $s>0$ there exists a constant $C>0$ such that   the inequality in \eqref{Carleson01} holds for all $z \in \B$.
\end{enumerate}
\end{lem}

The following lemma plays an important role in the sequel.

\begin{lem} \label{lemSec803}
Let $p \ge 1, q>0, s>\max\left\{0, 1-\frac{q}{n} \right\}$ and for a fixed $w \in \B$, put
$$
K_w(z)=\frac{1-|w|^2}{(1-\langle z, w \rangle)^{1+\frac{q}{p}}}, \quad z \in \B.
$$
Then $\sup_{w \in \B} \|K_w\| \lesssim 1$.
\end{lem}

\begin{proof}
For any fixed $w, a \in \B$, we consider two different cases.

\textit{Case I: $s>1$.}

By \cite[Proposition 1.4.10]{Rud} and the inequality $$|1-\langle z, w \rangle|> \max \{1-|z|, 1-|w| \}, \forall z, w \in \B,$$
 we have
\begin{eqnarray*}
&&\int_\B |K_w(z)|^p(1-|z|^2)^q(1-|\Phi_a(z)|^2)^{ns}d\lambda(z)\\
&= &\int_\B \frac{(1-|w|^2)^p(1-|z|^2)^{q-n-1}}{|1-\langle z, w \rangle|^{p+q}} \cdot \frac{(1-|a|^2)^{ns}(1-|z|^2)^{ns}}{|1-\langle z, a \rangle|^{2ns}}dV(z)\\
&\le& (1-|a|^2)^{ns}  \int_\B \frac{(1-|z|^2)^{ns-n-1}}{|1-\langle z, a \rangle|^{n+1+(ns-n-1)+ns}} dV(z)<\infty.
\end{eqnarray*}

\textit{Case II: $s \le 1$.} We may assume that $s<1$ since the proof for the case $s=1$ is similar to the case $s<1$. Now we take and fix an $L$ satisfying
$$
\max\left\{1, \frac{n}{q} \right\}<L<\frac{1}{1-s}
$$
and $\gamma=q-\frac{n+1}{L}$. Again, using \cite[Proposition 1.4.10]{Rud} and the fact that $n+1+(q+ns-n-1-\gamma)L'=nsL'$, where $L'$ is the conjugate of $L$, we have
\begin{eqnarray*}
&&\int_\B |K_w(z)|^p(1-|z|^2)^q(1-|\Phi_a(z)|^2)^{ns}d\lambda(z)\\
&=&\int_\B \frac{(1-|w|^2)^p(1-|z|^2)^\gamma}{|1-\langle z, w \rangle|^{p+q}} \cdot \frac{(1-|a|^2)^{ns}(1-|z|^2)^{q+ns-n-1-\gamma}}{|1-\langle z, a \rangle|^{2ns}}dV(z)\\
&\le & \left( \int_\B \frac{(1-|w|^2)^{pL} (1-|z|^2)^{\gamma L}}{|1-\langle z, w\rangle|^{(p+q)L}}dV(z)\right)^{\frac{1}{L}} \\
&&  \quad \cdot \left( \int_\B \frac{(1-|a|^2)^{nsL'}(1-|z|^2)^{(q+ns-n-1-\gamma)L'}}{|1-\langle z, a \rangle|^{2nsL'}}dV(z) \right)^{\frac{1}{L'}}\\
&<&\infty,
\end{eqnarray*}
where in the last inequality, we use the facts that $\gamma L>-1$ and $(q+ns-n-1-\gamma)L'>-1$.
\end{proof}

Note that for $p \ge 1$ and $\alpha>-1$, by \cite[Lemma 2.24]{Zhu} and \eqref{Charac11}, we have
\begin{eqnarray} \label{equation001}
|f(z)|^p%
&\lesssim& \frac{1}{(1-|z|^2)^{n+1+\alpha}} \int_{D(z, 1/2)} |f(w)|^pdV_\alpha(w)\\  \nonumber
&\simeq& \int_{D(z, 1/2)} \frac{|f(w)|^pdV_\alpha(w)}{|1-\langle z, w \rangle|^{n+1+\alpha}}\\  \nonumber
&\le& \int_\B \frac{|f(w)|^pdV_\alpha(w)}{|1-\langle z, w \rangle|^{n+1+\alpha}}. \nonumber
\end{eqnarray}

\begin{lem} \label{lemSec804}
Let $p \ge 1, q>0, s>\max\left\{0, 1-\frac{q}{n} \right\}$, $t \ge s+\frac{q}{n}$ and $\mu$ be an $(nt)$-Carleson measure. Then for any fixed $\xi \in \SSS$ and $0<\del \le 2$, we have
$$
\int_{Q_\del(\xi)} \int_\B \frac{|f(w)|^p(1-|w|^2)^\alpha}{|1-\langle z, w \rangle|^{n+1+\alpha}}dV(w)d\mu(z) \lesssim \del^{nt-q} \|\mu\|_{\calC \calM_{nt}} \|f\|^p
$$
for some $\alpha>nt-n-1$.
\end{lem}

\begin{proof}
Denote
\begin{eqnarray*}
I%
&=& \int_{Q_\del(\xi)} \int_{Q_{4\del}(\xi)} \frac{|f(w)|^p(1-|w|^2)^\alpha}{|1-\langle z, w \rangle|^{n+1+\alpha}}dV(w)d\mu(z)\\
&&+ \sum_{j=2}^\infty \int_{Q_\del(\xi)} \int_{A_j}\frac{|f(w)|^p(1-|w|^2)^\alpha}{|1-\langle z, w \rangle|^{n+1+\alpha}}dV(w)d\mu(z)\\
&=& I_1+I_2,
\end{eqnarray*}
where  $A_1=Q_{4\del}(\xi)$ and $A_j=Q_{4^j\del} \backslash Q_{4^{j-1}\del}(\xi), j \ge 2$.

$\bullet$ \textbf{Estimation of $I_1$.}

Note that for $w \in Q_{4\del}(\xi)$, we have $1-|w| \le |1-\langle w, \xi \rangle|<4\del,$
which implies $(1-|w|^2)^{nt-q-ns} \lesssim \del^{nt-q-ns}$ since $t \ge s+\frac{q}{n}$.  Thus, by Proposition \ref{Carleson01}, Lemma \ref{lemSec802} and Fubini's theorem, we have
\begin{eqnarray*}
&&I_1= \int_{Q_{4\del}(\xi)} \left[\int_{Q_\del(\xi)} \frac{(1-|w|^2)^\alpha}{|1-\langle z, w \rangle|^{n+1+\alpha}}d\mu(z)\right]|f(w)|^pdV(w) \\
&=& \int_{Q_{4\del}(\xi)} \left[\int_{Q_\del(\xi)} \frac{(1-|w|^2)^{\alpha+n+1-q-ns}}{|1-\langle z, w \rangle|^{n+1+\alpha}}d\mu(z)\right]|f(w)|^p(1-|w|^2)^{q+ns}d\lambda(w)\\
&=&  \int_{Q_{4\del}(\xi)} \left[\int_{Q_\del(\xi)} \frac{(1-|w|^2)^{\alpha+n+1-nt+(nt-q-ns)}}{|1-\langle z, w \rangle|^{n+1+\alpha-nt+nt}}d\mu(z)\right]\\
&& |f(w)|^p(1-|w|^2)^{q+ns}d\lambda(w)\end{eqnarray*}
\begin{eqnarray*}
&\lesssim& \del^{nt-q-ns} \int_{Q_{4\del}(\xi)} \left[\int_{Q_\del(\xi)} \frac{(1-|w|^2)^{\alpha+n+1-nt}}{|1-\langle z, w \rangle|^{(\alpha+n+1-nt)+nt}}d\mu(z)\right]\\
&& |f(w)|^p(1-|w|^2)^{q+ns}d\lambda(w) \\
&\lesssim & \del^{nt-q-ns} \|\mu\|_{\calC \calM_{nt}} \int_{Q_{4\del}(\xi)} |f(w)|^p(1-|w|^2)^{q+ns}d\lambda(w)\\
&\lesssim& \del^{nt-q} \|\mu\|_{\calC \calM_{nt}} \|f\|^p.
\end{eqnarray*}

$\bullet$ \textbf{Estimation of $I_2$.}

Note that by \cite[Proposition 5.1.2]{Rud}, for $j \ge 2, z \in Q_\del(\xi)$ and $w \in A_j$, we have
$$
|1-\langle w, z \rangle|^{\frac{1}{2}} \ge |1-\langle w, \xi \rangle|^{\frac{1}{2}}-|1-\langle z, \xi \rangle|^{\frac{1}{2}} \ge (4^{j-1}\del)^{\frac{1}{2}}-\del^{\frac{1}{2}} \ge 2^{j-2} \del^{\frac{1}{2}}
$$
and for $w \in Q_{4^j\del}(\xi)$,
\begin{equation} \label{ineqSec801}
1-|w|^2 \simeq 1-|w| \le |1-\langle w, \xi \rangle|<4^j\del.
\end{equation}

Moreover, for each $j \ge 2$, consider the term
$$
I_{2, j}= \int_{Q_\del(\xi)} \int_{A_j}\frac{|f(w)|^p(1-|w|^2)^\alpha}{|1-\langle z, w \rangle|^{n+1+\alpha}}dV(w)d\mu(z).
$$
Using Lemma \ref{Carleson01} and \eqref{ineqSec801}, we have
\begin{eqnarray*}
I_{2, j}%
&\lesssim& \frac{1}{(4^j\del)^{n+1+\alpha}} \int_{Q_\del(\xi)} \int_{A_j} |f(w)|^p(1-|w|^2)^\alpha dV(w)d\mu(z) \\
&\lesssim& \frac{(4^j \del)^{\alpha+n+1-q-ns}}{(4^j\del)^{n+1+\alpha}} \int_{Q_\del(\xi)} \int_{Q_{4^j\del}(\xi)} |f(w)|^p(1-|w|^2)^{q+ns} d\lambda(w)d\mu(z) \\
&=& \frac{\mu(Q_\del(\xi))}{(4^j\del)^q} \cdot \frac{1}{(4^j\del)^{ns}} \int_{Q_{4^j \del}(\xi)}  |f(w)|^p(1-|w|^2)^{q+ns} d\lambda(w)\\
&\le& \frac{\del^{nt-q}}{4^{jq}}\|\mu\|_{\calC \calM_{nt}} \|f\|^p.
\end{eqnarray*}
Thus,
$$
I_2=\sum_{j=2}^\infty I_{2, j} \le \sum_{j=2}^\infty  \frac{\del^{nt-q}}{4^{jq}}\|\mu\|_{\calC \calM_{nt}} \|f\|^p \lesssim \del^{nt-q} \|\mu\|_{\calC \calM_{nt}} \|f\|^p.
$$
Combining both estimation of $I_1$ and $I_2$, the proof is complete.
\end{proof}

We are now ready to establish the main result in this section.

\begin{thm} \label{thmSec801}
Let $p \ge 1, q>0, s>\max\left\{0, 1-\frac{q}{n}\right\}, t \ge s+\frac{q}{n}$ and $\mu$ be a positive Borel measure on $\B$. Then the identity operator $$I: \calN(p, q, s) \mapsto T_{p, \left(t-\frac{q}{n} \right)}^\infty(\mu)$$
 is bounded if and only if $\mu$ is an $(nt)$-Carleson measure.
\end{thm}

\begin{proof} {\it Sufficiency.} Suppose $\mu$ is an $(nt)$-Carleson measure. Take some $\alpha>nt-n-1 \ge n\left(s+\frac{q}{n}\right)-n-1=q+ns-n-1>-1$. By \eqref{equation001}, for $f \in \calN(p, q, s)$, we have
$$
|f(z)|^p \lesssim \int_\B \frac{|f(w)|^p(1-|w|^2)^\alpha dV(w)}{|1-\langle z, w \rangle|^{n+1+\alpha}}.
$$
Thus, for any $\xi \in \SSS$ and $\del>0$, by Lemma \ref{lemSec804}, we have
\begin{eqnarray*}
&&\frac{1}{\del^{nt-q}} \int_{Q_\del(\xi)} |f(z)|^pd\mu(z)\\
&\lesssim& \frac{1}{\del^{nt-q}} \int_{Q_\del(\xi)}  \int_\B \frac{|f(w)|^p(1-|w|^2)^\alpha dV(w)}{|1-\langle z, w \rangle|^{n+1+\alpha}}d\mu(z)\\
&\lesssim& \|\mu\|_{\calC\calM_{nt}}\|f\|^p,
\end{eqnarray*}
which implies the desired result.

{\it Necessity.} Suppose the identity operator $I: \calN(p, q, s) \mapsto T^\infty_{p, \left(t-\frac{q}{n}\right)}(\mu)$ is bounded. First, note that $1 \in \calN(p, q, s)$ and hence for any $\del>0$, by the boundedness of $I$, we have
$$
\frac{1}{\del^{n\left(t-\frac{q}{n}\right)}} \int_{Q_\del(\xi)}d\mu \lesssim 1,
$$
which, in particular, implies that $\mu$ is a finite measure on $\B$.

For any $\xi \in \SSS$ and $0<\del<1$, we consider the function
$$
K_{(1-\del)\xi}(z)=\frac{1-|1-\del|^2}{(1-\langle z, (1-\del)\xi \rangle)^{1+\frac{q}{p}}}, \quad z \in \B.
$$
Note that for $z \in Q_\del(\xi)$, we have
\begin{equation} \label{eq1001}
|1- \langle z, (1-\del)\xi \rangle| \le |1-\langle z, \xi \rangle|+|\langle z, \del\xi \rangle| \le 2\del
\end{equation}
and
\begin{equation} \label{eq1002}
|1-\langle z, (1-\del)\xi \rangle| \ge 1-(1-\del) |\langle z, \xi \rangle| \ge 1-(1-\del)|z||\xi| \ge \del,
\end{equation}
which implies that
\begin{equation} \label{equation999}
|K_{(1-\del)\xi}(z)| \simeq \frac{2-\del}{\del^{\frac{q}{p}}}, \quad \xi \in \SSS, \ 0<\del<1.
\end{equation}
Thus, by the boundedness of $I$ and Lemma \ref{lemSec803}, we have
$$
\frac{1}{\del^{n\left(t-\frac{q}{n}\right)}} \int_{Q_\del(\xi)} |K_{(1-\del)\xi}(z)|^pd\mu \lesssim \|K_{(1-\del)\xi)}\|^p \lesssim 1,
$$
which implies
$$
\sup_{\xi \in \SSS, \del>0} \frac{\mu(Q_\del(\xi))}{\del^{nt}} \lesssim 1,
$$
and hence $\mu$ is an $(nt)$-Carleson measure.\end{proof}


\subsection{Behavior of the Riemann-Stieltjes operators on $\calN(p, q, s)$-type spaces}

In this subsection, by using the embedding theorem in the previous subsection, we study the boundedness and compactness of the Riemann-Stieltjes operators on $\calN(p, q, s)$-type spaces.

\begin{thm} \label{Sec8bddness}
Let $p \ge 1, q>0, s_2 \ge s_1>\max\left\{0, 1-\frac{q}{n} \right\}$, $g\in H(\B)$ and $d\mu_{g, p, q, s_2}=|Rg(z)|^p(1-|z|^2)^{p+q+ns_2}d\lambda(z)$. Then
\begin{enumerate}
\item[(1)] $T_g: \calN(p, q, s_1) \mapsto \calN(p, q, s_2)$ is bounded if and only if $\mu_{g, p, q, s_2}$ is an $\left(ns_2+q \right)$-Carleson measure;
\item[(2)] $L_g: \calN(p, q, s_1) \mapsto \calN(p, q, s_2)$ is bounded if and only if $\|g\|_\infty<\infty$;
\item[(3)] $M_g: \calN(p, q, s_1) \mapsto \calN(p, q, s_2)$ is bounded if and only if $\mu_{g, p, q, s_2}$ is an $\left(ns_2+q \right)$-Carleson measure and $\|g\|_\infty<\infty$.
\end{enumerate}
\end{thm}

\begin{proof} (1) Note that $R(T_gf)(z)=f(z)Rg(z)$ and hence for any $\xi \in \SSS$ and $\del>0$, we have
\begin{eqnarray*}
&&\frac{1}{\del^{ns_2}} \int_{Q_\del(\xi)}  |R(T_gf)(z)|^p(1-|z|^2)^{p+q+ns_2}d\lambda(z)\\
& =&\frac{1}{\del^{ns_2}} \int_{Q_\del(\xi)}|f(z)|^p|Rg(z)|^p(1-|z|^2)^{p+q+ns_2}d\lambda(z)\\
& =&\frac{1}{\del^{n\left(s_2+\frac{q}{n}-\frac{q}{n}\right)}} \int_{Q_\del(\xi)}|f(z)|^p|Rg(z)|^p(1-|z|^2)^{p+q+ns_2}d\lambda(z).
\end{eqnarray*}
Thus, the boundedness of $T_g$ is equivalent to the boundedness of $I: \calN(p, q, s_1) \mapsto T_{p, \left(s_2+\frac{q}{n}\right)}^\infty(\mu_{g, p, q, s_2})$, which, by Theorem \ref{thmSec801}, is equivalent to $\|\mu_{g, p, q,  s_2}\|_{\calC\calM_{ns_2+q}}<\infty$.

(2) {\it Sufficiency.} Suppose $\|g\|_\infty<\infty$, then
\begin{eqnarray*}
&&\frac{1}{\del^{ns_2}} \int_{Q_\del(\xi)} |R(L_gf)(z)|^p(1-|z|^2)^{p+q+ns_2}d\lambda(z)\\
&=&\frac{1}{\del^{ns_2}} \int_{Q_\del(\xi)} |g(z)|^p|Rf(z)|^p(1-|z|^2)^{p+q+ns_2}d\lambda(z)\\
&\lesssim& \|g\|_\infty^p \|f\|^p_{\calN(p, q, s_2)}\le\|g\|_\infty^p \|f\|^p_{\calN(p, q, s_1)}.
\end{eqnarray*}
The boundedness of $L_g$ follows by taking the supremum on both sides of the above inequality over all $\xi \in \SSS$ and $\del>0$.

{\it Necessity.} Suppose $L_g: \calN(p, q, s_1) \mapsto \calN(p, q, s_2)$ is bounded. Take and fix $w \in \B$ with $|w|>\frac{2}{3}$, and further, we take $\xi=w/|w|$. Using \cite[Proposition 5.1.2]{Rud}, it is easy to see that there exists a $\del \in (0,1)$, such that
\begin{equation} \label{eq10000}
D\Big(w, \frac{1}{2}\Big) \subseteq Q_\del(\xi) \quad \textrm{and} \quad 1-|w|^2 \simeq \del.
\end{equation}
Recall that $K_w(z)=\frac{1-|w|^2}{(1-\langle z, w \rangle)^{1+\frac{q}{p}}}$ and an easy calculation shows that
\begin{equation} \label{eq100}
R(K_w)(z)=\left(1+\frac{q}{p}\right)\frac{(1-|w|^2)\langle z, w\rangle}{(1-\langle z, w \rangle)^{2+\frac{q}{p}}}.
\end{equation}
Moreover, when $z \in D\left(w, \frac{1}{2}\right)$, we have
\begin{equation} \label{eq101}
V\left( D\Big(w, \frac{1}{2} \Big) \right) \simeq (1-|w|^2)^{n+1},
\end{equation}
and (see, e.g., \cite{LH, Zhu})
\begin{equation} \label{eq102}
1-|w|^2 \simeq 1-|z|^2 \simeq |1-\langle z, w \rangle|, \quad z \in  D\Big(w, \frac{1}{2} \Big).
\end{equation}
 Also note that for $z \in D\left(w, \frac{1}{2} \right)$, we have
$$
1-|\Phi_w(z)|^2=\frac{(1-|w|^2)(1-|z|^2)}{|1-\langle z, w\rangle|^2}>\frac{3}{4},
$$
and hence
\begin{eqnarray*}
1-|\langle z, w \rangle|%
&\le& |1-\langle z, w\rangle|<\frac{2}{\sqrt{3}}(1-|w|^2)^{\frac{1}{2}}(1-|z|^2)^{\frac{1}{2}}\\
&\le&\frac{2}{\sqrt{3}}(1-|w|^2)^{\frac{1}{2}}<\frac{2}{\sqrt{3}}\cdot \frac{\sqrt{5}}{3}=\frac{2\sqrt{15}}{9},
\end{eqnarray*}
which implies $|\langle z, w \rangle|>1-\frac{2\sqrt{15}}{9}$, i.e. $|\langle z, w\rangle| \simeq 1$. Thus, by \eqref{eq100}, \eqref{eq101}, \eqref{eq102}, \cite[Lemma 2.24]{Zhu} and the boundedness of $L_g$, we have
\begin{eqnarray*}
|g(w)|^p%
&\lesssim& \frac{1}{V\left( D\left(w, \frac{1}{2} \right) \right)} \int_{D\left(w, \frac{1}{2} \right)} |g(z)|^pV(z)\\
&\simeq& \frac{1}{(1-|w|^2)^{n+1}}\int_{D\left(w, \frac{1}{2} \right)} |g(z)|^pV(z)\\
&\simeq& \frac{1}{\del^{ns_2}} \int_{D\left(w, \frac{1}{2} \right)}  \frac{|g(z)|^p |\langle z, w \rangle|^p (1-|w|^2)^p}{|1-\langle z, w \rangle|^{2p+q}} (1-|z|^2)^{p+q+ns_2}d\lambda(z)  \\
&\lesssim& \frac{1}{\del^{ns_2}} \int_{Q_\del(\xi)} \frac{|g(z)|^p |\langle z, w \rangle|^p (1-|w|^2)^p}{|1-\langle z, w \rangle|^{2p+q}} (1-|z|^2)^{p+q+ns_2}d\lambda(z)\\
&\simeq& \frac{1}{\del^{ns_2}} \int_{Q_\del(\xi)} |g(z)|^p|R(K_w)(z)|^p(1-|z|^2)^{p+q+ns_2}d\lambda(z)\\
&\lesssim& \|L_g(K_w)\|_{\calN(p, q, s_2)}^p \lesssim \|L_g\|^p \|K_w\|_{\calN(p, q, s_1)}^p \lesssim 1,
\end{eqnarray*}
which implies $|g(w)| \lesssim 1$ for $|w|>\frac{2}{3}$. The desired result then follows from the maximum modulus principle.

\medskip

(3) {\it Sufficiency.} Recall that for  $f \in H(\B)$, we have
$$
M_gf(z)=g(z)f(z)=g(0)f(0)+T_gf(z)+L_gf(z).
$$
Thus, the sufficient part is obvious from (1) and (2), as well as the inequality
\begin{equation} \label{ineqSec805}
|f(z)| \lesssim \frac{\|f\|_{\calN(p, q, s_1)}}{(1-|z|^2)^{\frac{q}{p}}}, \quad z \in \B,
\end{equation}
which follows from \ref{boundaryineq}.

{\it Necessity.} Suppose $M_g: \calN(p, q, s_1) \mapsto \calN(p, q, s_2)$ is bounded. Again, for any fixed $w \in \B$, it is clear that $K_w \in \calN(p, q, s_1)$ and hence $gK_w \in \calN(p, q, s_2)$.  Using \eqref{ineqSec805} and Lemma \ref{lemSec803}, we have
$$
|g(z)K_w(z)| \lesssim \frac{\|M_gK_w\|_{\calN(p, q, s_2)}}{(1-|z|^2)^{\frac{q}{p}}} \lesssim \frac{\|M_g\|}{(1-|z|^2)^{\frac{q}{p}}}, \quad  z \in \B.
$$
Put $z=w$ and we get $|g(w)| \le \|M_g\|$ for all $w \in \B$. Thus, $\|g\|_\infty<\infty$, which, by $(2)$, implies $L_g: \calN(p, q, s_1) \mapsto \calN(p, q, s_2)$ is bounded. Consequently, $T_gf=M_gf-L_gf-f(0)g(0)$ gives the boundedness of $T_g: \calN(p, q, s_1) \mapsto \calN(p, q, s_2)$, which implies $\mu_{g, p, q, s_2}$ is an $\left(ns_2+q \right)$-Carleson measure.
\end{proof}

Note that when $s_1=s_2=s$, clearly if $g \in H^\infty(\B)$, $M_g$ is bounded on $\calN(p, q, s)$. Keeping this trivial but crucial fact in mind, we can refine our result above as follows.

\begin{thm} \label{MultNpqs}
Let $p \ge 1, q>0, s_2 \ge s_1> \max\left\{0, 1-\frac{q}{n} \right\}$ and $g\in H(\B)$ with $\mu_{g, p, q, s_2}$ defined as above. Consider the following statements:
\begin{enumerate}
\item [(1)]  $\|g\|_\infty<\infty$;
\item [(2)] $L_g: \calN(p, q, s_1) \mapsto \calN(p, q, s_2)$ is bounded;
\item [(3)] $M_g: \calN(p, q, s_1) \mapsto \calN(p, q, s_2)$ is bounded;
\item [(4)] $T_g: \calN(p, q, s_1) \mapsto \calN(p, q, s_2)$ is bounded;
\item [(5)] $\mu_{g, p, q, s_2}$ is an $(ns_2+q)$-Carleson measure.
\end{enumerate}
We have $(1) \Longleftrightarrow (2) \Longleftrightarrow (3) \Longrightarrow (4) \Longleftrightarrow (5)$.
\end{thm}

\begin{proof}
Clearly, $(1) \Longleftrightarrow (2)$, $(4) \Longleftrightarrow (5)$, $(3) \Longrightarrow (1)$ and $(3) \Longrightarrow (4)$ are showed in Theorem \ref{Sec8bddness}. It suffices to show $(1) \Longrightarrow (3)$. Indeed, since $\|g\|_\infty<\infty$, we have $M_g$ is bounded on $\calN(p, q, s_2)$, which, by Theorem \ref{Sec8bddness} with $s_1=s_2$, implies that $\mu_{g, p,q, s_2}$ is an $(ns_2+q)$-Carleson measure. Thus, both $L_g$ and $T_g$ are bounded, and hence condition $(3)$ is established.
\end{proof}

To study the compactness, we need the following lemma, whose proof is standard by using Montel's theorem and Fatou's lemma, and hence we omit it here.

\begin{lem} \label{lemSec805}
Let $p \ge 1, q>0$ and $s_2 \ge s_1>\max\left\{0, 1-\frac{q}{n} \right\}$ and $g\in H(\B)$.  Then the following statements are equivalent:
\begin{enumerate}
\item[(i)] $T_g$ (respectively $L_g$) is a compact operator from $\calN(p, q, s_1)$ to $\calN(p, q, s_2)$;
\item[(ii)] For every bounded sequence $\{f_j\}$ in $\calN(p, q, s_1)$ such that $f_j \to 0$ uniformly on compact sets of $\B$, then the sequence $\{T_g(f_j)\}$ (respectively $L_g(f_j)$) converges to zero in the norm of $\calN(p, q, s_2)$.
\end{enumerate}
\end{lem}

For $\xi \in \SSS$  and $\del>0$, set
$$
D_\del'(\xi)=\{ \gamma \in \SSS: |1-\langle \gamma, \xi \rangle|<\del \},
$$
which is known as the nonisotropic metric ball on $\SSS$ with radius $\del^{1/2}$. Note that $D_\del'=\SSS$ when $\del>2$. (see, e.g., \cite[Page 65]{Rud}). Using this conception, we can define the following tent
$$
\widehat{Q}_\del(\xi)=\left\{ z \in \B: \frac{z}{|z|} \in D_\del'(\xi), 1-\del<|z|<1\right\}.
$$
It is known that $Q_\del(\xi) \subset \widehat{Q}_{4\del}(\xi) \subset Q_{16\del}(\xi)$ (see, e.g., \cite[Theorem 4.1.6]{LO}). Hence in the definition of (vanishing) Carleson measure, we can replace $Q_\del(\xi)$ by $\widehat{Q}_\del(\xi)$.

We have the following important covering lemma of $D_\del'(\xi)$ (see, e.g. \cite[Lemma 3.3]{RO2}).

\begin{lem} \label{lemSec806}
Given any natural number $m$, there exists a natural number $N$ such that every non-isotropic ball of radius $\del \le 2$, can be covered by $N$ non-isotropic balls of radius $\del/m$. Moreover, $N$ can be taken as
$$
\frac{\Gamma(n+1)}{4\Gamma^2\left(\frac{n}{2}+1\right)} \cdot \left(2m+\frac{1}{2}\right)^n.
$$
\end{lem}

\begin{lem} \label{lemSec807}
Let $f \in H^\infty$. Then for any $z_1, z_2 \in \B$,
$$
|f(z_1)-f(z_2)| \le 2\|f\|_\infty |\Phi_{z_1}(z_2)|.
$$
\end{lem}
(see, e.g. \cite[Lemma 3.4]{RO2}).

\begin{thm} \label{Sec8compactness}
Let $p \ge 1, q>0, s_2 \ge s_1>\max\left\{0, 1-\frac{q}{n} \right\}$, $g\in H(\B)$ and $d\mu_{g, p, q, s_2}=|Rg(z)|^p(1-|z|^2)^{p+q+ns_2}d\lambda(z)$. Then
\begin{enumerate}
\item[(1)] $T_g: \calN(p, q, s_1) \mapsto \calN(p, q, s_2)$ is compact if and only if $\mu_{g, p, q, s_2}$ is a vanishing $\left(ns_2+q \right)$-Carleson measure;
\item[(2)] $L_g: \calN(p, q, s_1) \mapsto \calN(p, q, s_2)$ is compact if and only if $g=0$;
\item[(3)] $M_g: \calN(p, q, s_1) \mapsto \calN(p, q, s_2)$ is compact if and only if $g=0$.
\end{enumerate}
\end{thm}

\begin{proof}
(1) {\it Sufficiency.} Suppose $\mu_{g, p, q, s_2}$ is a vanishing $\left(ns_2+q \right)$-Carleson measure. Let $\{f_j\}$ be any bounded sequence in $\calN(p, q, s_1)$ and $f_j \to 0$ uniformly on compact sets of $\B$. By Lemma \ref{lemSec805}, it suffices to prove $\lim_{j \to \infty} \|T_gf_j\|_{\calN(p, q, s_2)}=0$.

Let $\chi_E$ denotes the characteristic function of a set $E$ of $\B$. For $r \in (0, 1)$, define the cut-off measure $d\mu_r=\chi_{\{z \in \B: |z|>r\}}d\mu_{g, p, q, s_2}$ and for fixed $\xi \in \SSS$ and $0<\del<1$, we have
\begin{eqnarray*}
&&\frac{1}{\del^{ns_2}} \int_{Q_\del(\xi)} |R(T_gf_j)(z)|^p(1-|z|^2)^{p+q+ns}d\lambda(z)\\
&=& \frac{1}{\del^{ns_2}} \int_{Q_\del(\xi)} |f_j(z)|^p|Rg(z)|^p(1-|z|^2)^{p+q+ns}d\lambda(z)\\
&=& \frac{1}{\del^{ns_2}} \int_{Q_\del(\xi)} |f_j(z)|^pd\mu_{g, p, q, s_2}(z)\\
&=& \frac{1}{\del^{ns_2}} \int_{Q_\del(\xi)} |f_j(z)|^p\chi_{\{z \in \B: |z| \le r\}}d\mu_{g, p, q, s_2}(z)+\frac{1}{\del^{ns_2}} \int_{Q_\del(\xi)} |f_j(z)|^pd\mu_r(z)\\
&=& J_{1. r}+J_{2, r}.
\end{eqnarray*}

By the proof of Theorem \ref{thmSec801}, we have for any $j \in \N$,
\begin{eqnarray*}
J_{2, r}%
&=&\frac{1}{\del^{n\left(s_2+\frac{q}{n}-\frac{q}{n}\right)}} \int_{Q_\del(\xi)} |f_j(z)|^pd\mu_r(z)\\
&\lesssim& \|\mu_r\|_{\calC\calM_{ns_2+q}} \|f_j\|^p_{\calN(p, q, s_1)} \lesssim \|\mu_r\|_{\calC\calM_{ns_2+q}}.
\end{eqnarray*}
\textbf{Claim:}
$$
\|\mu_r\|_{\calC\calM_{ns_2+q}} \to 0 \quad \textrm{as} \quad r \to 1^{-}.
$$
Indeed, for any $\varepsilon>0$, by our assumption, there exists a $\del_0>0$, such that
$$
\mu_{g, p, q, s_2}\left(\widehat{Q}_{\del_1}(\xi)\right)<\varepsilon {\del_1}^{ns_2+q},
$$
for all $\del_1 \le \del_0$ and $\xi \in \SSS$ uniformly. If $\del \le \del_0$, it is clear that
\begin{equation} \label{ineq200}
\mu_r\left(\widehat{Q}_{\del}(\xi)\right) \le \mu_{g, p, q, s_2}\left(\widehat{Q}_{\del}(\xi)\right)<\varepsilon {\del}^{ns_2+q}.
\end{equation}
If $\del>\del_0$, take and fix $m=\left\lfloor \frac{\del}{\del_0} \right\rfloor+1 \le \frac{2\del}{\del_0}$, where $\lfloor\cdot\rfloor$ is the floor function. Note that $\frac{\del}{m}<\del_0$. Then by Lemma \ref{lemSec806}, we have $Q_\del'$ can be covered by $N$ balls $Q'_{\del/m}$ on $\SSS$ with $N \simeq m^n$. Thus, by the definition of $\widehat{Q}_\del(\xi)$, it follows that
$$
\widehat{Q}_\del \cap \left\{z \in \B: |z|>1-\frac{\del_0}{m} \right\} \subset \bigcup_N \widehat{Q}_{\del/m}.
$$
Putting $r_0=1-\frac{\del_0}{m}$, and for $r_0<r<1$,  we have
\begin{eqnarray*}
\mu_r(\widehat{Q}_\del(\xi))%
&\le& \mu_r \left( \bigcup_N \widehat{Q}_{\del/m} \right) \le \mu_r \left(\bigcup_N \widehat{Q}_{\del_0} \right) \\
&\le& \sum_N  \mu_r (\widehat{Q}_{\del_0}) \le \sum_N \mu_{g, p, q, s_2} (\widehat{Q}_{\del_0}) \\
&\le& N \varepsilon \del_0^{ns_2+q} \lesssim \varepsilon \del_0^{ns_2+q} \cdot \frac{\del^n}{\del_0^n} \le \varepsilon \del^{ns_2+q},
\end{eqnarray*}
where in the last inequality, we use the fact that $s_2>1-\frac{q}{n}$. Combining this estimation with \eqref{ineq200}, we prove the claim.

Now for any $\varepsilon>0$, by the above claim and estimation on $J_{2, r}$, there exists a $r_1 \in (0, 1)$, such that when $r_1 \le r<1$ and $j \in \N$, we have
$$
J_{2, r}=\frac{1}{\del^{ns_2}} \int_{Q_\del(\xi)} |f_j(z)|^pd\mu_r(z)<\varepsilon.
$$
Fix $r_1$. Noting that for $J_{1, r_1}$, we have
\begin{eqnarray*}
J_{1, r_1}%
&=& \frac{1}{\del^{ns_2}} \int_{Q_\del(\xi)} |f_j(z)|^p|Rg(z)|^p\chi_{\{z \in \B: |z| \le r_1\}}(1-|z|^2)^{p+q+ns}d\lambda(z) \\
&=& \sup_{|z| \le r_1} |f_j(z)Rg(z)|^p \cdot \frac{1}{\del^{ns_2}} \int_{Q_\del(\xi)} (1-|z|^2)^{q+ns}d\lambda(z)\\
&\lesssim& \|1\|_{\calN(p, q, s_2)}^p \cdot \sup_{|z| \le r_1} |f_j(z)Rg(z)|^p\\
&\lesssim& \sup_{|z| \le r_1} |f_j(z)Rg(z)|^p,
\end{eqnarray*}
where in the last inequality, we use \cite[Propostion 2.8]{HL} and $1$ refers to the constant function $f(z)=1, \forall z \in \B$. Since $f_j \to 0$ as $j \to \infty$ uniformly on the compact subset of $\B$, there exists some $j_0 \in \N$, such that when $j>j_0$, we have $J_{1, r_1} \lesssim \varepsilon$.
Combining the above estimation on $J_{1, r_1}$ and $J_{2, r_2}$, we get the desired result.

{\it Necessity.} Suppose $T_g: \calN(p, q, s_1) \mapsto \calN(p, q, s_2)$ is compact. For any $\xi \in \SSS, \del_j \to 0$ as $j \to \infty$, we consider the functions
$$
K_{(1-\del_j)\xi}(z)=\frac{1-|1-\del_j|^2}{(1-\langle z, (1-\del_j)\xi)^{1+\frac{q}{p}}}, \quad z \in \B.
$$
By Lemma \ref{lemSec803}, it is clear that $\sup_{j \in \N} \|K_{(1-\del_j)\xi}\|_{\calN(p, q, s_1)} \lesssim 1$. Moreover, it is easy to see that $K_{(1-\del_j)\xi} \to 0$ uniformly on compact subsets of $\B$ as $j \to \infty$.  Thus, using \eqref{equation999}, we have
\begin{eqnarray*}
&&\frac{\mu_{g, p, q, s_2}(Q_{\del_j}(\xi))}{\del_j^{ns_2+q}}\\
&=& \frac{1}{\del_j^{ns_2+q}} \int_{Q_{\del_j}(\xi)} d\mu_{g, p, q, s_2}(z)\\
&\simeq& \frac{1}{\del_j^{ns_2}} \int_{Q_{\del_j}(\xi)} |K_{(1-\del_j)\xi}(z)|^p d\mu_{g, p, q, s_2}(z)\\
&=&\frac{1}{\del_j^{ns_2}} \int_{Q_{\del_j}(\xi)} |K_{(1-\del_j)\xi}(z)|^p |Rg(z)|^p(1-|z|^2)^{p+q+ns}d\lambda(z)\\
&=&\frac{1}{\del_j^{ns_2}} \int_{Q_{\del_j}(\xi)} |R(T_g K_{(1-\del_j)\xi})(z)|^p(1-|z|^2)^{p+q+ns}d\lambda(z)\\
&\lesssim& \|T_g(K_{(1-\del_j)\xi})\|_{\calN(p, q, s_2)},
\end{eqnarray*}
which, by Lemma \ref{lemSec805}, converges to $0$ as $j \to \infty$. Thus, $\mu_{g, p, q, s_2}$ is a vanishing $\left(ns_2+q \right)$-Carleson measure.

\medskip

(2) The sufficiency is obvious and we only verify the necessity. By Theorem \ref{Sec8bddness}, (ii), the compactness of $L_g$ implies that $g \in H^\infty$.

We prove the statement by contradiction. Assume that $g$ is not identically equal to $0$, i.e., there exists some $w_0 \in \B$ such that $|g(w_0)|=\varepsilon_0>0$. By Lemma \ref{lemSec807}, we have
$$
|g(z_1)-g(z_2)| \le 2 \|g\|_\infty |\Phi_{z_1}(z_2)|, \quad z_1, z_2 \in \B.
$$
This inequality implies that there is a sufficient small $r>0$ such that for any $a \in \B$, $|g(z)| \ge \frac{\varepsilon_0}{2}$ for all  $z$ satisfying $|\Phi_a(z)|<r$. Fix the $r$ chosen above and note that the choice of $r$ only depends on $g$.

By the maximum modulus principle, we can take a sequence $\{w_m\}_{m \ge 1}$ of points in $\B$ with $|w_m| \to 1$ as $m \to \infty$, such that
$$
\max\left\{|w_0|, r^{1/2} \right\}<|w_1|<\dots<|w_m|<\dots<1,
$$
and
$$
 |g(w_m)| \ge \varepsilon, \ \forall m \ge 1.
$$
 Putting $\xi_j=\frac{w_j}{|w_j|}, j \in \N$ and applying the same argument in \eqref{eq10000}, we can find a sequence $\{\del_j\}_{j \in \N}$ such that
$$
D(w_j, r) \subseteq Q_{\del_j}(\xi_j) \quad \textrm{and} \quad 1-|w_j|^2 \simeq \del_j, \ j \in \N.
$$
Thus, the above argument implies for each $j \in \N$, $|g(z)|\ge \frac{\varepsilon_0}{2}$ for those $z$ satisfying $|\Phi_{w_j}(z)|<r$.

Again, we consider the test functions
$$
K_j(z)=\frac{1-|w_j|^2}{(1-\langle z, w_j\rangle)^{1+\frac{q}{p}}}, \quad z \in \B, j \in \N.
$$
It is clear that $\sup_{j \in \N} \|K_j\|_{\calN(p, q, s_1)} \lesssim 1$ and $K_j$ converges to $0$ as $j \to \infty$ uniformly on compact subset of $\B$. Thus, by Lemma \ref{lemSec805}, we have $\|L_g(K_j)\|_{\calN(p, q, s_2)} \to 0$ as $j \to \infty$.

Note that $|1-\langle z, w_j \rangle| \simeq \del_j$ for $z \in Q_{\del_j}(\xi_j)$ and $|\langle z, w_j \rangle| \simeq 1$ for $z \in D(w_j, r)$. Indeed, the first assertion follows from \eqref{eq1001} and \eqref{eq1002} and for the second one, we have for $z \in D(w_j, r)$,
$$
1-|\Phi_{w_j}(z)|^2=\frac{(1-|w_j|^2)(1-|z|^2)}{|1-\langle z, w_j \rangle|^2}>1-r^2,
$$
and hence
\begin{eqnarray*}
1-|\langle z, w_j \rangle|%
&\le& |1-\langle z, w_j \rangle|<\frac{1}{\sqrt{1-r^2}}(1-|w_j|^2)^{\frac{1}{2}}(1-|z|^2)^{\frac{1}{2}}\\
&\le& \frac{1}{\sqrt{1-r^2}}(1-|w_j|^2)^{\frac{1}{2}} \le \sqrt{\frac{1-r}{1-r^2}}=\sqrt{\frac{1}{1+r}},
\end{eqnarray*}
which implies the desired assertion.

Thus, for each $j \in \N$, by \eqref{eq100}, we have
\begin{eqnarray*}
&&\|L_g(K_j)\|_{\calN(p, q, s_2)}^p\\
&\gtrsim &  \frac{1}{\del_j^{ns_2}} \int_{Q_{\del_j}(\xi_j)} |R(K_j)(z)|^p |g(z)|^p(1-|z|^2)^{p+q+ns}d\lambda(z)\\
&\simeq& \frac{1}{\del_j^{ns_2}} \int_{Q_{\del_j}(\xi_j)} \frac{(1-|w_j|^2)^p |\langle z, w_j\rangle|^p}{|1-\langle z, w_j\rangle|^{2p+q}} |g(z)|^p(1-|z|^2)^{p+q+ns_2}d\lambda(z)\\
&\ge& \frac{1}{\del_j^{ns_2}} \int_{|\Phi_{w_j}(z)|<r} \frac{(1-|w_j|^2)^p |\langle z, w_j\rangle|^p}{|1-\langle z, w_j\rangle|^{2p+q}} |g(z)|^p(1-|z|^2)^{p+q+ns_2}d\lambda(z)\\
&\gtrsim & \left(\frac{\varepsilon_0}{2} \right)^p \cdot \frac{(1-|w_j|^2)^{ns_2-n-1}}{\del_j^{ns_2}} \cdot V(D(w_j, r))  \gtrsim \left(\frac{\varepsilon_0}{2} \right)^p,
\end{eqnarray*}
which contradicts to Lemma \ref{lemSec805}.

\medskip

(3) Again the sufficiency is obvious and for the necessity, we suppose that $M_g: \calN(p, q, s_1) \mapsto \calN(p, q, s_2)$ is compact. Then by Theorem \ref{Sec8bddness}, we have $g \in H^\infty$. Let $\{w_j\}_{j \ge 1}$ be a sequence in $\B$ such that $|w_j| \to 1$, and
$$
K_j(z)=\frac{1-|w_j|^2}{(1-\langle z, w_j\rangle)^{1+\frac{q}{p}}}, \quad z \in \B, j \in \N
$$
as above. Then  $\sup_{j \in \N} \|K_j\|_{\calN(p, q, s_1)} \lesssim 1$ and $K_j$ converges to $0$ as $j \to \infty$ uniformly on compact subset of $\B$. Thus, $\|M_g(K_j)\|_{\calN(p, q, s_2)} \to 0$ as $j \to \infty$ since $M_g$ is compact. Since
$$
|g(z)K_j(z)|=|M_g(K_j)(z)| \lesssim \frac{\|M_g(K_j)\|_{\calN(p, q, s_2)}}{(1-|z|^2)^{\frac{q}{p}}}, \quad \forall z \in \B,
$$
by letting $z=w_j$, we get
$$
|g(w_j)| \lesssim \|M_g(K_j)\|_{\calN(p, q, s_2)},
$$
hence $g(w_j) \to 0$ as $j \to \infty$. Since $g$ is bounded holomorphic function on $\B$, it follows that $g=0$.
\end{proof}


\subsection{Multipliers of $\calN(p, q, s)$-type spaces}


In this subsection, we discuss the pointwise multipliers of the $\calN(p, q, s)$-type spaces. Let $X, Y$ be two spaces of holomorphic functions in $\B$. We call $\varphi \in H(\B)$ a pointwise multiplier from $X$ to $Y$ if
$$
\varphi f \in Y
$$
for all $f \in X$. The collection of all such functions $\varphi$ is denoted by $M(X, Y)$. When $X=Y$, the set $M(X, Y)$ is denoted simply by $M(X)$. It is clear that from Theorem \ref{MultNpqs}, we have
\begin{equation} \label{Sec803300}
M(\calN(p, q, s))=H^\infty.
\end{equation}
We are interested in the following question: what can we say about $M(X, \calN(p, q, s))$ if $X$ is replaced by another function space?

Recall that in Corollary \ref{embeddingBercor}, we have shown that if $q>n$, then $A_{q-n-1}^p \subseteq \calN(p, q, s)$. Thus, it turns out that it is a natural choice if we consider the case when $X$ is the Bergman space $A^p_\alpha$, in particular for the case $\alpha=p-n-1$. However, before we go further, we would like to remove the restriction $\alpha>-1$ in the definition of the Bergman space first  with considering a general type of Bergman space.

For any positive $p$ and real $\alpha$, we let $N$ be the smallest nonnegative integer such that $pN+\alpha>-1$ and recall that the \emph{general Bergman space} (still denote it as $A_\alpha^p$) is defined as the collection of all $f \in H(\B)$ such that
$$
\|f\|_{p, \alpha}:=|f(0)|+ \left( \int_\B (1-|z|^2)^{pN}|R^Nf(z)|^pdV_\alpha(z) \right)^{1/p}<\infty.
$$
Once again, the $A_\alpha^p$ becomes a Banach space when $p \ge 1$ and a complete metric space when $0<p<1$. Moreover, for $\beta$ real, let $k$ be the smallest nonnegative integer greater than $\beta$ and define the \emph{holomorphic Lipschitz spaces} by
$$
\Lambda_\beta:=\left\{ f \in H(\B):\|f\|_{\Lambda_\beta}=|f(0)|+\sup_{z \in \B} (1-|z|^2)^{k-\beta} |R^kf(z)|<\infty \right\}.
$$

We have the following observation, which generalizes our previous Corollary \ref{embeddingBercor}.

\begin{lem} \label{motiSec803}
Let $p \ge 1, q>0$, $s>\max\left\{0, 1-\frac{q}{n} \right\}$, $m$ the smallest nonnegative integer such that $mp+q-n>0$ and $0<\varepsilon<mp+q-n$. Then we have
$$
\Lambda_{\frac{n+\varepsilon-q}{p}} \subseteq A^p_{q-n-1} \subseteq \bigcap_{s>\max\left\{0, 1-\frac{q}{n} \right\}} \calN^0(p, q, s).
$$
\end{lem}

\begin{proof}
For the first inclusion, first we note that $m$ is the smallest nonnegative integer such that $m-\frac{n+\varepsilon-q}{p}>0$. Then we have
\begin{eqnarray*}
\|f\|_{p, q-n-1}^p%
&\simeq& |f(0)|^p+\int_\B |R^mf(z)|^p(1-|z|^2)^{mp+q-n-1}dV(z)\\
&=& |f(0)|^p+\int_\B |R^mf(z)|^p(1-|z|^2)^{mp+q-n-\varepsilon} (1-|z|^2)^{\varepsilon-1}dV(z)\\
&\lesssim& \|f\|_{\Lambda_{\frac{n+\varepsilon-q}{p}}},
\end{eqnarray*}
which implies the desired result.

Let $f \in A^p_{q-n-1}$. We have to show $f \in \bigcap\limits_{s>\max\left\{0, 1-\frac{q}{n} \right\}} \calN^0(p, q, s)$. Clearly, by Theorem \ref{vanishingeqnorm}, it suffices to show for any $s>\max\left\{0, 1-\frac{q}{n} \right\}$,
$$
|R^mf(z)|^p (1-|z|^2)^{mp+q+ns}d\lambda(z)
$$
is a vanishing $(ns)$-Carleson measure. Thus, for any $\xi \in \SSS$ and $\del \in \left(0, \frac{1}{100} \right)$, we have
\begin{eqnarray*}
&&\frac{1}{\del^{ns}} \int_{Q_\del(\xi)} |R^m f(z)|^p(1-|z|^2)^{mp+q+ns}d\lambda(z)\\
&\le& \frac{1}{\del^{ns}} \int_{\widehat{Q}_{4\del}(\xi)} |R^m f(z)|^p(1-|z|^2)^{mp+q+ns}d\lambda(z) \\
&\lesssim& \int_{1-4\del}^1 r^{2n-1} (1-r^2)^{mp+q-n-1} \left(\int_{D'_{4\del}(\xi)} |R^m f(r\xi)|^pd\sigma(\xi)\right) dr\\
&=&  \int_{\widehat{Q}_{4\del}(\xi)} |R^m f(z)|^p(1-|z|^2)^{mp+q-n-1}dV(z),
\end{eqnarray*}
which clearly converges to zero uniformly with respect to all $\xi \in \SSS$ as $\del \to 0$. The proof is complete.
\end{proof}

\begin{rem}
From the above lemma, it is clear that by \eqref{Sec803300}
\begin{equation} \label{Sec803400}
H^\infty \subseteq M(A_{q-n-1}^p, \calN(p, q, s)) \subseteq M(\Lambda_{\frac{n+\varepsilon-q}{p}}, \calN(p, q, s) ).
\end{equation}
\end{rem}


\subsubsection{\textbf{The space $M(A^p_\alpha, \calN(p, q, s))$}}


\begin{lem} \label{atomicBergman}
Let $p \ge 1$ and $\alpha, b \in \R$ satisfying that $b$ is neither $0$ nor a negative integer, and
$$
b>n+\frac{\alpha+1}{p}.
$$
Let further, for each $w \in \B$,
$$
J_w(z)=\frac{(1-|w|^2)^{b-\frac{n+1+\alpha}{p}}}{(1-\langle z, w \rangle)^b}.
$$
Then $\sup_{w \in \B} \|J_w\|_{p, \alpha} \le 1$.
\end{lem}

\begin{proof}
The above lemma is an easy consequence of the atomic decomposition for $A^p_\alpha$ (see, e.g., \cite[Theorem 32]{ZZ}) and hence the proof is omitted here.
\end{proof}

The following result gives a general description of the space of multipliers $M(A_\alpha^p, \calN(p, q, s))$.

\begin{thm} \label{4casesMult}
Let $q>0$ and $s>\max\left\{0, 1-\frac{q}{n} \right\}$. Then we have the following assertions:
\begin{enumerate}
\item [(1)] If $p \ge 1$ and $n+1+\alpha<0$, then
$$
M(A^p_\alpha, \calN(p, q, s))=\calN(p, q, s).
$$

\item [(2)] If $p=1$ and $n+1+\alpha=0$, then
$$
M(A^p_\alpha, \calN(p, q, s))=\calN(p, q, s).
$$

\item [(3)] If $p>1$ and $n+1+\alpha=0$, then
$$
M(A_\alpha^p, \calN(p, q, s)) \subseteq \calN(p, q, s).
$$
Conversely, if $\varphi \in H(\B)$ satisfies
\begin{equation} \label{Sec803100}
\sup_{a \in \B} \int_\B |\varphi(z)|^p\left(\log \frac{2}{1-|z|^2} \right)^{p-1} (1-|z|^2)^q(1-|\Phi_a(z)|^2)^{ns}d\lambda(z)<\infty,
\end{equation}
then we have $\varphi \in M(A^p_\alpha, \calN(p, q, s))$.

\item [(4)] If $p \ge 1$ and $n+1+\alpha>0$, then
\begin{enumerate}
\item[(a)] If $n+1+\alpha>q$, then $M(A^p_\alpha, \calN(p, q, s))$ only contains the constant function $\varphi(z) \equiv 0$;
\item[(b)] If $n+1+\alpha=q$, then $M(A^p_\alpha, \calN(p, q, s))=H^\infty$;
\item[(c)] If $n+1+\alpha<q$, then
$$
\calN(p, q-n-1-\alpha, s) \subseteq M(A^p_\alpha, \calN(p, q, s)) \subseteq \calN(p, q, s) \cap A^{-\frac{q-(n+1+\alpha)}{p}}(\B).
$$
\end{enumerate}
\end{enumerate}
\end{thm}

\begin{proof}
$(1) \& (2)$ It is clear that $M(A_\alpha^p, \calN(p, q, s)) \subseteq \calN(p, q, s)$ since the constant function $\varphi(z) \equiv 1$ is in $A^p_\alpha$.

Conversely, by \cite[Thereom 21, 22(a)]{ZZ}, we know that all the functions in $A^p_\alpha$ are bounded under the assumption of $(1)$ and $(2)$, which implies the desired result.

\medskip

$(3)$ The first claim is clear. Now take any $\varphi \in H(\B)$ satisfying \eqref{Sec803100}. To prove $\varphi \in M(A^p_\alpha, \calN(p, q, s))$, it suffices to show that $\varphi f \in \calN(p, q, s)$ for any $f \in A^P_{-n-1}$. Indeed, by \cite[Theroem 22(b)]{ZZ}, we have for any $\xi \in \SSS$ and $\del>0$,
\begin{eqnarray*}
&&\frac{1}{\del^{ns}} \int_{Q_\del(\xi)} |\varphi(z) f(z)|^p (1-|z|^2)^{q+ns}d\lambda(z)\\
&\lesssim & \frac{1}{\del^{ns}} \int_{Q_\del(\xi)} |\varphi(z)|^p \left(\log \frac{2}{1-|z|^2} \right)^{p-1}(1-|z|^2)^{q+ns}d\lambda(z)\\
&<&\infty,
\end{eqnarray*}
where the last inequality follows from \eqref{Sec803100} and Lemma \ref{lemSec802}. By \eqref{Carleson}, we get the desired result.

\medskip

$(4)$ Take $g \in M(A^p_\alpha, \calN(p, q, s))$. By Lemma \ref{atomicBergman},  we have $gJ_w \in \calN(p, q, s), \forall w \in \B$, and
$$
\|gJ_w\| \le \|M_g\| \|J_w\|_{p, \alpha} \le \|M_g\|.
$$
Thus, by Proposition \ref{boundaryineq}, for each fixed $w \in \B$,
$$
|g(z)J_w(z)| \lesssim \frac{\|gJ_w\|}{(1-|z|^2)^{\frac{q}{p}}} \lesssim \frac{1}{(1-|z|^2)^{\frac{q}{p}}}, \quad \forall z \in \B.
$$
Let $z=w$. We have
\begin{equation} \label{Sec803200}
|g(w)| (1-|w|^2)^{\frac{q-n-1-\alpha}{p}} \lesssim 1.
\end{equation}

\medskip

$(4a)$ The desired claim follows from the maximum modulus principle since it is clear that $\lim\limits_{|w| \to 1^{-}} |g(w)|=0$ by \eqref{Sec803200}.

\medskip

$(4b)$ From \eqref{Sec803400}, we already know that $H^\infty \subseteq M(A^p_{q-n-1}, \calN(p, q, s))$, while $M(A^p_{q-n-1}, \calN(p, q, s)) \subseteq H^\infty$ clearly follows from \eqref{Sec803200}.

\medskip

$(4c)$ The second inclusion clearly follows from \eqref{Sec803200} and the fact that the constant function $\varphi(z) \equiv 1$ belongs to $A^p_\alpha$. We have to show the first inclusion, that is , for any fixed $g \in \calN(p, q-n-1-\alpha, s)$, the multiplication operator $M_g$ is bounded from $A^p_\alpha$ to $\calN(p, q, s)$.

Indeed, for any $f \in A^p_\alpha$, we have
\begin{eqnarray*}
&&\|gf\|_{\calN(p, q, s)}^p=\sup_{a \in \B} \int_\B |f(z)g(z)|^p(1-|z|^2)^q(1-|\Phi_a(z)|^2)^{ns}d\lambda(z)\\
&&\lesssim  \sup_{a \in \B} \int_\B |g(z)|^p(1-|z|^2)^q \cdot \frac{\|f\|_{p, \alpha}^p}{(1-|z|^2)^{n+1+\alpha}}(1-|\Phi_a(z)|^2)^{ns}d\lambda(z)\\
&& \le \|g\|^p_{\calN(p, q-n-1-\alpha, s)}\|f\|_{p, \alpha}^p.
\end{eqnarray*}
Here, in the first inequality, we use the following estimation
$$
|f(z)| \le \frac{C\|f\|_{p, \alpha}}{(1-|z|^2)^{\frac{n+1+\alpha}{p}}},
$$
where $f \in A^p_\alpha$ where $p>0$ and $n+1+\alpha>0$ (see, e.g. \cite[Theroem 20]{ZZ}).
\end{proof}

Similarly, we have the following description on the multipliers between $A^p_\alpha$ and $\calN^0(p, q, s)$.

\begin{thm}
Let $q>0$ and $s>\max\left\{0, 1-\frac{q}{n} \right\}$. Then we have the following assertions:
\begin{enumerate}
\item [(1)] If $p \ge 1$ and $n+1+\alpha<0$, then
$$
M(A^p_\alpha, \calN^0(p, q, s))=\calN^0(p, q, s).
$$

\item [(2)] If $p=1$ and $n+1+\alpha=0$, then
$$
M(A^p_\alpha, \calN^0(p, q, s))=\calN^0(p, q, s).
$$

\item [(3)] If $p>1$ and $n+1+\alpha=0$, then
$$
M(A_\alpha^p, \calN^0(p, q, s)) \subseteq \calN^0(p, q, s).
$$
Conversely, if $\varphi \in H(\B)$ satisfies \eqref{Sec803100} and
$$
\lim_{|a| \to 1} \int_\B |\varphi(z)|^p\left(\log \frac{2}{1-|z|^2} \right)^{p-1} (1-|z|^2)^q(1-|\Phi_a(z)|^2)^{ns}d\lambda(z)=0,
$$
then we have $\varphi \in M(A^p_\alpha, \calN^0(p, q, s))$.

\medskip

\item [(4)] If $p \ge 1$ and $n+1+\alpha>0$, then
\begin{enumerate}
\item[(a)] If $n+1+\alpha>q$, then $M(A^p_\alpha, \calN^0(p, q, s))$ only contains the constant function $\varphi(z) \equiv 0$;
\item[(b)] If $n+1+\alpha=q$, then $M(A^p_\alpha, \calN^0(p, q, s))=H^\infty$;
\item[(c)] If $n+1+\alpha<q$, then
$$
\calN^0(p, q-n-1-\alpha, s) \subseteq M(A^p_\alpha, \calN^0(p, q, s)) \subseteq \calN^0(p, q, s) \cap A^{-\frac{q-(n+1+\alpha)}{p}}(\B).
$$
\end{enumerate}
\end{enumerate}
\end{thm}

\begin{proof}
The proof for the above theorem is an easy modification of Theorem \ref{4casesMult}, and hence is omitted here.
\end{proof}

\begin{rem}
From the above theorem, we can conclude that $H^\infty=M(\calN^0(p, q, s))$. Indeed, $H^\infty \subseteq M(\calN^0(p, q, s)) $ is clearly by the definition. Conversely, by the above theorem, we have
$$
M(\calN^0(p, q, s)) \subseteq M(A^p_{q-n-1}, \calN^0(p, q, s))=H^\infty,
$$
which implies the desired result.
\end{rem}


\subsubsection{\textbf{The space $M(\Lambda_\beta, \calN(p, q, s))$}}


Next, we study the space of multipliers between $\Lambda_\beta$ and $\calN(p, q, s)$-type spaces. We need the following lemma, which gives an integral  representation of the functions in $\Lambda_\beta$.

\begin{lem} \cite[Theorem 17]{ZZ} \label{holoLipspa}
Suppose $f\in H(\B)$ and $\beta$ is real. If $\re \gamma>-1$ and $n+\gamma-\beta$ is not a negative integer, then $f \in \Lambda_\beta$ if and only if there exists a function $g \in L^\infty(\B)$ such that
$$
f(z)=\int_\B \frac{g(w)dV_\gamma(w)}{(1-\langle z, w \rangle)^{n+1+\gamma-\beta}}
$$
for $z \in \B$.
\end{lem}

\begin{thm}
Let $p \ge 1, q>0$ and $s>\max\left\{0, 1-\frac{q}{n} \right\}$. Then we have the following results.
\begin{enumerate}
\item[(1)] If $\beta>0$, then $M(\Lambda_\beta, \calN(p, q, s))=\calN(p, q, s)$.
\item[(2)] If $\beta=0$, then $M(\Lambda_\beta, \calN(p, q, s)) \subseteq \calN(p, q, s)$. Conversely, if $\varphi \in H(\B)$ satisfying
\begin{equation} \label{Sec803101}
\sup_{a \in \B} \int_\B |\varphi(z)|^p\left(\log \frac{2}{1-|z|^2} \right)^p (1-|z|^2)^q(1-|\Phi_a(z)|^2)^{ns}d\lambda(z)<\infty,
\end{equation}
then $\varphi \in M(\Lambda_\beta, \calN(p, q, s))$.
\end{enumerate}
\end{thm}

\begin{proof}
$(1)$ Once again, since clearly the constant function $\varphi(z) \equiv 1$ belongs to $\Lambda_\beta$, we have $M(\Lambda_\beta, \calN(p, q, s)) \subseteq \calN(p, q, s)$.

Let $\varphi \in \Lambda_\beta$. We claim that $\varphi \in H^\infty$. Indeed, by Lemma \ref{holoLipspa} and \cite[Proposition 1.4.10]{Rud}, we have
$$
|\varphi(z)| \lesssim \int_\B \frac{(1-|w|^2)^\gamma dV(z)}{(1-\langle z, w\rangle)^{n+1+\gamma-\beta}}<\infty,
$$
where $\gamma$ is some constant satisfying the requirement in Lemma \ref{holoLipspa}. This clearly implies that for any $f \in \calN(p, q, s)$,  we have $f\varphi \in \calN(p, q, s)$. The proof is complete.

\medskip

$(2)$ The first assertion is clear. To show the second assertion, first we note that when $\beta=0$, $\Gamma_\beta$=$\calB$, the classical Bloch space. Let $\varphi \in H(\B)$ satisfying \eqref{Sec803101} and $f \in \calB$. We need to show that $f\varphi \in \calN(p, q, s)$. First we claim that $\varphi \in \calN(p, q, s)$. Indeed,
\begin{eqnarray*}
&&\|\varphi\|^p= \sup_{a \in \B} \int_\B |\varphi(z)|^2(1-|z|^2)^q(1-|\Phi_a(z)|^2)^{ns}d\lambda(z)\\
&=&\sup_{a \in \B} \left(\int_{|z|<0.9}+\int_{0.9<|z|<1} \right) |\varphi(z)|^p(1-|z|^2)^q(1-|\Phi_a(z)|^2)^{ns}d\lambda(z) \\
&\lesssim& \sup_{a \in \B} \int_\B |\varphi(z)|^p\left(\log \frac{2}{1-|z|^2} \right)^p (1-|z|^2)^q(1-|\Phi_a(z)|^2)^{ns}d\lambda(z)<\infty.
\end{eqnarray*}
Thus, by the growth estimation of  $f \in \calB$, we have
\begin{eqnarray*}
&& \sup_{a \in \B}\int_\B |f(z)\varphi(z)|^p(1-|z|^2)^q(1-|\Phi_a(z)|^2)^{ns}d\lambda(z)\\
&\lesssim& \sup_{a \in \B} \int_\B |\varphi(z)|^p\left(\log \frac{2}{1-|z|^2} \right)^p (1-|z|^2)^q(1-|\Phi_a(z)|^2)^{ns}d\lambda(z) \\
&<&\infty.
\end{eqnarray*}
 We get the desired result.
\end{proof}

Finally, we consider the case when $\beta<0$. Let $l=-\beta$. Note that by the definition of $\Lambda_\beta$, in the present case, we have
$$
\Lambda_\beta=A^{-l}(\B).
$$
Thus, in the sequel, instead of using the notation $M(\Lambda_\beta, \calN(p, q, s))$, we consider the space $M(A^{-l}(\B), \calN(p, q, s))$. First we note that for each $w \in \B$, the function
$$
L_w=\frac{(1-|w|^2)^l}{(1-\langle z, w\rangle)^{2l}}
$$
belongs to $A^{-l}(\B)$ and $\sup\limits_{w \in \B} |L_w|_l \le 1$.

We have the following result.

\begin{thm}
Let $p \ge 1, q>0, s>\max\left\{0, 1-\frac{q}{n} \right\}$ and $l>0$. Then we have the following assertions:
\begin{enumerate}
\item[(1)] If $l>\frac{q}{p}$, then $M(A^{-l}(\B), \calN(p, q, s))$ only contains the constant function $\varphi(z) \equiv 0$.
\item[(2)] If $l=\frac{q}{p}$, then $M(A^{-l}(\B), \calN(p, q, s))=H^\infty$.
\item[(3)] If $0<l<\frac{q}{p}$, then
$$
\calN(p, q-pl, s) \subseteq M(A^{-l}(\B), \calN(p, q, s)) \subseteq \calN(p, q, s) \cap A^{-\left(l-\frac{q}{p}\right)}.
$$
\end{enumerate}
\end{thm}

\begin{proof}
Take any $g \in M(A^{-l}(\B), \calN(p, q, s))$. Then we have $gL_w \in \calN(p, q, s), \forall w \in \B$. We get
$$
|g(z)L_w(z)| \le \frac{\|gL_w\|}{(1-|z|^2)^{\frac{q}{p}}} \le \frac{\|M_g\||L_w|_l}{(1-|z|^2)^{\frac{q}{p}}}.
$$
Let $z=w$. We have
\begin{equation} \label{Sec80330099}
|g(w)| \lesssim (1-|w|^2)^{l-\frac{q}{p}}.
\end{equation}

\medskip

$(1)$ It is clear that by \eqref{Sec80330099}, $\lim\limits_{|w| \to 1} |g(w)|=0$, which implies $g(z)$ equals to $0$ everywhere by maximum modulus principle.

\medskip

$(2)$ The inclusion $M(A^{-l}(\B), \calN(p, q, s)) \subseteq H^\infty$ clearly follows from \eqref{Sec80330099}. For the converse direction, let $\varphi \in H^\infty$ and $f \in A^{-\frac{q}{p}}(\B)$, and we have
$$
\|f \varphi\| \le \|\varphi\|_{H^\infty} \|f\| \le \|\varphi\|_{H^\infty} |f|_{\frac{q}{p}}<\infty,
$$
where the second inequality follows from Proposition \ref{boundaryineq}.

\medskip

$(3)$ The second inclusion follows clearly from \eqref{Sec80330099}, and hence we omit the proof here. For the first inclusion,
we take and fix a $g \in \calN(p, q-pl, s)$, and then for any $f \in A^{-l}(\B)$, we have
\begin{eqnarray*}
&&\|gf\|_{\calN(p, q, s)}^p=\sup_{a \in \B} \int_\B |g(z)f(z)|^p(1-|z|^2)^q(1-|\Phi_a(z)|^2)^{ns}d\lambda(z)\\
&&\lesssim \sup_{z \in \B}  \int_\B |g(z)|^p(1-|z|^2)^q  \cdot \frac{|f|_l^p}{(1-|z|^2)^{pl}} \cdot (1-|\Phi_a(z)|^2)^{ns}d\lambda(z)\\
&&\le \|g\|_{\calN(p, q-pl, s)}^p |f|_l^p,
\end{eqnarray*}
which implies the desired result.
\end{proof}

The corresponding results for $\calN^0(p, q, s)$ again are immediate by an easy modification, and hence we only state the result as follows.

\begin{thm}
Let $p \ge 1, q>0$, $s>\max\left\{0, 1-\frac{q}{n} \right\}$ and $\beta, l \ge 0$. Then we have the following results.
\begin{enumerate}
\item[(1)] If $\beta>0$, then $M(\Lambda_\beta, \calN^0(p, q, s))=\calN^0(p, q, s)$.
\item[(2)] If $\beta=0$, then $M(\Lambda_\beta, \calN^0(p, q, s)) \subseteq \calN^0(p, q, s)$. Conversely, if $\varphi \in H(\B)$ satisfying \eqref{Sec803101} and
$$
\lim_{|a| \to 1} \int_\B |\varphi(z)|^p\left(\log \frac{2}{1-|z|^2} \right)^p (1-|z|^2)^q(1-|\Phi_a(z)|^2)^{ns}d\lambda(z)=0,
$$
then $\varphi \in M(\Lambda_\beta, \calN^0(p, q, s))$.
\item[(3)] If $l>\frac{q}{p}$, then $M(A^{-l}(\B), \calN^0(p, q, s))$ only contains the constant function $\varphi(z) \equiv 0$.
\item[(4)] If $l=\frac{q}{p}$, then $M(A^{-l}(\B), \calN^0(p, q, s))=H^\infty$.
\item[(5)] If $0<l<\frac{q}{p}$, then
$$
\calN^0(p, q-pl, s) \subseteq M(A^{-l}(\B), \calN^0(p, q, s)) \subseteq \calN^0(p, q, s) \cap A^{-\left(l-\frac{q}{p}\right)}(\B).
$$
\end{enumerate}
\end{thm}

As a conclusion of this subsection, we have the following result with consider the cases $\alpha=q-n-1$ and $\beta=-l=-\frac{q}{p}$.

\begin{thm}
Let $p \ge 1, q>0, s>\max\left\{0, 1-\frac{q}{n} \right\}$ and $\varphi \in H(\B)$. The following statements are equivalent:
\begin{enumerate}
\item[(a)] $\varphi \in H^\infty$;
\item[(b)] $\varphi \calN(p, q, s) \subseteq \calN(p, q, s)$;
\item[(c)] $\varphi \calN^0(p, q, s) \subseteq \calN^0(p, q, s)$;
\item[(d)] $\varphi A^p_{q-n-1} \subseteq \calN(p, q, s)$;
\item[(e)] $\varphi A^p_{q-n-1} \subseteq \calN^0(p, q, s)$;
\item[(f)] $\varphi A^{-\frac{q}{p}}(\B) \subseteq \calN(p, q, s)$;
\item[(g)] $\varphi A^{-\frac{q}{p}}(\B) \subseteq \calN^0(p, q, s)$.
\end{enumerate}
\end{thm}

\bigskip

\noindent {\bf Acknowledgments.}  The corresponding author was supported by the Macao Science and Technology Development
Fund (No.083/2014/A2) and NSF of China (No. 11471143 and No.11720101003).

\end{document}